\documentclass[11pt]{amsart}
\usepackage{mathtools,amsfonts,amssymb,amsthm,amscd,amsmath,amsgen,upref,enumerate,nicefrac,esint}

 \usepackage[utf8]{inputenc} 
\usepackage[colorinlistoftodos,prependcaption,color=yellow,textsize=tiny]{todonotes}

\usepackage[margin=1in]{geometry}

\theoremstyle{plain}
\newtheorem{cor}{Corollary}
\newtheorem{lem}[cor]{Lemma}
\newtheorem{prop}[cor]{Proposition}

\newtheorem{thm}[cor]{Theorem}
\newtheorem*{thm*}{Theorem}
\theoremstyle{definition}
\newtheorem{definition}[cor]{Definition}
\newtheorem{remark}[cor]{Remark}
\newtheorem{example}[cor]{Example}
\newtheorem{assumption}[cor]{Assumption}

\numberwithin{cor}{section}
\numberwithin{equation}{section}

\DeclareMathOperator{\C}{C}

\newcommand{\E}{\mathbb{E}}

\renewcommand{\d}{\delta}

\renewcommand{\and}{\quad\textrm{ and }\quad}

\renewcommand{\P}{\mathbb{P}}
\renewcommand{\a}{\alpha}
\renewcommand{\b}{\beta}

\renewcommand{\o}{\omega}
\renewcommand{\O}{\Omega}

\newcommand{\sgn}{\text{sgn}}

\newcommand{\Ent}{\text{Ent}}

\newcommand{\F}{\mathcal F}

\newcommand{\R}{\mathbb{R}}
\newcommand{\TT}{\mathbb{T}}

\newcommand{\N}{\mathbb{N}}
\newcommand{\Z}{\mathbb{Z}}

\newcommand{\norm}[1]{\left\| #1 \right\|}

\newcommand{\D}{\Delta}

\newcommand{\ve}{\varepsilon}

\newcommand{\abs}[1]{\left|#1\right|}
\providecommand{\ud}[1]{\, \mathrm{d} #1}
\providecommand{\dx}{\ud{x}}
\providecommand{\dy}{\ud{y}}
\providecommand{\dxi}{\ud \xi}
\providecommand{\deta}{\ud{\eta}}

\providecommand{\dxp}{\ud{x'}}
\providecommand{\dxip}{\ud{\xi'}}

\providecommand{\ds}{\ud{s}}
\providecommand{\dt}{\ud{t}}

\providecommand{\dd}{\ud}

\def\XXint#1#2#3{{\setbox0=\hbox{$#1{#2#3}{\int}$ }
\vcenter{\hbox{$#2#3$ }}\kern-.6\wd0}}

\usepackage{graphicx}
\usepackage{color}

\begin{document}

\title[Well-posedness of the Dean--Kawasaki equation]{Well-posedness of the Dean--Kawasaki and the nonlinear Dawson--Watanabe equation with correlated noise}

\author{Benjamin Fehrman}
\address{Mathematical Institute, University of Oxford, OX2 6GG Oxford, United Kingdom}
\email{Benjamin.Fehrman@maths.ox.ac.uk}

\author{Benjamin Gess}
\address{Max Planck Institute for Mathematics in the Sciences, 04103 Leipzig, Germany \newline \indent Fakult\"at f\"ur Mathematik, Universit\"at Bielefeld, 33615 Bielefeld, Germany}
\email{Benjamin.Gess@mis.mpg.de}

\date{\today}

\begin{abstract} In this paper we prove the well-posedness of the generalized Dean--Kawasaki equation driven by noise that is white in time and colored in space. The results treat diffusion coefficients that are only locally $\nicefrac{1}{2}$-H\"older continuous, including the square root. This solves several open problems, including the well-posedness of the Dean--Kawasaki equation and the nonlinear Dawson--Watanabe equation with correlated noise.  \end{abstract}


\maketitle

\section{Introduction}
In this paper, we introduce a general framework to establish the well-posedness of nonnegative solutions to stochastic PDE of the type\footnote{In the introduction, we formally write $\partial_t\rho$ and $\dot{\xi}^F$ to denote the time-derivative of the stochastic processes $\rho$ and $\xi^F$.  In the remainder of the paper, we will use the probabilistic notation $\dd\rho$ and $\dd\xi^F$.}
\begin{equation}
\partial_t\rho=\Delta\Phi(\rho)-\nabla\cdot\left(\sigma(\rho)\circ\dot{\xi}^{F}+\nu(\rho)\right)+\phi(\rho)\dot{\xi}^{G}+\lambda(\rho)\;\;\textrm{in}\;\;\TT^{d}\times(0,T),\label{1_000}
\end{equation}
for Stratonovich-noise $\circ\;\dot{\xi}^{F}$ and It\^o-noise $\dot{\xi}^{G}$ white in time and sufficiently regular in space.  The assumptions on the nonlinearities $\Phi,\sigma,\nu,\phi,\lambda$ are given in detail below and apply to the full range of fast diffusion and porous medium equations, that is, $\Phi(\rho)=\rho^{m}$, for every $m\in(0,\infty)$, to irregular $\sigma$ including the square root, and to all locally $\nicefrac{1}{2}$-H\"older continuous $\phi$.

The results solve several previously open problems, including most notably the well-posedness of conservative stochastic PDE with correlated noise
\begin{equation}
\partial_t\rho=\Delta\Phi(\rho)-\nabla\cdot\left(\sigma(\rho)\circ\dot{\xi}^{F}+\nu(\rho)\right)\;\;\textrm{in}\;\;\TT^{d}\times(0,T),\label{1_0}
\end{equation}
such as the Dean--Kawasaki equation with correlated noise
\begin{equation}
\partial_t\rho=\Delta\rho-\nabla\cdot(\rho^{\frac{1}{2}}\circ\dot{\xi}^{F})\;\;\textrm{in}\;\;\TT^{d}\times(0,T).\label{1_0101}
\end{equation}
Equations like \eqref{1_0} arise as fluctuating continuum models for interacting particle systems, see, for instance, Giacomin, Lebowitz, and Presutti \cite[Section~4]{GiaLebPre1999}, and have been used to describe the hydrodynamic large deviations of simple exclusion and zero range particle processes, see, for example, Dirr, Stamatakis, and Zimmer \cite{DirStaZim2016}, Quastel, Rezakhanlou, and Varadhan \cite{QuaRezVar1999}, Benois, Kipnis, and Landim \cite{BenKipLan1995}, and the authors and Dirr \cite{FehGesDir2020,FehGes2020}. Yet, despite their physical relevance, and despite that analogous versions of \eqref{1_0101} were introduced nearly thirty years ago by Dean \cite{Dea1996} and Kawasaki \cite{Kaw1994}, both \eqref{1_0} and \eqref{1_0101} have until now lacked a precise mathematical meaning, see, for example, Donev, Fai, and Vanden-Eijnden \cite[Page~6]{DonFaiVan2014} and Konarovskyi and von Renesse \cite{KonvRe2019}.

The primary contribution of this work is the development of a robust solution theory for equations like \eqref{1_0} and \eqref{1_0101} with degenerate diffusions $\Phi$, irregular noise coefficients $\sigma$ including the square root, and locally Lipschitz continuous $\nu$.  In addition, the inclusion of nontrivial zeroth order terms $\phi$ and $\lambda$, while comparatively simpler, still provides an answer to several problems, such as the well-posedness of the nonlinear Dawson--Watanabe equation with correlated noise, see \eqref{intro_DW}.



The well-posedness of \eqref{1_000} is challenging, due to the nonlinearity and possible degeneracy of the diffusion, the possible irregularity caused by the fluctuations entering in the form of a stochastic conservation law, and the  lack of Lipschitz continuity of the noise coefficients.  Furthermore, the singularity appearing in the Stratonovich-to-It\^o-correction makes it necessary to develop a renormalized solution theory and substantially complicates the construction of solutions due to the lack of standard a priori estimates (see Section~\ref{sec_elements} below).

We next summarize the main results of this work, first in the conservative case \eqref{1_0}, second in the non-conservative case. In the first case we prove that the equation satisfies an almost sure $L^1$-contraction, and in both cases we develop an $L^p$-based theory for every $p\in[1,\infty)$ and a theory for initial data with finite entropy.  The detailed statement of the results can be found in the main text in the indicated Theorems and Corollaries.

\begin{thm}[Theorems~\ref{thm_unique}, \ref{thm_exist}, Corollaries~\ref{cor_exist}, \ref{cor_exist_1}]
Let $\dot{\xi}^{F}$, $\Phi$, $\sigma$, and $\nu$ satisfy Assumptions~\ref{assumption_noise}, \ref{assume_1}, and \ref{assume} for some $p\in[2,\infty)$, and let $\rho_{0}\in L^p(\TT^{d})$ be nonnegative. Then, there exists a unique stochastic kinetic solution of \eqref{1_0} in the sense of Definition~\ref{def_sol}, and the solution satisfies the estimates of Proposition~\ref{prop_approx_est}. 

In addition, if $\rho^{1}$ and $\rho^{2}$ are two stochastic kinetic solutions with initial data $\rho_{0}^{1}$ and $\rho_{0}^{2}$ then, a.s.\ for every $t\in[0,T]$, 
\begin{equation}
\norm{\rho^{1}(\cdot,t)-\rho^{2}(\cdot,t)}_{L^{1}(\TT^{d})}\leq\norm{\rho_{0}^{1}-\rho_{0}^{2}}_{L^{1}(\TT^{d})}.\label{eq:contraction}
\end{equation}
The same results hold if $\rho_{0}$ is nonnegative with finite entropy in the sense of Definition~\ref{def_entropy} and Assumption~\ref{assumption_10} is satisfied, or if $\rho_{0}\in L^{1}(\TT^{d})$ and $\sigma^{2}$ and $\nu$ grow at most linearly at infinity. 
\end{thm}

The pathwise contraction property \eqref{eq:contraction} is a key observation that implies the pathwise continuity of solutions with respect to the initial condition, a property rarely known for solutions to stochastic PDE. In fact, this is a key obstacle in the development of a random dynamical systems (RDS) approach to stochastic PDEs, see Flandoli \cite{F95} for a detailed discussion. The results here, in particular \eqref{eq:contraction}, constitute the basis for the construction of RDS for conservative stochastic PDEs by the authors and Gvalani \cite{FehGesGva2022}.

In addition to the applications to the non-equilibrium fluctuations of conservative systems, we detail further applications of \eqref{1_000} in the non-conservative case $(\sigma,\lambda\neq 0)$ to branching interacting diffusions and the nonlinear Dawson--Watanabe equation with correlated noise
\begin{equation}\label{intro_DW}\partial_t\rho=\Delta\Phi(\rho)+\rho^{\frac{1}{2}}\dot{\xi}^{G}\;\;\textrm{in}\;\;\TT^{d}\times(0,T),\end{equation}
to interacting particle systems with common noise, and to stochastic geometric PDE in Section~\ref{sec_app} below.  The main result in this case is as follows.

\begin{thm}[Theorems~\ref{thm_gen_unique}, \ref{thm_gen_exist}, Corollaries~\ref{cor_exist_3}, \ref{cor_exist_2}]
Let $\dot{\xi}^{F}$, $\dot{\xi}^{G}$, $\Phi$, $\sigma$, $\nu$, $\phi$, and $\lambda$ satisfy Assumptions~\ref{assumption_noise_1}, \ref{assume_5}, and \ref{assume_4} for some $p\in[2,\infty)$, and let $\rho_{0}\in L^p(\TT^{d})$ be nonnegative. Then there exists a unique stochastic kinetic solution of \eqref{1_000} in the sense of Definition~\ref{def_gen_sol}, and the solution satisfies the estimates of Proposition~\ref{prop_gen_est}. 

In addition, if $\rho^{1}$ and $\rho^{2}$ are two stochastic kinetic solutions with initial data $\rho_{0}^{1}$ and $\rho_{0}^{2}$ then there exists $c\in(0,\infty)$ such that, for every $t\in[0,T]$, 
\[
\E\left[\norm{\rho^{1}(\cdot,t)-\rho^{2}(\cdot,t)}_{L^{1}(\TT^{d}))}\right]\leq c\exp(ct)\norm{\rho_{0}^{1}-\rho_{0}^{2}}_{L^{1}(\TT^{d})}.
\]
Furthermore, there exists $c\in(0,\infty)$ such that 
\[\E\left[\sup_{t\in[0,T]}\norm{\rho^{1}(\cdot,t)-\rho^{2}(\cdot,t)}_{L^{1}(\TT^{d})}\right] \leq c\exp(cT)\left(\norm{\rho_{0}^{1}-\rho_{0}^{2}}_{L^{1}(\TT^{d})}^{\frac{1}{2}}+\norm{\rho_{0}^{1}-\rho_{0}^{2}}_{L^{1}(\TT^{d})}\right).\]
The same results hold if $\rho_{0}$ is nonnegative with finite entropy in the sense of Definition~\ref{def_entropy} and Assumption~\ref{assumption_11} is satisfied, or if $\rho_{0}\in L^{1}(\TT^{d})$ and $\sigma^{2}$, $\nu$, and $\phi^2$ grow at most linearly at infinity. 
\end{thm}

\subsection{Elements of the proof}\label{sec_elements}  The Stratonovich equation \eqref{1_0} with $\nu=0$ is formally equivalent to the It\^o equation
\begin{equation}\label{1_20}\partial_t\rho = \Delta\Phi(\rho)-\nabla\cdot(\sigma(\rho)\dot{\xi}^F)+(\nicefrac{1}{2})\nabla\cdot(F_1[\sigma'(\rho)]^2\nabla\rho+\sigma(\rho)\sigma'(\rho)F_2),\end{equation}
for coefficients $F_1=\sum_{k=1}^\infty f_k^2$ and $F_2=\frac{1}{2}\nabla F_1=\sum_{k=1}^\infty f_k\nabla f_k$ which, in the Dean--Kawasaki case, takes the form
\begin{equation}\label{1_21}\partial_t\rho = \Delta\rho-\nabla\cdot(\rho^\frac{1}{2}\dot{\xi}^F)+(\nicefrac{1}{8})\nabla\cdot\left(F_1\rho^{-1}\nabla\rho\right)+(\nicefrac{1}{4})\nabla\cdot F_2.\end{equation}
This illustrates two fundamental difficulties in treating \eqref{1_000}--\eqref{1_0101}:  not only is it necessary to treat nonlinearities that are only $\nicefrac{1}{2}$-H\"older continuous---which correspond to the most relevant applications of \eqref{1_000}--\eqref{1_0101} and which has remained a fundamental open problem despite much effort (see Section~\ref{sec_lit} below)---it is also necessary to treat the in general singular term 
\begin{equation}\label{eq:singular_term}
   \nabla\cdot\left(F_1\rho^{-1}\nabla\rho\right)=\nabla\cdot\left(F_1\nabla \mathrm{log}(\rho)\right)
\end{equation}
arising from the It\^o-correction. In fact, it is not even clear how to define a concept of weak solution to \eqref{1_21} since $\mathrm{log}(\rho)$ is not known to be locally integrable. 

In the theory of renormalized solutions introduced by DiPerna and Lions \cite{DiPLio1989} and extended to conservation laws by B\'enilan, Carrillo, and Wittbold \cite{BenCarWit2000}, the possible lack of local integrability of nonlinear terms is resolved by the notion of renormalized solutions, which, roughly speaking are required to satisfy the PDE only after cutting out large values, thereby enforcing local integrability. A key and fundamental idea in the treatment of equations like \eqref{1_20} developed in the present work is to introduce a concept of renormalized solutions, cutting out both large and small values of the solution, in order to avoid the possible lack of local integrability caused by the singular term \eqref{eq:singular_term} at \emph{small} values. This concept of renormalized solutions is here derived from stochastic kinetic solutions, see Definition~\ref{def_sol}, which is based on the kinetic formulation of scalar conservation laws introduced by Lions, Perthame, and Tadmor \cite{LioPerTad1994}, Perthame \cite{Per2002}, and Chen and Perthame \cite{ChePer2003} (see Section~\ref{sec_kinetic} below).

The renormalization away from large and small values causes substantial difficulties both in the proof of the uniqueness and the existence of solutions. We will comment on these next.

In the proof of uniqueness, the localization away from infinity and zero causes the necessity to introduce cutoff functions that create singularities near zero. These singularities have to be carefully controlled and compensated by properties of the entropy dissipation measure and nonlinearities.  We develop a new and precise characterization of the behavior of the entropy dissipation measure, and therefore the entropy inequality, on the zero set of the solution in Proposition~\ref{prop_measure} below. It is this characterization that we use to treat the singularities appearing due to the compact support of the test functions.

The singularity appearing due to the It\^o-correction \eqref{eq:singular_term} also significantly complicates the proof of a priori estimates on the time regularity of solutions and, thereby, the proof of the  existence of solutions. Indeed, even if a solution $\rho$ to \eqref{1_21} is spatially regular, due to the possible divergence of $\log(\rho)$ at $\rho \approx 0$, this does not imply an estimate on the time regularity of $\rho$.  To overcome this fact, also in the proof of existence we rely on cutting out small values of the solutions. Precisely, in Proposition~\ref{prop_approx_time} we prove stable estimates in time for nonlinear functions $\Psi_\d(\rho)$ of the solution (see Definition~\ref{def_cutoff}) that localize the solutions away from zero.  This implies the tightness in $L^1_tL^1_x$ for these cut-off solutions  $\Psi_\d(\rho)$. We then introduce a corresponding new metric on $L^1_tL^1_x$ (see Definition~\ref{def_metric} below) whose topology coincides with the usual strong norm topology, and show that the tightness the $\Psi_\d(\rho)$ implies the tightness in law of the approximating solutions $\rho$ themselves (see Proposition~\ref{prop_tight}).

Finally, as a consequence of the lack of a stable $W^{\beta,1}_tH^{-s}_x$-estimate and the kinetic formulation, it is not clear that the laws of the approximating solutions are tight in a space of continuous in time, $H^{-s}_x$-valued functions.  It is for this reason that we prove directly the tightness of the martingale terms of equation \eqref{1_20} in Proposition~\ref{prop_mart_tight}, which relies on the fact that the noise is sufficiently smooth in space.  We prove the existence of a probabilistically strong solution to \eqref{1_0} in Theorem~\ref{thm_exist}, and we extend these results to equation \eqref{1_000} in Section~\ref{sec_gen}.

\subsection{Applications}

\subsubsection*{Non-equilibrium fluctuations for symmetric systems\label{sec_app}}

The framework of fluctuating hydrodynamics postulates conservative, singular stochastic PDEs of the type \eqref{1_0} as mesoscopic descriptions of microscopic dynamics and their fluctuations, far from equilibrium. As a particular example, Ferrari, Presutti, and Vares \cite{FerPreVar1987} studied the fluctuating hydrodynamics of the zero range process about its hydrodynamic limit, and these were informally shown (see, for example, Dirr, Stamatakis, and Zimmer \cite{DirStaZim2016}) to satisfy the equation 
\begin{equation}
\partial_t\rho=\Delta\Phi(\rho)-\nabla\cdot\left(\Phi^{\frac{1}{2}}(\rho)\dot{\xi}\right)\label{1_0000}
\end{equation}
for $\dot{\xi}$ a $d$-dimensional space-time white noise, and for $\Phi$ the mean-local jump rate (see, for example, Kipnis and Landim \cite[Chapter~7]{KipLan1999}). Notably, \eqref{1_0000} is an informal equation and giving rigorous meaning to it would require renormalization. This would include, in particular, a choice of renormalization constants, and, thereby, a choice of the interpretation of the stochastic integral (\`a la Stratonovich vs.\ It\^o). The case $\Phi(\xi)=\xi$ corresponds to the Dean--Kawasaki equation (see Dean~\cite{Dea1996}, Kawsaki~\cite{Kaw1994}, Marconi and Tarazona \cite{MarTar1999}, and te Vrugt, L{\"o}wen and Wittkowski \cite{teVLowWit2020}). A rigorous justification of this ansatz, and of the choice of correlated Stratonovich noise, is given in \cite{FehGesDir2020,FehGes2020} through the analysis of the corresponding large deviations rate function of the symmetric simple exclusion and zero range particle processes, see, as well, \cite{BenKipLan1995,QuaRezVar1999}.

Notably, because of the irregularity of space-time white noise, the equation \eqref{1_0000} is supercritical in the language of regularity structures \cite{Hai2014}. However, it can be argued that the microscopic system comes with a typical de-correlation length for the noise, like the grid-size, which leads to \eqref{1_0000} with spatially correlated noise (see Giacomin, Lebowitz, and Presutti \cite[Section 4]{GiaLebPre1999}). This viewpoint has also been taken in \cite{FehGes2020}, where it has been shown that the small noise large deviations of 
\begin{equation}
\partial_t\rho^\ve=\Delta\Phi(\rho^\ve)-\ve \nabla\cdot\left(\Phi^{\frac{1}{2}}(\rho)\circ \dot{\xi}^{F,\ve}\right)\label{1_0000-1}
\end{equation}
with $\dot{\xi}^{F,\ve}$ being a spatially correlated noise converging, as $\ve\rightarrow 0$, to space-time white noise correctly reproduce the large deviations of the zero range process.   For the reasons mentioned above, the well-posedness of \eqref{1_0000-1} had remained a long-standing open problem in the literature, which is solved in the present work.
\begin{example}
 Let $\dot{\xi}^F$, $\Phi$, and $\sigma=\Phi^\frac{1}{2}$ satisfy Assumptions~\ref{assumption_noise}, \ref{assume_1}, and \ref{assume} for $p=2$, and let $\rho_0\in L^2(\TT^d)$ be nonnegative. Then there is a unique kinetic solution to \eqref{1_0000-1} and each two solutions satisfy \eqref{eq:contraction}. 
In particular, this includes the case of the Dean--Kawasaki equation with correlated Stratonovich noise, that is, $\Phi(\rho)=\rho$, and porous medium equation $\Phi(\rho)=\rho^{m}$, for all $m\in (1,\infty)$.
\end{example}

\subsubsection*{Non-equilibrium fluctuations for asymmetric systems}

Along the motivation of the previous example, a continuum, ``mesoscopic'' description of asymmetric systems is, informally, given by (see  \cite[Section 4]{GiaLebPre1999} and Mariani \cite{Mar2010})
\[
\partial_{t}\rho^\ve=\frac{\ve}{2}\Delta\Phi(\rho^\ve)+\nabla\cdot\nu(\rho^\ve)+\ve^{\frac{1}{2}}\nabla\cdot(\sqrt{a^{2}(\rho^\ve)}\dot{\xi}^\ve),
\]
where the bulk diffusion $\Phi'(\rho^\ve)$ and fluctuation intensity $\sqrt{a^2(\rho^\ve)}$ satisfy a fluctuation-dissipation relation, $\nu(\rho)$ corresponds to the asymmetric, nonlinear transport and $\dot{\xi}^\ve$ has spatial correlation length $\ve$. The asymmetric nature of the system in the Eulerian scaling is expressed by the coefficient $\ve\in(0,\infty)$ which, in the hydrodynamic limit, causes both the diffusion and fluctuations to vanish. 

A concrete example is given by the asymmetric zero range process. In this case, we have that $\Phi$ is the mean-local jump rate, $\sqrt{a^{2}(\rho^\ve)}=\sqrt{\Phi(\rho^\ve)}$, and $\nu(\rho^\ve)=\Phi(\rho^\ve)$, see, for example, Gon\c{c}alves \cite{Gon2014}.  That is,
\[\partial_{t}\rho^\ve=\frac{\ve}{2}\Delta\Phi(\rho^\ve)+\nabla\cdot\Phi(\rho^\ve)+\ve^{\frac{1}{2}}\nabla\cdot(\sqrt{\Phi(\rho^\ve)} \dot{\xi}^\ve),\]
where we emphasize that, as in the case of \eqref{1_0000-1}, even in the case of correlated noise this equation had until now lacked a rigorous mathematical meaning.  The well-posedness of this class of stochastic PDE,  with the choice of spatially correlated Stratonovich noise, was an open problem in the literature that is solved by the results of the present work. 

\begin{example}
 Let $\dot{\xi}^{F,\ve}$, $\Phi$, and $\sigma$ satisfy Assumptions~\ref{assumption_noise}, \ref{assume_1}, and \ref{assume} for $p=2$, and let $\rho_0\in L^2(\TT^d)$ be nonnegative. Then there is a unique kinetic solution to
\[\partial_{t}\rho^\ve=\frac{\ve}{2}\Delta\Phi(\rho^\ve)+\nabla\cdot\Phi(\rho^\ve)+\ve^{\frac{1}{2}}\nabla\cdot(\sqrt{\Phi(\rho^\ve)}\circ \dot{\xi}^{F,\ve}),\]
and each two solutions satisfy \eqref{eq:contraction}. 
\end{example}

The asymmetric simple exclusion process corresponds to $\Phi(\rho)=\rho$, $\nu(\rho)=a^{2}(\rho)=\rho(1-\rho)$. In this case, the exclusion rule prevents concentration of mass, which allows a much simpler treatment of the stochastic PDE, see \cite{Mar2010} and \cite{FehGesDir2020}.  However, prior to this work, even for this case it remained necessary to introduce an approximation of the square root $\sqrt{\rho(1-\rho)}$ in order to obtain the well-posedness of the equation.

\subsubsection*{Nonlinear Dawson--Watanabe equation}

The scaling limits of independent branching Brownian motions are known to converge to solutions of the Dawson--Watanabe stochastic PDE. In the case of mean-field interacting, branching processes, the analogous scaling limits are described in terms of (non-local) quasilinear stochastic PDE, see M\'el\'eard and Roelly \cite{MelRoe1992}. The localized interaction limit then, informally, leads to solutions to the nonlinear Dawson--Watanabe equation
\begin{equation}
\partial_{t}\rho=\Delta\Phi(\rho)dt+\phi(\rho)\dot{\xi}^{G},\label{eq:nl-dw}
\end{equation}
with $\Phi(\rho)=\rho^{2},\phi(\rho)=\rho^\frac{1}{2}$, see Dareiotis, Gerencs\'er and Gess \cite[Section 1.1]{DarGerGes2019}. Based on the work of Oelschl\"ager \cite{Oel1990}, one may expect that moderate interaction regimes could produce other cases of nonlinearities $\Phi$. 

The well-posedness of \eqref{eq:nl-dw} has been considered by Dareiotis, Gerencs\'er and Gess in \cite{DarGerGes2019,DarGerGes2020}. First, in \cite{DarGerGes2019}, for spatially correlated noise $\dot{\xi}^G$ and assuming that $\phi$ is $C^{\nicefrac{1}{2}+\ve}$, for some $\ve\in(0,\nicefrac{1}{2}]$, the well-posedness of entropy solutions to \eqref{eq:nl-dw} has been shown. This left the most relevant case $\phi(\rho)=\rho^\frac{1}{2}$ as an open problem, which is solved by the results of the present work.
\begin{example}
Let $\rho_{0}\in L^{2}(\TT^d)$ be nonnegative and assume that the noise satisfies Assumption \ref{assumption_noise_1}. Then there is a unique solution to \eqref{eq:nl-dw}, and the solutions satisfy, for some $c\in(0,\infty)$,
\[\E\sup_{t\in[0,T]}\|\rho^{1}(t)-\rho^{2}(t)\|_{L^{1}(\TT^{d})} \le c\exp(cT)(\|\rho_{0}^{1}-\rho_{0}^{2}\|_{L^{1}(\TT^{d})}^{\nicefrac{1}{2}}+\|\rho_{0}^{1}-\rho_{0}^{2}\|_{L^{1}(\TT^{d})}).\]
 \end{example}

\subsubsection*{Interacting particle systems with common noise}

In Kurtz, Xiong \cite{KurXio1999} and Coghi, Gess \cite{CogGes2019} it has been shown that the conditional empirical density measure $\mu^{N}:=\mathcal{L}(\frac{1}{N}\sum_{j=1}^{N}\delta_{X_{t}^{j}}\mid B)$ of a mean field interacting particle system 
\begin{align*}
\dot{X}_{t}^{i} & =\frac{1}{N}\sum_{j=1}^{N}\left(V_{1}(X_{t}^{i}-X_{t}^{j})+V_{2}(X_{t}^{i}-X_{t}^{j})\circ \dot{B}_t+V_{3}(X_{t}^{i}-X_{t}^{j})\dot{\beta}_{t}^{i}\right)
\end{align*}
with sufficiently regular interaction kernels $V_{i}$ and with $B$ and $\beta^{i}$ independent Brownian motions, converges in the mean field limit $N\to\infty$ to the solution of a nonlinear, nonlocal, stochastic Fokker Planck equation 
\begin{equation}
\begin{array}{l}
\partial_{t}\rho=\Delta((V_{3}\ast\rho)\rho)-\nabla\cdot((V_{1}\ast\rho)\rho)-\nabla\cdot((V_{2}\ast\rho)\rho\circ \dot{B}_{t}).\end{array}
\label{target equation}
\end{equation}
See also Kotelenez \cite{Kot2008} for a motivation of the same class of SPDE arising in statistical mechanics. 

For simplicity we now restrict to the case $d=1$. One may next consider the localized interaction limit, that is, when $V_{i}$ are replaced by Dirac sequences $V_{i,\ve}$ with corresponding solutions $\mu^{\ve}$ to \eqref{target equation}. Then, informally, in the limit $\ve\to0$, we obtain that $\mu^{\ve}\to\rho\,dx$ with $\rho$ being the solution to the nonlinear, stochastic Fokker Planck equation 
\begin{equation}
\begin{array}{l}
\partial_{t}\rho=\Delta\Phi(\rho)-\nabla\cdot\nu(\rho)-\nabla\cdot(\sigma(\rho)\circ \dot{B}_{t}),\end{array}\label{eq:MF_common_noise}
\end{equation}
with $\Phi(\rho)=\nu(\rho)=\sigma(\rho)=\rho^{2}$. For a notable relation to the theory of mean field games with common noise we refer to Lasry and Lions \cite{LasLio2006,LasLio20062,LasLio2007}. 
\begin{example}
Let $\rho_{0}\in L^{3}(\TT^1)$ be nonnegative. Then there is a unique solution to \eqref{eq:MF_common_noise} and each two solutions satisfy \eqref{eq:contraction}. 
\end{example}

\subsubsection*{Stochastic geometric PDE}

In \cite{KawOht1982}, Kawasaki and Ohta (see also Katsoulakis and Kho \cite{KatKho2001}) have derived the following informal stochastic PDE, describing the graph of the fluctuating interface, in the sharp interface limit of the fluctuating Ginzburg-Landau equation
\[
\partial_{t} u =\left(\sqrt{1+\abs{\nabla u}^2}\right)\nabla\cdot\left(\frac{\nabla u}{\sqrt{1+\abs{\nabla u}^{2}}}\right)+(1+\abs{\nabla u}^2)^{\frac{1}{4}}\dot{\xi},
\]
for $\dot{\xi}$ space-time white noise.  In one spatial dimension and passing to the first derivative $\rho=\partial_x u$, this corresponds to
\begin{equation}
\partial_{t}\rho=\Delta\Phi(\rho)+\nabla\cdot(\sigma(\rho)\dot{\xi})\label{eq:KO}
\end{equation}
with $\Phi(\rho)=\arctan(\rho)$ and $\sigma(\rho)=(1+\abs{\rho}^{2})^{\frac{1}{4}}$, see Es-Sarhir and von Renesse \cite{EsRen2012}. For further background on the fluctuating mean-curvature equation we refer to Souganidis and Yip \cite{SouYip2004} and Dirr, Luckhaus, and Novaga \cite{DirLucNov2001}. 

The well-posedness of the Kawasaki--Ohta equation \eqref{eq:KO} is a challenging problem due to the degeneracy of the diffusion $\Phi'(\rho)=\frac{1}{1+\rho^{2}}$ at large values of $\rho$, and due to the stochastic conservation law structure of the noise. Therefore, in \cite{EsRen2012} the analysis had to be restricted to spatially constant noise, which was first resolved in \cite{DarGes2020}. The general theory developed in the present paper contains this example with spatially correlated noise as a special case.
\begin{example}
Let $\rho_{0}\in L^{3}(\TT^1)$ be nonegative and assume that $\dot{\xi}^F$ satisfies Assumption~\ref{assumption_noise}. Then there is a unique solution to 
\[
\partial_{t}\rho=\nabla\cdot\left(\frac{\nabla \rho}{1+\rho^{2}}\right)+\nabla\cdot ((1+\rho^{2})^{\frac{1}{4}}\circ \dot{\xi}^{F}),
\]
and the solutions satisfy \eqref{eq:contraction}. 
\end{example}

\subsubsection*{Extensions}

Fluctuating branching interacting diffusion systems lead to a combination of the effects discussed above and, thereby, to stochastic PDE of the form \eqref{1_000}, combining both conservative and non-conservative fluctuations and transport. 

By a slight adaptation, see \cite{FehGesDir2020}, the methods of this work can be extended to stochastic PDE with diffusion coefficients having multiple points of irregularity, as it is typical for Fleming--Viot type stochastic PDEs 
\[
\partial_{t}\rho=\Delta \rho+\ve\sqrt{\rho(1-\rho)}\dot{\xi}^G,
\]
the stochastic Fisher-Kolmogorov-Petrovsky-Piscounov (Fisher-KPP) equation
\[
\partial_{t}\rho=\Delta \rho+\gamma\rho(1-\rho)+\ve\sqrt{\rho(1-\rho)}\dot{\xi}^G,
\]
assuming that $\dot{\xi}^{G}$ satisfies Assumption \ref{assumption_noise_1}, and the asymmetric simple exclusion process
\[
\partial_{t}\rho=\frac{\ve}{2}\Delta \rho+\nabla\cdot(\rho(1-\rho))+\ve^{\frac{1}{2}}\nabla\cdot(\sqrt{\rho(1-\rho)}\circ \dot{\xi}^F),
\]
assuming that $\dot{\xi}^F$ satisfies Assumption \ref{assumption_noise}.

\subsection{Comments on the results and assumptions}

\subsubsection*{The initial datum}
The above results are stated for deterministic initial data for simplicity.  In Theorems~\ref{thm_unique} and \ref{thm_gen_unique} we treat random initial data in $L^p(\O;L^p(\TT^d))$ for $p\in[2,\infty)$, Corollary~\ref{cor_exist} treats random initial data with finite entropy in the sense of Definition~\ref{def_entropy} below, and Corollary~\ref{cor_exist_1} treats random initial data that is only $L^1$-integrable.

Our assumptions require mild \emph{local} regularity assumptions for $\sigma$ on $(0,\infty)$, which in the model case $\Phi(\xi)=\xi^m$ require that $\sigma^2$ grows at most like $\xi^{m+1}$ at infinity, and that $\sigma^2$ vanishes linearly at zero (see Assumption~\ref{assume_1}, Assumption~\ref{assume}, and Example~\ref{example_1} below).  In the Dean--Kawasaki case $\sigma=\Phi^{\nicefrac{1}{2}}$, the final of these assumptions requires $m\in[1,\infty)$, but notably we do not impose any further regularity of $\sigma$ at zero which allows to treat the square root.  The general results for \eqref{1_000} are exactly analogous to those for \eqref{1_0} concerning the integrability of the data and regularity of the coefficients.

\subsubsection*{The noise}
Concerning the noise $\dot{\xi}^F$, we require that the coefficients $F_1$, $F_2$, and $F_3=\sum_{k=1}^\infty\abs{\nabla f_k}^2$ are continuous on $\TT^d$ and that the divergence $\nabla\cdot F_2$ is bounded on $\TT^d$.  In the model Dean--Kawasaki case with $m=1$, we assume further that $\nabla\cdot F_2=0$.  This means that the noise is probabilistically stationary, a property satisfied by space-time white noise and all of its standard approximations (see Remarks~\ref{rem_noise_1} and \ref{remark_infinite_noise} below).  Concerning the noise $\dot{\xi}^G$, we require only that the sum $G_1=\sum_{k=1}^\infty g_k^2$ is bounded and continuous on $\TT^d$.

\subsection{Overview of the literature}\label{sec_lit}  The methods of this paper are most closely related to the works of the two authors \cite{FehGes2019,FehGes2021}, which develop a kinetic approach to prove the path-by-path well-posedness of equations like \eqref{1_0} and \eqref{1_000} with linear $\phi$.  However, unlike the probabilistic approach taken in this paper, the methods of these works were motivated by the theory of stochastic viscosity solutions (see Lions and Souganidis \cite{LioSou5,LioSou4,LioSou3,LioSou2,LioSou1}), and the work on stochastic conservation laws of Lions, Perthame, and Souganidis \cite{LioPerSou2013,LioPerSou2014} and the second author and Souganidis \cite{GesSou2015,GesSou2017}.  Furthermore, the pathwise well-posedness theory of \cite{FehGes2019,FehGes2021} is based on rough path techniques, which in the context of this paper would require the nonlinearity $\sigma$ to be six-times continuously differentiable.

In addition to \cite{FehGes2019} the only other approach to equations like \eqref{1_0} with spatially inhomogeneous noise was developed by the second author and Dareiotis \cite{DarGes2020}, who construct probabilistic solutions to equations like \eqref{1_0} in a simpler context using the entropy formulation of the equation.  The work \cite{DarGes2020} applied only to the conservative case $\phi,\lambda=0$ and required the considerably stronger regularity assumption $\sigma\in\C^{1,\beta}$ for some $\beta\in(0,\infty)$ sufficiently large, in addition to the other conditions of \cite[Assumption 2.3]{DarGes2020}. In particular, this excludes the important case of square root diffusion coefficients as in \eqref{1_0101}. Furthermore, in the conservative case \eqref{1_0}, their main result \cite[Theorem~2.7]{DarGes2020} obtains the $L^1$-contraction of solutions only in expectation, as opposed to the pathwise result of \eqref{eq:contraction} above.  Finally, a significant advantage of the kinetic formulation over the entropy formulation is that, due to the precise identification of the kinetic defect measures, in this work we treat $L^1$-integrable initial data, as opposed to $L^{m+1}$-integrable initial data in the porous media case $\Phi(\xi)=\xi^m$, and require only local as opposed to global regularity from the solution, see Definition~\ref{def_sol} and specifically \eqref{def_2500000} below.

Equations of the form \eqref{1_000} with linear diffusions, $\sigma,\nu=0$, and with a $\gamma$-H\"older continuous noise coefficient have received significant attention in the literature  going back to Viot \cite{Vio1976}, due to their relevance to branching diffusion processes and population genetics. In particular, the strong uniqueness of solutions to semilinear stochastic heat equations with non-Lipschitz continuous coefficients and correlated noise has been shown by Mytnik, Perkins and Sturm in \cite{MytPerStu2006}, relying on regularity estimates obtained by  Sanz-Sol\'e and Sarr\`a in \cite{SanSar2002}.  We emphasize that \cite[Theorem~1.4]{MytPerStu2006} treats noise that is less regular than that considered in this paper, and investigates the question of pathwise uniqueness in regimes relating the H\"older regularity of $\phi$ to the decay of the spatial correlations of the noise.

In contrast, still in the case $\sigma,\nu=0$, the results of the present work treat nonlinear diffusions $\Phi$ and noise with bounded covariance.  This solves a problem left open by the second author, Dareiotis, and Gerencs\'{e}r \cite{DarGerGes2019}, where nonlinearities $\phi$ which are $(\nicefrac{1}{2}+\d)$-H\"older continuous for some $\d\in(0,\nicefrac{1}{2}]$ could be handled, thus leaving open the most relevant case $\phi(\xi)=\xi^{\nicefrac{1}{2}}$ treated in this work.

In the case of space-time white noise (now for $\sigma,\lambda,\nu=0$) the weak uniqueness of non-negative solutions to the stochastic heat equation has been shown by Perkins in \cite[Corollary III.4.3]{Perk2002} and Mytnik \cite{Myt1998} for noise coefficients of the form $\phi(\xi)=\xi^\gamma$ for $\gamma\in(\nicefrac{1}{2},1)$.  Pathwise uniqueness for $\gamma$-H\"older continuous $\phi$ for $\gamma\in(\nicefrac{3}{4},1)$ was shown by Mytnik and Perkins in \cite{MytPer2011}. Counter-examples for pathwise uniqueness for signed solutions have been developed by Mueller, Mytnik, and Perkins \cite{MueMytPer2014} for noise coefficients with $\gamma < \nicefrac{3}{4}$. Pathwise uniqueness for non-negative solutions with $\gamma < \nicefrac{3}{4}$ and space-time white noise is an open problem.

The well-posedness of the Dean--Kawasaki equation has attracted considerable interest in the literature. The existence of solutions to corrected / modified Dean Kawasaki equations has been shown by Sturm, von Renesse in \cite{StuvRe2009} by means of Dirichlet forms techniques. Subsequently, alternative constructions have been given in \cite{AndvRe2010} and \cite{KonvRe2019}. Negative results on the existence to the unmodified Dean--Kawasaki equation with space-time white It\^o-noise have been recently obtained in \cite{KonLehVon2019,KonLehvRe2020}. A regularized model replacing the Dean--Kawasaki equation was analyzed by Cornalba, Shardlow, and Zimmer in \cite{CorShaZim2019,CorShaZim2020}, by means of smoothed particles with second order (underdamped) Langevin dynamics. The well-posedness of the Dean--Kawasaki equation with correlated noise was until now an open problem.

Previous works considering the kinetic formulation of scalar conservation laws in simpler settings include, for example, Debussche and Vovelle \cite{DebVov2010}, Hofmanov\'a \cite{Hof2013}, and Debussche, Hofmanov\'a, and Vovelle \cite{DebHofVov2016}.  Finally, there is an extensive additional literature on stochastic nonlinear diffusion equations with additive or multiplicative noise.  See, for example, Barbu, Bogachev, Da Prato, and R\"ockner \cite{BarBogDapRock2006}, Barbu, Da Prato, and R\"ockner \cite{BarDapRoe2008,BarDapRoe20082,BarDapRoe2009,BarDapRoe2016}, Barbu and R\"ockner \cite{BarRoe2011}, Barbu, R\"ockner, and Russo \cite{BarbuRoecknerRusso}, Da Prato and R\"ockner \cite{DapRoe2004}, Da Prato, R\"ockner, Rozovski\u\i, and Wang \cite{DapRoeRozWan2006}, the second author \cite{Gess2012}, Kim \cite{Kim2006}, Krylov and Rozovski\u\i\ \cite{KryRoz1977,KryRoz2007}, Pardoux \cite{Par1972}, Pr\'ev\^ot and R\"ockner \cite{PreRoe2007}, Ren, R\"ockner, and Wang \cite{RenRoeWan2007}, R\"ockner and Wang \cite{RoeWan2008}, and Rozovski\u\i\ \cite{Roz1990}.

\subsection{Organization of the paper}  Section~\ref{sec_noise} introduces the assumptions for the noise.  Section~\ref{sec_kinetic} derives the kinetic formulation of \eqref{1_0} and defines in Definition~\ref{def_sol} a stochastic kinetic solution.  Section~\ref{sec_unique} proves the uniqueness of stochastic kinetic solutions.  We construct the solution in Section~\ref{sec_whole}, which is split into three subsections.  Section~\ref{sec_a_priori} establishes a priori estimates and Section~\ref{sec_exist_approx} proves the existence of solutions to approximating versions of \eqref{1_0} with a smooth and bounded nonlinearity $\sigma$.  Section~\ref{sec_exist} proves the existence of stochastic kinetic solutions for general $\sigma$.  Finally, Section~\ref{sec_gen} extends the results for \eqref{1_0} to equation \eqref{1_000}.

\section{The definition of the noise}\label{sec_noise}

In this section, we define the noise $\xi^F$.  For a sequence of continuously differentiable functions $F=(f_k)_{k\in\N}$ on $\TT^d$ and independent Brownian motions $\{B^k\}_{k\in\N}$, we define the noise $\xi^F = \sum_{k=1}^\infty f_k(x) B^k_t$.  It then follows from the definition of $\xi^F$ that the Stratonovich equation \eqref{1_0} is formally equivalent to the It\^o equation
\[\dd\rho = \Delta \Phi(\rho)\dt -\nabla\cdot \left(\sigma(\rho)\dd\xi^F+\nu(\rho)\dt\right)\dt+\frac{1}{2}\sum_{k=1}^\infty \nabla\cdot\left(f_k\sigma'(\rho)\nabla\left(f_k\sigma(\rho)\right)\right)\dt\;\;\textrm{in}\;\;\TT^d\times(0,T),\]
which can be written in the form
\[\dd\rho = \Delta \Phi(\rho)\dt -\nabla\cdot \left(\sigma(\rho)\dd\xi^F+\nu(\rho)\dt\right)+\frac{1}{2}\nabla\cdot\left(F_1[\sigma'(\rho)]^2\nabla\rho+\sigma'(\rho)\sigma(\rho)F_2\right)\dt\;\;\textrm{in}\;\;\TT^d\times(0,T),\]
for $F_1\colon\TT^d\rightarrow\R$ and $F_2\colon\TT^d\rightarrow\R^d$ defined by
\[F_1(x)=\sum_{k=1}^\infty f_k^2(x)\;\;\textrm{and}\;\;F_2(x)=\sum_{k=1}^\infty f_k(x)\nabla f_k(x).\]
We will make the following assumptions on the randomness in the equation, which includes the initial condition.

\begin{assumption}\label{assumption_noise}  Let $\{B^k\}_{k\in\N}$ be independent $d$-dimensional Brownian motions defined on a probability space $(\O,\F,\P)$ with respect to a filtration $(\F_t)_{t\in[0,\infty)}$ and let $\{f_k\colon\TT^d\rightarrow\R\}_{k\in\N}$ be continuously differentiable functions.  Assume that the sums $\{F_i\}_{i\in\{1,2,3\}}$ defined by
\[F_1=\sum_{k=1}^\infty f_k^2\;\;\textrm{and}\;\;F_2=\frac{1}{2}\sum_{k=1}^\infty \nabla f^2_k\;\;\textrm{and}\;\;F_3=\sum_{k=1}^\infty\abs{\nabla f_k}^2\]
are continuous on $\TT^d$---where the finiteness of $F_1$ and $F_3$ implies the absolute convergence of $F_2$---and assume that the divergence
\[\nabla\cdot F_2=\frac{1}{2}\Delta F_1\;\;\textrm{is bounded on $\TT^d$.}\]
We define $\xi^F$ as above and we assume that the initial data $\rho_0\in L^1(\O;L^1(\TT^d))$ is nonnegative and $\F_0$-measurable.  \end{assumption}

\begin{remark}\label{rem_noise_1}  For some statements, and particularly in the Dean--Kawasaki case \eqref{1_0101}, we will require that $\nabla\cdot F_2=\frac{1}{2}\Delta F_1=0$ (in fact, we only require that $\nabla\cdot F_2\geq 0$ but on the torus this is equivalent to $\nabla\cdot F_2=0$).  This amounts to $F_1$ being constant on the torus, which states that the noise is probabilistically stationary in the sense that it has the same law at every point in space.  This is a property satisfied by space-time white noise $\xi$ and all of its standard approximations, like every spatial convolution $\xi^\ve=(\xi*\kappa^\ve)$ and the noise discussed in Remark~\ref{remark_infinite_noise} below.

\end{remark}

\begin{remark}\label{remark_infinite_noise}  Important examples falling into the framework of this paper are arbitrary spatial convolutions of space-time white noise $\xi$ defined by $\xi^\ve=(\xi*\kappa^\ve)$ for smooth kernels $\kappa^\ve$, and the noise $\xi^a$ defined by
\[\xi^a = \sum_{k\in\Z^d} a_k\left(\sin(k\cdot x)B^k_t+\cos(k\cdot x)W^k_t\right),\]
for $\{B^k,W^k\}_{k\in\Z^d}$ independent Brownian motions defined on a probability space $(\O,\F,\P)$ with respect to a filtration $(\F_t)_{t\in[0,\infty)}$ and coefficients $a=(a_k)_{k\in\Z^d}$.  This is the standard spectral approximation of space-time white noise, and an explicit computation proves that
\[F_1=\sum_{k\in\Z^d}a_k^2\;\;\textrm{and}\;\;F_2=0\;\;\textrm{and}\;\;F_3=\sum_{k\in\Z^d}\abs{k}^2a_k^2.\]
Our methods apply to noise of this type provided $\sum_{k\in\Z^d}\abs{k}^2a_k^2<\infty$.\end{remark}

\section{The kinetic formulation of \eqref{1_0}}\label{sec_kinetic}

The following formal computations motivate the definition of a stochastic kinetic solution to the It\^o equation
\begin{align}\label{2_0} \dd \rho = \Delta\Phi(\rho)\dt -\nabla\cdot\left(\sigma(\rho)\dd\xi^F+\nu(\rho)\dt\right)+\frac{1}{2}\nabla\cdot\left(F_1[\sigma'(\rho)]^2\nabla\rho+\sigma(\rho)\sigma'(\rho)F_2\right)\dt,\end{align}
for coefficients $F_i$ defined in Assumption~\ref{assumption_noise}.  The computations are based on the kinetic formulation of scalar conservation laws.  If $S\colon\R\rightarrow\R$ is a smooth function, then after applying It\^o's formula we have formally that (see the discussion following \eqref{2_04} below---in general this equation will only be satisfied with an inequality)
\begin{align*}
\dd S(\rho) & = \nabla\cdot\left(\Phi'(\rho)S'(\rho)\nabla\rho\right)\dt+\frac{1}{2}\nabla\cdot \left(F_1[\sigma'(\rho)]^2S'(\rho)\nabla\rho+\sigma(\rho)\sigma'(\rho)S'(\rho)F_2\right)\dt
\\ &  \quad- S''(\rho)\left(\Phi'(\rho)\abs{\nabla\rho}^2 -\frac{1}{2}F_1\abs{\nabla\sigma(\rho)}^2-\frac{1}{2}\sigma(\rho)\nabla\sigma(\rho)\cdot F_2\right)\dt
\\ & \quad +\frac{1}{2}S''(\rho)\sum_{k=1}^\infty(\nabla(\sigma(\rho)f_k))^2\dt-S'(\rho)\nabla\cdot\nu(\rho)\dt- S'(\rho)\nabla\cdot(\sigma(\rho)\dd\xi^F).
\end{align*}
Since we have that
\[\sum_{k=1}^\infty(\nabla(\sigma(\rho)f_k))^2=F_1\abs{\nabla\sigma(\rho)}^2+2\sigma(\rho)\nabla\sigma(\rho)\cdot F_2+\sigma(\rho)^2F_3,\]
it follows that
\begin{align*}
\dd S(\rho) & = \nabla\cdot\left(\Phi'(\rho)S'(\rho)\nabla\rho\right)\dt+\frac{1}{2}\nabla\cdot \left(F_1[\sigma'(\rho)]^2S'(\rho)\nabla\rho+\sigma(\rho)\sigma'(\rho)S'(\rho)F_2\right)\dt
\\ & \quad  - S''(\rho)\Phi'(\rho)\abs{\nabla\rho}^2\dt +\frac{1}{2}S''(\rho)\left(\sigma(\rho)\sigma'(\rho)\nabla\rho\cdot F_2+\sigma(\rho)^2F_3\right)\dt
\\ & \quad -S'(\rho)\nabla\cdot\nu(\rho)\dt- S'(\rho)\nabla\cdot(\sigma(\rho)\dd\xi^F).
\end{align*}
Given a nonnegative solution $\rho$ of \eqref{2_0}, the kinetic function $\chi\colon\TT^d\times\R\times[0,T]\rightarrow\{0,1\}$ of $\rho$ is defined by
\[\overline{\chi}(\rho)=\chi(x,\xi,t)=\mathbf{1}_{\{0<\xi<\rho(x,t)\}}.\]
Provided $S(0)=0$ it follows from the identity
\[S(\rho(x,t))=\int_{\R}S'(\xi)\chi(x,\xi,t)\dxi,\]
and the density of linear combinations of functions of the type $S'(\xi)\psi(x)$ in $\C^\infty_c(\TT^d\times(0,\infty))$ for $\psi\in\C^\infty(\TT^d)$, that the kinetic function $\chi$ solves
\begin{align}\label{2_04}
\dd\chi  & = \nabla\cdot\left(\delta_0(\xi-\rho)\Phi'(\xi)\nabla\rho\right)\dt+\frac{1}{2}\nabla\cdot \left(\delta_0(\xi-\rho)\left(F_1[\sigma'(\xi)]^2\nabla\rho+\sigma(\xi)\sigma'(\xi)F_2\right)\right)\dt
\\ \nonumber & \quad  +\partial_\xi\left(\delta_0(\xi-\rho)\left(\Phi'(\xi)\abs{\nabla\rho}^2-\frac{1}{2}\sigma(\xi)\sigma'(\rho)\nabla\rho\cdot F_2-\frac{1}{2}\sigma(\xi)^2F_3\right)\right)\dt
\\ \nonumber & \quad -\delta_0(\xi-\rho)\nabla\cdot\nu(\rho)\dt-\delta_0(\xi-\rho)\nabla\cdot(\sigma(\rho)\dd\xi^F).
\end{align}
Equation \eqref{2_04} is the starting point for our solution theory.  However, the kinetic function will not in general satisfy \eqref{2_04} exactly.  On the level of an entropy solution this appears in terms of an entropy inequality, where for convex $S$ the equation satisfied by $S(\rho)$ is not satisfied with equality but with the inequality that $\dd S(\rho)$ is less than or equal to the righthand side.

On the kinetic level the entropy inequality is quantified exactly by a kinetic measure.  This is a nonnegative measure $q$ on $\TT^d\times\R\times[0,T]$ such that, in the sense of measures,
\[\delta_0(\xi-\rho)\Phi'(\xi)\abs{\nabla\rho}^2\leq q,\]
and such that in the sense of distributions the kinetic function $\chi$ solves the equation
\begin{align}\label{2_4}
\dd\chi  & = \nabla\cdot\left(\delta_0(\xi-\rho)\Phi'(\xi)\nabla\rho\right)\dt+\frac{1}{2}\nabla\cdot \left(\delta_0(\xi-\rho)\left(F_1[\sigma'(\xi)]^2\nabla\rho+\sigma(\xi)\sigma'(\xi)F_2\right)\right)\dt
\\ \nonumber & \quad  +\partial_\xi q\dt-\frac{1}{2}\partial_\xi\left(\delta_0(\xi-\rho)\left(\sigma(\xi)\sigma'(\xi)\nabla\rho\cdot F_2+\sigma(\xi)^2F_3\right)\right)\dt
\\ \nonumber & \quad -\delta_0(\xi-\rho)\nabla\cdot\nu(\rho)\dt-\delta_0(\xi-\rho)\nabla\cdot(\sigma(\rho)\dd\xi^F).
\end{align}
Motivated by \eqref{2_4}, we define a kinetic measure in Definition~\ref{def_measure} and we define a \emph{stochastic kinetic solution} of \eqref{2_4} in Definition~\ref{def_sol}.

\begin{definition}\label{def_measure}  Let $(\O,\F,\P)$ be a probability space with a filtration $(\F_t)_{t\in[0,\infty)}$.  A kinetic measure is a map $q$ from $\O$ to the space of nonnegative, locally finite measures on $\TT^d\times(0,\infty)\times[0,T]$ that satisfies the property that the process
\[(\o,t)\in\O\times[0,T]\rightarrow \int_0^t\int_\R\int_{\TT^d}\psi(x,\xi)\dd q(\o)\;\;\textrm{is $\F_t$-predictable,}\]
for every $\psi\in\C^\infty(\TT^d\times(0,\infty))$.
\end{definition}

\begin{remark}\label{rem_derivative}  In the kinetic formulation, we will frequently encounter derivatives of functions $\psi\in\C^\infty_c(\TT^d\times\R)$ evaluated at the point $\xi=\rho(x,t)$.  We will write
\[(\nabla\psi)(x,\rho(x,t))=\left.\nabla\psi(x,\xi)\right|_{\xi=\rho(x,t)}\]
to mean the gradient of $\nabla\psi$ evaluate at the point $(x,\rho(x,t))$ as opposed to the full gradient of the composition $\psi(x,\rho(x,t))$.\end{remark}

\begin{remark}  For $\F_t$-adapted processes $g_t\in L^2(\O\times[0,T];L^2(\TT^d))$ and $h_t\in L^2(\O\times[0,T];H^1(\TT^d))$ and for $t\in[0,T]$, we will write
\begin{align*}
& \int_0^t\int_{\TT^d}g_s\nabla\cdot\left(h_s\dd\xi^F\right) =\sum_{k=1}^\infty\left(\int_0^t\int_{\TT^d}g_sf_k\nabla h_s\cdot \dd B^k_s+\int_0^t\int_{\TT^d}g_sh_s\nabla f_k\cdot \dd B^k_t\right),
\end{align*}
where the integrals are interpreted in the It\^o sense.  \end{remark}

\begin{definition}\label{def_sol}  Let $\rho_0\in L^1(\O;L^1(\O))$ be nonnegative and $\F_0$-measurable.  A \emph{stochastic kinetic solution} of \eqref{2_0} is a nonnegative, almost surely continuous $L^1(\TT^d)$-valued $\F_t$-predictable function $\rho\in L^1(\O\times[0,T];L^1(\TT^d))$ that satisfies the following three properties.
\begin{enumerate}[(i)]
\item \emph{Preservation of mass}:  almost surely for every $t\in[0,T]$,
\begin{equation}\label{def_3535}\norm{\rho(\cdot,t)}_{L^1(\TT^d)}=\norm{\rho_0}_{L^1(\TT^d)}.\end{equation}
\item \emph{Integrability of the flux}:  we have that
\[\sigma(\rho)\in L^2(\O;L^2(\TT^d\times[0,T]))\;\;\textrm{and}\;\;\nu(\rho)\in L^1(\O;L^1(\TT^d\times[0,T];\R^d)).\]
\item \emph{Local regularity}:  for every $K\in\N$,
\begin{equation}\label{def_2500000} [(\rho\wedge K)\vee\nicefrac{1}{K}]\in L^2(\O;L^2([0,T];H^1(\TT^d))).\end{equation}
\end{enumerate}
Furthermore, there exists a kinetic measure $q$ that satisfies the following three properties.
\begin{enumerate}[(i)]
\setcounter{enumi}{3}
\item \emph{Regularity}: almost surely as nonnegative measures,
\begin{equation}\label{2_500}\delta_0(\xi-\rho)\Phi'(\xi)\abs{\nabla\rho}^2\leq q\;\;\textrm{on}\;\;\TT^d\times(0,\infty)\times[0,T].\end{equation}
\item \emph{Vanishing at infinity}:  we have that
\begin{equation}\label{def_5353}\lim_{M\rightarrow\infty}\E\left[q(\TT^d\times[M,M+1]\times[0,T])\right]=0.\end{equation}
\item \emph{The equation}: for every $\psi\in \C^\infty_c(\TT^d\times(0,\infty))$, almost surely for every $t\in[0,T]$,
\begin{align}\label{2_5000}
& \int_\R\int_{\TT^d}\chi(x,\xi,t)\psi(x,\xi) = \int_\R\int_{\TT^d}\overline{\chi}(\rho_0)\psi(x,\xi)-\int_0^t\int_{\TT^d}\Phi'(\rho)\nabla\rho\cdot(\nabla\psi)(x,\rho)
\\ \nonumber & \quad -\frac{1}{2}\int_0^t\int_{\TT^d}F_1(x)[\sigma'(\rho)]^2\nabla\rho\cdot(\nabla\psi)(x,\rho)-\frac{1}{2}\int_0^t\int_{\TT^d}\sigma(\rho)\sigma'(\rho)F_2(x)\cdot(\nabla\psi)(x,\rho)
\\ \nonumber & \quad -\int_0^t\int_\R\int_{\TT^d}\partial_\xi\psi(x,\xi)\dd q+\frac{1}{2}\int_0^t\int_{\TT^d}\left(\sigma(\rho)\sigma'(\rho)\nabla\rho\cdot F_2(x)+F_3(x)\sigma^2(\rho)\right)(\partial_\xi\psi)(x,\rho)
\\ \nonumber & \quad -\int_0^t\int_{\TT^d}\psi(x,\rho)\nabla\cdot\nu(\rho)\dt-\int_0^t\int_{\TT^d}\psi(x,\rho)\nabla\cdot\left(\sigma(\rho)\dd\xi^F\right).
\end{align}
\end{enumerate}
\end{definition}

\begin{remark} Since a stochastic kinetic solution is a continuous in time $L^1(\TT^d)$-valued process, every term of \eqref{2_5000} is a priori continuous in time, except possibly the term involving the kinetic measure $q$. As a consequence, also this term is continuous in time, which implies that the measure $q$ has no atoms in time in the sense that $q(\TT^d\times\R\times\{t\})=0$ for every $t\in[0,T]$. This means that there is no ambiguity when interpreting the integral in the fourth term on the righthand side of \eqref{2_5000}.
\end{remark}

\begin{remark}\label{rem_0}In Definition~\ref{def_sol} it is essential that test functions $\psi$ are restricted to be compactly supported in $\TT^d\times(0,\infty)$.  This amounts to a renormalization of the equation that restricts the solution to be away from its zero set.  As was shown in \eqref{1_21}, the third term on the righthand side of \eqref{2_5000} is not in general integrable unless $\psi$ is compactly supported away from zero in the $\xi$-variable, and the kinetic measure $q$ is not finite unless the initial data is $L^2$-integrable. Furthermore, in general, the solution $\rho$ is only regular on sets compactly supported away from its zero set.  \end{remark}

\section{Uniqueness of stochastic kinetic solutions to \eqref{1_0}}\label{sec_unique}

The proof of uniqueness is based on the following formal computation, which we will demonstrate for the particular choice of noise defined in Remark~\ref{remark_infinite_noise} and for $\nu=0$.  In this case, for the nonnegative real numbers $N_a,M_a\in\R$ defined by
\[N_a=\sum_{k=1}^\infty a_k^2\;\;\textrm{and}\;\;M_a=\sum_{k=1}^\infty \abs{k}^2a_k^2,\]
equation \eqref{2_4} takes the simpler form
\begin{align}\label{2_022}
\dd\chi  & = \nabla\cdot\left(\delta_0(\xi-\rho)\Phi'(\xi)\nabla\rho\right)\dt+\frac{N_a}{2}\nabla\cdot \left(\delta_0(\xi-\rho)[\sigma'(\xi)]^2\nabla\rho\right)\dt
\\ \nonumber & \quad  +\partial_\xi q\dt-\frac{M_a}{2}\partial_\xi\left(\delta_0(\xi-\rho)\sigma(\xi)^2\right)\dt-\delta_0(\xi-\rho)\nabla\cdot(\sigma(\rho)\dd\xi^a).
\end{align}
If $\rho^1,\rho^2$ are two stochastic kinetic solutions of \eqref{2_022} with initial data $\rho^1_0,\rho^2_0$, it follows from properties of the corresponding kinetic functions $\chi^1,\chi^2$ that
\begin{equation}\label{3_0000}\int_\R\abs{\chi^1-\chi^2}^2\dxi = \abs{\rho^1(x,t)-\rho^2(x,t)}\;\;\textrm{and}\;\;\abs{\chi^1-\chi^2}^2 = \chi^1+\chi^2-2\chi^1\chi^2.\end{equation}
Therefore,
\begin{equation}\label{3_000}\dd\int_{\TT^d}\abs{\rho^1-\rho^2} = \dd\int_\R\int_{\TT^d}\chi^1\dx\dxi+\dd\int_\R\int_{\TT^d}\chi^2\dx\dxi-2\dd\int_\R\int_{\TT^d}\chi^1\chi^2\dx\dxi.\end{equation}
After choosing $\psi=1$ in \eqref{2_5000}---which is however not justified in view of Remark~\ref{rem_0} and needs a careful rigorous analysis---it follows almost surely that, for every $i\in\{1,2\}$,
\begin{equation}\label{3_00}\dd\int_\R\int_{\TT^d}\chi^i\dx\dxi=-\int_{\TT^d}\nabla\cdot(\sigma(\rho^i)\dd\xi^a)=0.\end{equation}
The stochastic product rule and the distributional equalities $\partial_\xi\chi^i=\delta_0(\xi)-\delta_0(\xi-\rho^i)$ and $\nabla_x\chi^i = \delta_0(\xi-\rho^i)\nabla\rho^i$ prove that the mixed term almost surely satisfies
\begin{equation}\label{3_0}\dd\int_\R\int_{\TT^d}\chi^1\chi^2  = \dd I^\textrm{meas}+\dd I^\textrm{mart}+\dd I^\textrm{err},\end{equation}
for the measure term
\begin{align*}
\dd I^{\textrm{meas}} & = -2\int_\R\int_{\TT^d}\delta_0(\xi-\rho^1)\delta_0(\xi-\rho^2)[\Phi'(\rho^1)]^\frac{1}{2}[\Phi'(\rho^2)]^\frac{1}{2}\nabla\rho^1\cdot\nabla \rho^2\dt
\\ \nonumber & \quad -\int_\R\int_{\TT^d}(\delta_0(\xi)-\delta_0(\xi-\rho^2))\dd q^1-\int_\R\int_{\TT^d}(\delta_0(\xi)-\delta_0(\xi-\rho^1))\dd q^2\dt, 
\end{align*}
for the martingale term
\[\dd I^\textrm{mart} = -\int_{\TT^d}\chi^2(x,\rho^1,t)\nabla\cdot(\sigma(\rho^1)\dd\xi^a)+\chi^1(x,\rho^2,t)\nabla\cdot(\sigma(\rho^2)\dd\xi^a),\]
and for the error term
\begin{align*}
 \dd I^\textrm{err} & = -N_a\int_\R\int_{\TT^d}\delta_0(\xi-\rho^1)\delta_0(\xi-\rho^2)\sigma'(\rho^1)\sigma'(\rho^2)\nabla\rho^1\cdot\nabla\rho^2\dt
\\ \nonumber & +\frac{M_a}{2}\int_\R\int_{\TT^d}\delta_0(\xi-\rho^1)(\delta_0(\xi)-\delta_0(\xi-\rho^2))\sigma(\rho^1)^2+\delta_0(\xi-\rho^2)(\delta_0(\xi)-\delta_0(\xi-\rho^1))\sigma(\rho^2)^2\dt
\\ \nonumber & +\dd\int_\R\int_{\TT^d}\langle\chi^1,\chi^2\rangle_t.
\end{align*}
It follows formally from the definition of $\xi^a$ and \eqref{2_4} that
\begin{equation}\label{3_1}\dd\int_\R\int_{\TT^d}\langle \chi^1,\chi^2\rangle_t = \int_\R\int_{\TT^d}\delta_0(\xi-\rho^1)\delta_0(\xi-\rho^2)\left(N_a\sigma'(\rho^1)\sigma'(\rho^2)\nabla\rho^1\cdot\nabla\rho^2+M_a\sigma(\rho^1)\sigma(\rho^2)\right)\dt.\end{equation}
Returning to \eqref{3_0}, it follows from \eqref{2_500}, \eqref{3_1}, and H\"older's inequality that the measure term is nonnegative, and it follows from \eqref{3_1} that the error term vanishes.  For the martingale term, the formal identities $\chi^2(x,\rho^1,t)=\mathbf{1}_{\{\rho^2-\rho^1>0\}}$ and $\chi^1(x,\rho^2,t)=\mathbf{1}_{\{\rho^1-\rho^2>0\}}$ prove that
\[\dd I^\textrm{mart}= \int_{\TT^d}\delta_0(\rho^1-\rho^2)(\nabla\rho^2-\nabla\rho^1)\sigma(\rho^1)\dd\xi^a+\int_{\TT^d}\delta_0(\rho^2-\rho^1)(\nabla\rho^1-\nabla\rho^2)\sigma(\rho^2)\dd\xi^a=0.\]
Returning to \eqref{3_000}, it follows almost surely from \eqref{3_00}, \eqref{3_0}, and the $L^1(\TT^d)$-continuity of stochastic kinetic solutions that, for every $t\in[0,T]$,
\[\int_{\TT^d}\abs{\rho^1(x,t)-\rho^2(x,t)}\dx\leq \int_{\TT^d}\abs{\rho^1_0(x)-\rho^2_0(x)}\dx.\]
We make these computations rigorous in Theorem~\ref{thm_unique}.  In particular, the products of delta distributions are not defined and must be treated using commutator estimates.  And in accordance with Remark~\ref{rem_0}, neither the function $\psi=1$ nor the kinetic functions are admissible test functions.  For this reason it is necessary to introduce cutoff functions that create a singularities at zero.  These singularities are treated using Proposition~\ref{prop_measure}.

The proof of uniqueness is presented in Theorem~\ref{thm_unique} under the assumptions on $\Phi$, $\sigma$, and $\nu$ presented in Assumption~\ref{assume_1}.  Lemma~\ref{lem_ibp} proves an integration by parts formula for the kinetic function, Definition~\ref{def_smooth} introduces the convolution kernels and cutoff functions that will be used repeatedly in the proof, and Proposition~\ref{prop_measure} controls the kinetic measure at zero.

\begin{assumption}\label{assume_1}  Assume that $\Phi,\sigma\in\C([0,\infty))$ and $\nu\in\C([0,\infty);\R^d)$ satisfy the following six assumptions.
\begin{enumerate}[(i)]
\item We have that $\Phi,\sigma\in \C^{1,1}_{\textrm{loc}}((0,\infty))$ and $\nu\in\C^1_{\textrm{loc}}((0,\infty);\R^d)$.
\item  We have that $\Phi(0)=0$ and with $\Phi'>0$ on $(0,\infty)$.
\item There exists $c\in(0,\infty)$ such that
\[\limsup_{\xi\rightarrow 0^+}\frac{\sigma^2(\xi)}{\xi}\leq c,\]
which implies that $\sigma(0)=0$.
\item We have either that $\sigma\sigma'\in\C([0,\infty))$ with $(\sigma\sigma')(0)=0$ or that $\nabla\cdot F_2=0$ for $F_2$ defined in Assumption~\ref{assumption_noise}.
\item There exists $c\in[1,\infty)$ such that
\begin{equation}\label{assume_f6} \left(\sup\nolimits_{\xi'\in[0,\xi]}\sigma^2(\xi')\right)\leq c(1+\xi+\sigma^2(\xi))\;\;\textrm{for every}\;\;\xi\in[0,\infty).\end{equation}
\item There exists $c\in[1,\infty)$ such that
\begin{equation}\label{assume_f7}\left(\sup\nolimits_{\xi'\in[0,\xi]}\abs{\nu(\xi')}\right)\leq c\left(1+\xi+\abs{\nu(\xi)}\right)\;\;\textrm{for every}\;\;\xi\in[0,\infty).\end{equation}
\end{enumerate}
\end{assumption}

\begin{remark}  We observe that assumption \eqref{assume_f6} essentially amounts to a regularity assumption on the magnitude of the oscillations of $\sigma$ at infinity. This is demonstrated by the following aspects:
\begin{enumerate}
\item Condition \eqref{assume_f6}  is satisfied if $\sigma^2$ is monotone or if $\sigma^2$ grows linearly for large values of $\xi$.  Condition \eqref{assume_f6} is furthermore satisfied if  $\sigma^2$ has locally uniformly bounded oscillations, in the sense that, for some $c\in(0,\infty)$, for every $M\in(0,\infty)$,
\[\left(\sup\nolimits_{\xi\in[M,M+1]}\sigma^2(\xi)-\inf\nolimits_{\xi\in[M,M+1]}\sigma^2(\xi)\right)\leq c.\]
To see this, we observe that in this case $\sigma^2$ grows linearly in the sense that
\begin{align*}
\left(\sup\nolimits_{\xi'\in[0,\xi]}\sigma^2(\xi')\right) & \leq \sum_{k=0}^\infty\left(\sup\nolimits_{\xi'\in[M\wedge\xi,(M+1)\wedge\xi]}\sigma^2(\xi')-\inf\nolimits_{\xi'\in[M\wedge\xi,(M+1)\wedge\xi]}\sigma^2(\xi')\right)
\\ & \leq c(1+\xi).
\end{align*}
In particular, if $\sigma^2$ is uniformly continuous or if $\sigma$ satisfies, for some $c\in(0,\infty)$, for every $M\in(0,\infty)$, and for every $\xi\in[M,M+1]$,
\[\left(\sup\nolimits_{\xi'\in[M,(M+1)\wedge\xi]}\sigma^2(\xi')-\sigma^2(\xi)\right)\leq c,\]
then $\sigma^2$ has locally uniformly bounded oscillations and, thus, Condition \eqref{assume_f6} is satisfied.
\item Condition \eqref{assume_f6} is substantially more general than the above five conditions, and it allows for the oscillations of $\sigma^2$ to grow linearly at infinity.  Precisely, it follows from \eqref{assume_f6} that, for every $M\in(0,\infty)$,
\begin{align*}
\left(\sup\nolimits_{\xi\in[M,M+1]}\sigma^2(\xi)-\inf\nolimits_{\xi\in[M,M+1]}\sigma^2(\xi)\right)& \leq\left(\sup\nolimits_{\xi\in[M,M+1]}\sigma^2(\xi)-c\sigma^2(M+1)\right)
\\ & \leq c(1+M).
\end{align*}
A model case satisfying all of the conditions of Assumption~\ref{assume_1} is the function $\sigma^2(\xi)=\xi^m+\xi\sin(\xi^p)$ for every $m,p\in[1,\infty)$---that is, condition \eqref{assume_f6} imposes a condition on the growth of the magnitude of the oscillations of $\sigma^2$ but not on the growth of the frequency of the oscillations.

\item Assumption~\eqref{assume_f6} is satisfied by every model case, and it is used to guarantee the following condition:  for every $\rho\in L^\infty([0,T];L^1(\TT^d))$ that satisfies $\sigma(\rho)\in L^2([0,T];L^2(\TT^d))$, we have that
\begin{equation}\label{assume_f1}\lim_{M\rightarrow\infty}\left(\sup\nolimits_{\xi\in[M,(M+1)\wedge \rho]}\abs{\sigma(\xi)}\mathbf{1}_{\{\rho>M\}}\right)=0\;\;\textrm{strongly in}\;\;L^2(\TT^d\times[0,T]),\end{equation}
which follows from \eqref{assume_f6} and an application of Chebyshev's inequality. Assumption \eqref{assume_f6} could be replaced by the somewhat more general condition \eqref{assume_f1}  with no change to the arguments.
\end{enumerate}\end{remark}

\begin{remark}
 Assumption~\eqref{assume_f7} is used in the identical way for $\nu$.  Technically, it is used to guarantee that whenever $\rho\in L^\infty([0,T];L^1(\TT^d))$ with $\nu(\rho)\in L^1([0,T];L^1(\TT^d))$, we have that
\begin{equation}\label{assume_f2}\lim_{M\rightarrow\infty}\left(\sup\nolimits_{\xi\in[M,(M+1)\wedge \rho]}\abs{\nu(\xi)}\mathbf{1}_{\{\rho>M\}}\right)=0\;\;\textrm{strongly in}\;\;L^1(\TT^d\times[0,T]).\end{equation}
In this case the $L^1$-integrability suffices, whereas in the case of $\sigma$ the $L^2$-integrability is used to treat certain stochastic integrals.  Assumptions \eqref{assume_f7} could be replaced by the somewhat more general condition \eqref{assume_f2} with no change to the arguments.\end{remark}

\begin{lem}\label{lem_ibp}  Let $\rho\in H^1(\TT^d)$ be a nonnegative measurable function, and let $\chi=\overline{\chi}(\rho)$ be the kinetic function of $\rho$.  Then, for every $\psi\in\C^\infty_c(\TT^d\times(0,\infty))$,
\[\int_\R\int_{\TT^d}\nabla_x\psi(x,\xi)\chi(x,\xi,r)\dx\dxi  =-\int_{\TT^d}\psi(x,\rho(x))\nabla\rho\dx.\]
In particular, if $\rho$ is a stochastic kinetic solution in the sense of Definition~\ref{def_sol} then, almost surely for every $\psi\in C^\infty_c(\TT^d\times(0,\infty)\times[0,T])$,
\[\int_0^T\int_\R\int_{\TT^d}\chi(x,\xi,s)\nabla_x\psi(x,\xi,s)=-\int_0^T\int_{\TT^d}\nabla\rho (x,s)\psi(x,\rho(x,s),s).\]
\end{lem}

\begin{proof}  The second statement is an immediate consequence of the local regularity property of Definition~\ref{def_sol}, the compact support of $\psi$, and the first statement.  The first statement is a consequence of the $H^1$-regularity of $\rho$, the distributional equality $\nabla_x\chi = \delta_0(\xi-\rho)\nabla\rho$, and an approximation argument, which completes the proof.  \end{proof}

\begin{definition}\label{def_smooth}  For every $\ve,\d\in(0,1)$ let $\kappa^\ve_d\colon\TT^d\rightarrow[0,\infty)$ and $\kappa^\d_1\colon\R\rightarrow[0,\infty)$ be standard convolution kernels of scales $\ve$ and $\d$ on $\TT^d$ and $\R$ respectively, and let $\kappa^{\ve,\d}$ be defined by
\[\kappa^{\ve,\d}(x,y,\xi,\eta)=\kappa^\ve_d(x-y)\kappa^\d_1(\xi-\eta)\;\;\textrm{for every}\;\;(x,y,\xi,\eta)\in(\TT^d)^2\times\R^2.\]
For every $\beta\in(0,1)$ let $\varphi_\beta\colon\R\rightarrow[0,1]$ be the unique nondecreasing piecewise linear function that satisfies
\begin{equation}\label{3_7}\varphi_\beta(\xi)=1\;\;\textrm{if}\;\;\xi\geq \beta,\;\;\varphi_\beta(\xi)=0\;\; \textrm{if}\;\;\xi\leq \nicefrac{\beta}{2},\;\;\textrm{and}\;\;\varphi'_\beta=\frac{2}{\beta}\mathbf{1}_{\{\nicefrac{\beta}{2}<\xi<\beta\}},\end{equation}
and for every $M\in\N$ let $\zeta_M\colon\R\rightarrow [0,1]$ be the unique nonincreasing piecewise linear function that satisfies
\[\zeta_M(\xi)=0\;\; \textrm{if}\;\;\xi\geq M+1,\;\;\zeta_M(\xi)= 1\;\;\textrm{if}\;\;\xi\leq M,\;\;\textrm{and}\;\;\zeta'_M=-\mathbf{1}_{\{M<\xi<M+1\}}.\]
\end{definition}

\begin{prop}\label{prop_measure}  Let $\xi^F$, $\Phi$, $\sigma$, and $\nu$ satisfy Assumptions~\ref{assumption_noise} and \ref{assume_1} and let $\rho_0\in L^1(\O;L^1(\TT^d))$ be nonnegative and $\F_0$-measurable.  Then, if $\rho$ is a stochastic kinetic solution of \eqref{2_0} in the sense of Definition~\ref{def_sol} with initial data $\rho_0$, it follows almost surely that
\[\liminf_{\beta\rightarrow 0}\left(\beta^{-1}q(\TT^d\times [\nicefrac{\beta}{2},\beta]\times[0,T])\right)= 0.\]
\end{prop}

\begin{proof}  For $M\in\N$ and $\beta\in(0,1)$, after approximating $\zeta_M\varphi_\beta$ by smooth functions whose derivatives converge everywhere to the indicator functions $2\beta^{-1}\mathbf{1}_{\{\nicefrac{\beta}{2}\leq \xi\leq \beta\}}$ and $\mathbf{1}_{\{M\leq \xi\leq M+1\}}$ and testing equation \eqref{2_5000} with these approximations, it follows using the dominated convergence theorem that, after passing to the limit with respect to these approximations,
\begin{align}\label{44_1}
& \E\left[2\beta^{-1}q(\TT^d\times[\nicefrac{\beta}{2},\beta]\times[0,T])\right]=\E\left[q(\TT^d\times[M,M+1]\times[0,T])\right]
\\ \nonumber & \quad -\E\left[\left.\int_\R\int_{\TT^d}\chi(x,\xi,s)\varphi_\beta\zeta_M\dx\dxi\right|_{s=0}^{s=T}\right]
\\ \nonumber & \quad  -\frac{1}{2}\E\left[\int_0^T\int_{\TT^d}\mathbf{1}_{\{M<\rho< M+1\}}\sigma(\rho)\sigma'(\rho)\nabla\rho\cdot F_2+\int_0^T\int_{\TT^d}\mathbf{1}_{\{M<\rho<M+1\}}F_3\sigma^2(\rho)\right]
\\ \nonumber & \quad +\beta^{-1}\E\left[\int_0^T\int_{\TT^d}\mathbf{1}_{\{\nicefrac{\beta}{2}<\rho<\beta\}}\sigma(\rho)\sigma'(\rho)\nabla\rho\cdot F_2+\int_0^T\int_{\TT^d}\mathbf{1}_{\{\nicefrac{\beta}{2}<\rho<\beta\}}F_3\sigma^2(\rho)\right],
\end{align}
where it follows from Stampacchia's lemma (see, for example, Evans \cite[Chapter~5, Exercises~17,18]{Eva2010}) and \eqref{def_2500000} that we could equally have taken indicator functions of the form $\mathbf{1}_{\{M\leq \rho\leq M+1\}}$ and $\mathbf{1}_{\{\nicefrac{\beta}{2}\leq\rho\leq\beta\}}$ without changing the resulting integrals.  We first observe by the $L^1$-integrability of $\rho$, the dominated convergence theorem, and Definition~\ref{def_smooth} that
\[\lim_{M\rightarrow\infty}\E\left[\left.\int_\R\int_{\TT^d}\chi(x,\xi,s)\varphi_\beta\zeta_M\dx\dxi\right|_{s=0}^{s=T}\right]=\E\left[\left.\int_\R\int_{\TT^d}\chi(x,\xi,s)\varphi_\beta\dx\dxi\right|_{s=0}^{s=T}\right].\]
After integrating by parts in the third term on the righthand side of \eqref{44_1}, we have almost surely that
\[\int_0^T\int_{\TT^d}\mathbf{1}_{\{M<\rho< M+1\}}\sigma(\rho)\sigma'(\rho)\nabla\rho\cdot F_2=-\frac{1}{2}\int_0^T\int_{\TT^d}\left(\sigma^2((\rho\wedge (M+1))\vee M)-\sigma^2(M)\right)\nabla\cdot F_2.\]
It then follows from the $L^2$-integrability of $\sigma(\rho)$ and \eqref{assume_f1}, the $L^1$-integrability of $\rho$, the boundedness of $\nabla\cdot F_2$ and $F_3$, and the dominated convergence theorem that, almost surely,
\[\lim_{M\rightarrow\infty}\left(\abs{\int_0^T\int_{\TT^d}\mathbf{1}_{\{M<\rho< M+1\}}\sigma(\rho)\sigma'(\rho)\nabla\rho\cdot F_2}+\abs{\frac{1}{2}\int_0^T\int_{\TT^d}\mathbf{1}_{\{M<\rho<M+1\}}F_3\sigma^2(\rho)}\right)=0,\]
from which it follows from the $L^2$-integrability of $\sigma(\rho)$ and the dominated convergence theorem that
\[\lim_{M\rightarrow\infty}\E\left[\int_0^T\int_{\TT^d}\mathbf{1}_{\{M<\rho< M+1\}}\sigma(\rho)\sigma'(\rho)\nabla\rho\cdot F_2+\int_0^T\int_{\TT^d}\mathbf{1}_{\{M<\rho<M+1\}}F_3\sigma^2(\rho)\right]=0.\]
Since it follows by Definition~\ref{def_sol} that
\[\lim_{M\rightarrow\infty}\E\left[q(\TT^d\times[M,M+1]\times[0,T])\right]=0,\]
returning to \eqref{44_1} we have, for every $\beta\in(0,1)$,
\begin{align}\label{44_2}
& \E\left[2\beta^{-1}q(\TT^d\times[\nicefrac{\beta}{2},\beta]\times[0,T])\right]=-\E\left[\left.\int_\R\int_{\TT^d}\chi(x,\xi,s)\varphi_\beta\dx\dxi\right|_{s=0}^{s=T}\right]
\\ \nonumber & \quad +\beta^{-1}\E\left[\int_0^T\int_{\TT^d}\mathbf{1}_{\{\nicefrac{\beta}{2}<\rho<\beta\}}\sigma(\rho)\sigma'(\rho)\nabla\rho\cdot F_2+\int_0^T\int_{\TT^d}\mathbf{1}_{\{\nicefrac{\beta}{2}<\rho<\beta\}}F_3\sigma^2(\rho)\right].
\end{align}
For the second term on the righthand side of \eqref{44_2}, it follows from \eqref{def_2500000} and Stampacchia's lemma (see \cite[Chapter~5, Exercises~17,18]{Eva2010}) that, almost surely,
\begin{align}\label{44_3}
\int_0^T\int_{\TT^d}\mathbf{1}_{\{\nicefrac{\beta}{2}<\rho<\beta\}}\sigma(\rho)\sigma'(\rho)\nabla\rho\cdot F_2 & =\frac{1}{2}\int_0^T\int_{\TT^d}\nabla\left(\sigma^2((\beta\wedge \rho)\vee\nicefrac{\beta}{2})-\sigma^2(\nicefrac{\beta}{2})\right)\cdot F_2
\\ \nonumber & = -\frac{1}{2}\int_0^T\int_{\TT^d}\left(\sigma^2((\beta\wedge \rho)\vee\nicefrac{\beta}{2})-\sigma^2(\nicefrac{\beta}{2})\right)\nabla\cdot F_2.
\end{align}
Assumption~\ref{assume_1} proves either that \eqref{44_3} is zero or it it proves using the boundedness of $\nabla\cdot F_2$, the continuity of $2\sigma\sigma'=(\sigma^2)'$, the fundamental theorem of calculus, and the dominated convergence theorem that, almost surely,
\begin{align*}
& \lim_{\beta\rightarrow 0}\left(\beta^{-1}\int_0^T\int_{\TT^d}\mathbf{1}_{\{\nicefrac{\beta}{2}<\rho<\beta\}}\sigma(\rho)\sigma'(\rho)\nabla\rho\cdot F_2\right)
\\ & = -\left(\lim_{\beta\rightarrow 0}\beta^{-1}\frac{1}{2}\int_0^T\int_{\TT^d}\left(\sigma^2((\beta\wedge \rho)\vee\nicefrac{\beta}{2})-\sigma^2(\nicefrac{\beta}{2})\right)\nabla\cdot F_2\right)
\\ & = -\frac{1}{2}\int_0^T\int_{\TT^d}(\sigma\sigma')(0)\mathbf{1}_{\{\rho>0\}}\nabla\cdot F_2 = 0,
\end{align*}
where the final inequality follows from the assumption that $(\sigma\sigma')(0)=0$.  Since it follows from Assumption~\ref{assume_1}, the boundedness of $F_3$, and the dominated convergence theorem that there exists $c\in(0,\infty)$ such that, almost surely,
\[\lim_{\beta\rightarrow 0}\left(\beta^{-1}\int_0^T\int_{\TT^d}\mathbf{1}_{\{\nicefrac{\beta}{2}<\rho<\beta\}}F_3\sigma^2(\rho)\right)\leq \lim_{\beta\rightarrow 0}\left(c\int_0^T\int_{\TT^d}\mathbf{1}_{\{\nicefrac{\beta}{2}<\rho<\beta\}}\right)=0,\]
it follows from the $L^1$-integrability of $\rho$, property \eqref{def_3535} of Definition~\ref{def_sol}, Definition~\ref{def_smooth}, and the dominated convergence theorem that, almost surely,
\[\lim_{\beta\rightarrow 0}\left.\int_\R\int_{\TT^d}\chi(x,\xi,s)\varphi_\beta\dx\dxi\right|_{s=0}^{s=T}=\norm{\rho(\cdot,T)}_{L^1(\TT^d)}-\norm{\rho_0}_{L^1(\TT^d)}=0.\]
Returning to \eqref{44_2}, we conclude that
\[\lim_{\beta\rightarrow 0}\E\left[2\beta^{-1}q(\TT^d\times[\nicefrac{\beta}{2},\beta]\times[0,T])\right]=0,\]
from which the claim follows by Fatou's lemma.  This completes the proof.  \end{proof}

\begin{thm}\label{thm_unique}  Let $\xi^F$, $\Phi$, $\sigma$, and $\nu$ satisfy Assumptions~\ref{assumption_noise} and \ref{assume_1}, let $\rho^1_0,\rho^2_0\in L^1(\O;L^1(\TT^d))$ be $\F_0$-measurable and let $\rho^1,\rho^2$ be stochastic kinetic solutions of \eqref{2_0} in the sense of Definition~\ref{def_sol} with initial data $\rho_0^1,\rho_0^2$.  Then, almost surely,
\[\sup_{t\in[0,T]}\norm{\rho^1(\cdot,t)-\rho^2(\cdot,t)}_{L^1(\TT^d)}\leq\norm{\rho^1_0-\rho^2_0}_{L^1(\TT^d)}.\]
\end{thm}

\begin{proof}   Let $\chi^1$ and $\chi^2$ be the kinetic functions of $\rho^1$ and $\rho^2$ and for every $\ve,\d\in(0,1)$ and $i\in\{1,2\}$ let $\chi^{\ve,\d}_{t,i}(y,\eta)=(\chi^i(\cdot,\cdot,t)*\kappa^{\ve,\d})(y,\eta)$ for the convolution kernel $\kappa^{\ve,\d}$ defined in Definition~\ref{def_smooth}.  It follows from Definition~\ref{def_sol} and the Kolmogorov continuity criterion (see, for example, Revuz and Yor \cite[Chapter~1, Theorem~2.1]{RevYor1999}) that for every $\ve,\d\in(0,1)$ there exists a subset of full probability such that, for every $i\in\{1,2\}$, $(y,\eta)\in\TT^d\times(\nicefrac{\d}{2},\infty)$, and $t\in[0,T]$,
\begin{align}\label{3_6}
& \left.\chi^{\ve,\d}_{s,i}(y,\eta)\right|_{s=0}^t=\nabla_y\cdot \left(\int_0^t\int_{\TT^d}\Phi'(\rho^i)\nabla\rho^i \kappa^{\ve,\d}(x,y,\rho^i,\eta)\right)
\\ \nonumber & \quad +\nabla_y\cdot\left(\frac{1}{2}\int_0^t\int_{\TT^d}\left(F_1(x)[\sigma'(\rho^i)]^2\nabla\rho^i+\sigma(\rho^i)\sigma'(\rho^i)F_2(x)\right)\kappa^{\ve,\d}(x,y,\rho^i,\eta)\right)
\\ \nonumber & \quad + \partial_\eta\left(\int_0^t\int_\R\int_{\TT^d}\kappa^{\ve,\d}(x,y,\xi,\eta)\dd q^i\right)
\\ \nonumber & \quad -\partial_\eta\left(\frac{1}{2}\int_0^t\int_{\TT^d}\left(F_3(x)\sigma^2(\rho^i)+\sigma(\rho^i)\sigma'(\rho^i)\nabla\rho^i\cdot F_2(x)\right)\kappa^{\ve,\d}(x,y,\rho^i,\eta) \right)
\\ \nonumber & \quad -\int_0^t\int_{\TT^d}\kappa^{\ve,\d}(x,y,\rho^i,\eta)\nabla\cdot\nu(\rho) - \int_0^t\int_{\TT^d}\kappa^{\ve,\d}(x,y,\rho^i,\eta)\nabla\cdot\left(\sigma(\rho^i)\cdot\dd\xi^F\right).
\end{align}
We will first treat the analogues of the first two terms on the righthand side of \eqref{3_000}.  For the cutoff functions defined in Definition~\ref{def_smooth}, it follows almost surely from \eqref{3_7} and \eqref{3_6} that, for every $\ve,\beta\in(0,1)$, $M\in\N$, and $\d\in(0,\nicefrac{\beta}{4})$, for every $t\in[0,T]$ and $i\in\{1,2\}$,
\begin{equation}\label{3_8}
\left.\int_\R\int_{\TT^d}\chi^{\ve,\d}_{s,i}(y,\eta)\varphi_\beta(\eta)\zeta_M(\eta)\dy\deta\right|^t_{s=0} = I^{i,\textrm{cut}}_t+I^{i,\textrm{mart}}_t+I^{i,\textrm{cons}}_t
\end{equation}
for the cutoff term defined by
\begin{align*}
& I^{i,\textrm{cut}}_t =-\int_0^t\int_{\R^2}\int_{(\TT^d)^2}\kappa^{\ve,\d}(x,y,\xi,\eta)\partial_\eta(\varphi_\beta(\eta)\zeta_M(\eta))\dd q^i(x,\xi,s)
 \\  & +\frac{1}{2}\int_0^t\int_\R\int_{(\TT^d)^2}\left(F_3(x)\sigma^2(\rho^i(x,s))+\sigma(\rho^i)\sigma'(\rho^i)\nabla\rho^i\cdot F_2(x)\right)\kappa^{\ve,\d}(x,y,\rho^i(x,s),\eta)\partial_\eta(\varphi_\beta(\eta)\zeta_M(\eta)),
 \end{align*}
for the martingale term defined by
\[I^{i,\textrm{mart}}_t = -\int_0^t\int_\R\int_{(\TT^d)^2}\kappa^{\ve,\d}(x,y,\rho^i(x,s),\eta)\nabla\cdot(\sigma(\rho^i(x,s))\dd\xi^F)\varphi_\beta(\eta)\zeta_M(\eta)\dx\dy\deta,\]
and for the conservative term defined by
\[I^{i,\textrm{cons}}_t=-\int_0^t\int_{\TT^d}\kappa^{\ve,\d}(x,y,\rho^i,\eta)\nabla\cdot\nu(\rho)\dt,\]
where we emphasize that the terms $I^{i,\textrm{cut}}_t$, $I^{i,\textrm{mart}}_t$, and $I^{i,\textrm{cons}}_t$ depend on $\ve,\d,\beta\in(0,1)$ and $M\in\N$.

We will now treat the analogue of the mixed term on the righthand side of \eqref{3_000}.  In the following, we will write $(x,\xi)\in\TT^d\times\R$ for the arguments of $\chi^1$ and all related quantities, and we will write $(x',\xi')\in\TT^d\times\R$ for the arguments of $\chi^2$ and all related quantities.  Let
\[\overline{k}^{\ve,\d}_{s,1}(x,y,\eta)=\kappa^{\ve,\d}(x,y,\rho^1(x,s),\eta)\;\;\textrm{and}\;\;\overline{k}^{\ve,\d}_{s,2}(x',y,\eta)=\kappa^{\ve,\d}(x',y,\rho^2(x',s),\eta).\]
From \eqref{3_7}, \eqref{3_6}, and the stochastic product rule we have almost surely that, for every $\ve,\beta\in(0,1)$, $M\in\N$, and $\d\in(0,\nicefrac{\beta}{4})$, for every $t\in[0,T]$,
\begin{align}\label{3_09}
& \left.\int_\R\int_{\TT^d}\chi^{\ve,\d}_{s,1}(y,\eta)\chi^{\ve,\d}_{s,2}(y,\eta)\varphi_\beta(\eta)\zeta_M(\eta)\dy\deta\right|_{s=0}^t
\\ \nonumber & = \int_0^t\int_\R\int_{\TT^d}\left(\chi^{\ve,\d}_{s,2}(y,\eta)\dd \chi^{\ve,\d}_{s,1}(y,\eta)+\chi^{\ve,\d}_{s,1}(y,\eta)\dd \chi^{\ve,\d}_{s,2}(y,\eta)+\dd\langle \chi^{\ve,\d}_2,\chi^{\ve,\d}_1\rangle_s(y,\eta)\right)\varphi_\beta(\eta)\zeta_M(\eta)\dy\deta.
\end{align}
It follows from Lemma~\ref{lem_ibp}, \eqref{3_6}, the definition of $\varphi_\beta$, $\d\in(0,\nicefrac{\beta}{4})$, and the distributional equalities involving the kinetic function that
\begin{equation}\label{3_0009}
\int_0^t\int_\R\int_{\TT^d}\chi^{\ve,\d}_{s,2}(y,\eta)\dd \chi^{\ve,\d}_{s,1}(y,\eta)\varphi_\beta(\eta)\zeta_M(\eta)\dy\deta\ds = I^{2,1,\textrm{err}}_t+I^{2,1,\textrm{meas}}_t+I^{2,1,\textrm{cut}}_t+I^{2,1,\textrm{mart}}_t+I^{2,1,\textrm{cons}}_t\end{equation}
where, after adding the second term of \eqref{3_000009} below and subtracting it in \eqref{3_0000009} below, the error term is
\begin{align}\label{3_000009}
I^{2,1,\textrm{err}}_t & = -\int_0^t\int_{\R}\int_{(\TT^d)^3}\Phi'(\rho^1)\nabla\rho^1\cdot\nabla\rho^2\overline{\kappa}^{\ve,\d}_{s,1}\overline{\kappa}^{\ve,\d}_{s,2}\varphi_\beta(\eta)\zeta_M(\eta)
\\ \nonumber & \quad +\int_0^t\int_{\R}\int_{(\TT^d)^3}[\Phi'(\rho^1)]^\frac{1}{2}[\Phi'(\rho^2)]^\frac{1}{2}\nabla\rho^1\cdot\nabla\rho^2\overline{\kappa}^{\ve,\d}_{s,1}\overline{\kappa}^{\ve,\d}_{s,2}\varphi_\beta(\eta)\zeta_M(\eta)
\\ \nonumber & \quad - \frac{1}{2}\int_0^t\int_\R\int_{(\TT^d)^3}\left(F_1(x)[\sigma'(\rho^1)]^2\nabla\rho^1\cdot\nabla\rho^2+\sigma(\rho^1)\sigma'(\rho^1)\nabla\rho^2\cdot F_2(x)\right)\overline{\kappa}^{\ve,\d}_{s,1}\overline{\kappa}^{\ve,\d}_{s,2}\varphi_\beta(\eta)\zeta_M(\eta)
\\ \nonumber & \quad - \frac{1}{2}\int_0^t\int_\R\int_{(\TT^d)^3}\left(F_3(x)\sigma^2(\rho^1)+\sigma(\rho^1)\sigma'(\rho^1)\nabla\rho^1\cdot F_2(x)\right)\overline{\kappa}^{\ve,\d}_{s,1}\overline{\kappa}^{\ve,\d}_{s,2}\varphi_\beta(\eta)\zeta_M(\eta),
\end{align}
the measure term is
\begin{align}\label{3_0000009}
& I^{2,1,\textrm{meas}}_t  = \int_0^t\int_{\R^2}\int_{(\TT^d)^3}\kappa^{\ve,\d}(x,y,\xi,\eta)\overline{\kappa}^{\ve,\d}_{s,2}\varphi_\beta(\eta)\zeta_M(\eta)\dd q^1(x,\xi,s)\dxp\dy\deta
\\ \nonumber & \quad - \int_0^t\int_{\R}\int_{(\TT^d)^3}[\Phi'(\rho^1)]^\frac{1}{2}[\Phi'(\rho^2)]^\frac{1}{2}\nabla\rho^1\cdot\nabla\rho^2\overline{\kappa}^{\ve,\d}_{s,1}\overline{\kappa}^{\ve,\d}_{s,2}\varphi_\beta(\eta)\zeta_M(\eta)\dx\dxp\dy\deta\ds,
\end{align}
the cutoff term is
\begin{align*}
& I^{2,1,\textrm{cut}}_t  = -\int_0^t\int_{\R^2}\int_{(\TT^d)^2}\kappa^{\ve,\d}(x,y,\xi,\eta)\chi^{\ve,\d}_{s,2}(y,\eta)\partial_\eta(\varphi_\beta(\eta)\zeta_M(\eta))\dd q^1(x,\xi,s)\dy\deta\ds
\\  & \quad +\frac{1}{2}\int_0^t\int_{\R}\int_{(\TT^d)^2}\left(F_3(x)\sigma^2(\rho^1)+\sigma(\rho^1)\sigma'(\rho^1)\nabla\rho^1\cdot F_2(x)\right)\overline{\kappa}^{\ve,\d}_{s,1}\chi^{\ve,\d}_{s,2}\partial_\eta(\varphi_\beta(\eta)\zeta_M(\eta))\dx\dy\deta\ds,
\end{align*}
the martingale term is
\[I^{2,1,\textrm{mart}}_t=-\int_0^t\int_\R\int_{(\TT^d)^2}\overline{\kappa}^{\ve,\d}_{s,1}\chi^{\ve,\d}_{s,2}\varphi_\beta(\eta)\zeta_M(\eta)\nabla\cdot(\sigma(\rho^1)\dd\xi^F(x))\dx\dy\deta,\]
and the conservative term is
\[I^{2,1,\textrm{cons}}_t = -\int_0^t\int_\R\int_{(\TT^d)^2}\overline{\kappa}^{\ve,\d}_{s,1}\chi^{\ve,\d}_{s,2}\varphi_\beta(\eta)\zeta_M(\eta)\nabla\cdot\nu(\rho^1)\dx\dy\deta\ds.\]
The analogous formula holds for the second term on the righthand side of \eqref{3_09}, with the decomposition $I^{1,2,\textrm{err}}_t$ and similarly for the remaining four parts.  For the final term of \eqref{3_09}, it follows from \eqref{3_6} and the definition of $\xi^F$ that
\begin{align}\label{3_009}
& \int_0^t\int_\R\int_{\TT^d}\dd\langle\chi^{\ve,\d}_1,\chi^{\ve,\d}_{s,2}\rangle_s(y,\eta)\varphi_\beta(\eta)\zeta_M(\eta)\dy\deta\ds
\\ \nonumber & = \sum_{k=1}^\infty\int_0^t\int_\R\int_{(\TT^d)^3}f_k(x)f_k(x')\sigma'(\rho^1)\sigma'(\rho^2)\nabla\rho^1\cdot\nabla\rho^2\overline{\kappa}^{\ve,\d}_{s,1}\overline{\kappa}^{\ve,\d}_{s,2}\varphi_\beta(\eta)\zeta_M(\eta)\dx\dxp\dy\deta\ds
\\ \nonumber & \quad + \sum_{k=1}^\infty\int_0^t\int_\R\int_{(\TT^d)^3}\nabla f_k(x)\cdot\nabla f_k(x')\sigma(\rho^1)\sigma(\rho^2)\overline{\kappa}^{\ve,\d}_{s,1}\overline{\kappa}^{\ve,\d}_{s,2}\varphi_\beta(\eta)\zeta_M(\eta)\dx\dxp\dy\deta\ds
\\ \nonumber &  \quad + \sum_{k=1}^\infty \int_0^t\int_\R\int_{(\TT^d)^3}\sigma'(\rho^1)\sigma(\rho^2)f_k(x)\nabla f_k(x')\cdot\nabla\rho^1\overline{\kappa}^{\ve,\d}_{s,1}\overline{\kappa}^{\ve,\d}_{s,2}\varphi_\beta(\eta)\zeta_M(\eta)\dx\dxp\dy\deta\ds
\\ \nonumber &  \quad + \sum_{k=1}^\infty \int_0^t\int_\R\int_{(\TT^d)^3}\sigma(\rho^1)\sigma'(\rho^2)f_k(x')\nabla f_k(x)\cdot\nabla\rho^2\overline{\kappa}^{\ve,\d}_{s,1}\overline{\kappa}^{\ve,\d}_{s,2}\varphi_\beta(\eta)\zeta_M(\eta)\dx\dxp\dy\deta\ds.
\end{align}
It follows from \eqref{3_09}, \eqref{3_0009}, and \eqref{3_009} that
\[\left.\int_\R\int_{\TT^d}\chi^{\ve,\d}_{s,1}(y,\eta)\chi^{\ve,\d}_{s,2}(y,\eta)\varphi_\beta(\eta)\zeta_M(\eta)\dy\deta\right|^t_{s=0} = I^{\textrm{err}}_t+I^{\textrm{meas}}_t+I^{\textrm{mix},\textrm{cut}}_t+I^{\textrm{mix},\textrm{mart}}_t+I^{\textrm{mix},\textrm{cons}}_t,\]
where the error terms \eqref{3_000009} and \eqref{3_009} combine to form, using Einstein's summation convention over repeated indices,\small
\begin{align*}
& I^{\textrm{err}}_t = -\int_0^t\int_\R\int_{(\TT^d)^3}\left([\Phi'(\rho^1)]^\frac{1}{2}-[\Phi'(\rho^2)]^\frac{1}{2}\right)^2\nabla\rho^1\cdot\nabla\rho^2\overline{\kappa}^{\ve,\d}_{s,1}\overline{\kappa}^{\ve,\d}_{s,2}\varphi_\beta\zeta_M
\\ \nonumber & -\frac{1}{2} \int_0^t\int_\R\int_{(\TT^d)^3}(F_1(x)[\sigma'(\rho^1)]^2+F_1(x')[\sigma'(\rho^2)]^2-2f_k(x)f_k(x')\sigma'(\rho^1)\sigma'(\rho^2))\overline{\kappa}^{\ve,\d}_{s,1}\overline{\kappa}^{\ve,\d}_{s,2}\varphi_\beta\zeta_M
\\ \nonumber & -\frac{1}{2} \int_0^t\int_\R\int_{(\TT^d)^3}(F_3(x)\sigma^2(\rho^1)+F_3(x')\sigma^2(\rho^2)-2\nabla f_k(x)\cdot\nabla f_k(x')\sigma(\rho^1)\sigma(\rho^2))\overline{\kappa}^{\ve,\d}_{s,1}\overline{\kappa}^{\ve,\d}_{s,2}\varphi_\beta\zeta_M
\\ \nonumber & -\frac{1}{2} \int_0^t\int_\R\int_{(\TT^d)^3}(\sigma(\rho^1)\sigma'(\rho^1)F_2(x)+\sigma(\rho^2)\sigma'(\rho^2)F_2(x')-2\sigma'(\rho^1)\sigma(\rho^2)f_k(x)\nabla f_k(x'))\cdot\nabla\rho^1\overline{\kappa}^{\ve,\d}_{s,1}\overline{\kappa}^{\ve,\d}_{s,2}\varphi_\beta\zeta_M
\\ \nonumber & -\frac{1}{2} \int_0^t\int_\R\int_{(\TT^d)^3}(\sigma(\rho^1)\sigma'(\rho^1)F_2(x)+\sigma(\rho^2)\sigma'(\rho^2)F_2(x')-2\sigma(\rho^1)\sigma'(\rho^2)f_k(x')\nabla f_k(x))\cdot\nabla\rho^2\overline{\kappa}^{\ve,\d}_{s,1}\overline{\kappa}^{\ve,\d}_{s,2}\varphi_\beta\zeta_M,
\end{align*}
\normalsize and where the measure terms \eqref{3_0000009} combine to form
\begin{align*}
I^{\textrm{meas}}_t & =  \int_0^t\int_{\R^2}\int_{(\TT^d)^3}\kappa^{\ve,\d}(x,y,\xi,\eta)\overline{\kappa}^{\ve,\d}_{s,2}\varphi_\beta(\eta)\zeta_M(\eta)\dd q^1(x,\xi,s)\dxp\dy\deta
\\ \nonumber & \quad +\int_0^t\int_{\R^2}\int_{(\TT^d)^3}\kappa^{\ve,\d}(x',y,\xi',\eta)\overline{\kappa}^{\ve,\d}_{s,1}\varphi_\beta(\eta)\zeta_M(\eta)\dd q^2(x',\xi',s)\dx\dy\deta
\\ \nonumber & \quad -2\int_0^t\int_{\R}\int_{(\TT^d)^3}[\Phi'(\rho^1)]^\frac{1}{2}[\Phi'(\rho^2)]^\frac{1}{2}\nabla\rho^1\cdot\nabla\rho^2\overline{\kappa}^{\ve,\d}_{s,1}\overline{\kappa}^{\ve,\d}_{s,2}\varphi_\beta(\eta)\zeta_M(\eta)\dx\dxp\dy\deta\ds.
\end{align*}
For the cutoff, martingale, and conservative terms defined respectively by
\[I^{\textrm{cut,mart,cons}}_t= I^{1,\textrm{cut,mart,cons}}_t+I^{2,\textrm{cut,mart,cons}}_t-2(I^{2,1,\textrm{cut,mart,cons}}_t+I^{1,2,\textrm{cut,mart,cons}}_t),\]
we have from \eqref{3_8} and \eqref{3_0009} that, almost surely for every $t\in[0,T]$,
\begin{equation}\label{3_21}\left.\int_\R\int_{\TT^d}\left(\chi^{\ve,\d}_{s,1}+\chi^{\ve,\d}_{s,2}-2\chi^{\ve,\d}_{s,1}\chi^{\ve,\d}_{s,2}\right)\varphi_\beta\zeta_M\right|_{s=0}^t=-2I^{\textrm{err}}_t-2I^{\textrm{meas}}_t+I^{\textrm{mart}}_t+I^{\textrm{cut}}_t+I^{\textrm{cons}}_t.\end{equation}
We will handle the five terms on the righthand side of \eqref{3_21} separately.

\textbf{The measure term}.  It follows from property \eqref{2_500} of the kinetic measure and H\"older's inequality that the measure term almost surely satisfies, for every $t\in[0,T]$,
\begin{equation}\label{3_22} I^{\textrm{meas}}_t\geq 0.\end{equation}

\textbf{The error term}.  For the error term, it follows from $\d\in(0,\nicefrac{\beta}{4})$, the definition of the convolution kernel, the definitions of $\varphi_\beta$ and $\zeta_M$, the local $H^1$-regularity of the solutions $\rho^i$, the continuity of the $F_i$, and the dominated convergence theorem that, for $\overline{\kappa}^\d_{i,s}(y,\eta)=\kappa^\d_1(\rho^i(y,s)-\eta)$ for each $i\in\{1,2\}$,
\begin{align}\label{3_0010000}
& \lim_{\ve\rightarrow 0} I^{\textrm{err}}_t = -\int_0^t\int_\R\int_{\TT^d}\left([\Phi'(\rho^1)]^\frac{1}{2}-[\Phi'(\rho^2)]^\frac{1}{2}\right)^2\nabla\rho^1\cdot\nabla\rho^2\overline{\kappa}^{\d}_{s,1}\overline{\kappa}^{\d}_{s,2}\varphi_\beta\zeta_M
\\ \nonumber & -\frac{1}{2} \int_0^t\int_\R\int_{\TT^d}\left(F_1(y)\left(\sigma'(\rho^1)-\sigma'(\rho^2)\right)^2+F_3(y)\left(\sigma(\rho^1)-\sigma(\rho^2)\right)^2\right)\overline{\kappa}^{\d}_{s,1}\overline{\kappa}^{\d}_{s,2}\varphi_\beta\zeta_M
\\ \nonumber & -\frac{1}{2} \int_0^t\int_\R\int_{\TT^d}(\sigma(\rho^1)\sigma'(\rho^1)+\sigma(\rho^2)\sigma'(\rho^2)-2\sigma'(\rho^1)\sigma(\rho^2))F_2(y)\cdot\nabla\rho^1\overline{\kappa}^{\ve,\d}_{s,1}\overline{\kappa}^{\ve,\d}_{s,2}\varphi_\beta\zeta_M
\\ \nonumber & -\frac{1}{2} \int_0^t\int_\R\int_{\TT^d}(\sigma(\rho^1)\sigma'(\rho^1)+\sigma(\rho^2)\sigma'(\rho^2)-2\sigma(\rho^1)\sigma'(\rho^2))F_2(y)\cdot\nabla\rho^2\overline{\kappa}^{\ve,\d}_{s,1}\overline{\kappa}^{\ve,\d}_{s,2}\varphi_\beta\zeta_M.
\end{align}
It follows from Assumption~\ref{assume_1} and the definition of the convolution kernel that there exists $c\in(0,\infty)$ depending on $\beta$ and $M$ such that, for every $\d\in(0,\nicefrac{\beta}{4})$, whenever $\overline{\kappa}^{\ve,\d}_{s,1}\overline{\kappa}^{\ve,\d}_{s,2}\varphi_\beta\zeta_M\neq 0$,
\begin{align*}
& \left([\Phi'(\rho^1)]^\frac{1}{2}-[\Phi'(\rho^2)]^\frac{1}{2}\right)^2+\left([\sigma'(\rho^1)]^\frac{1}{2}-[\sigma'(\rho^2)]^\frac{1}{2}\right)^2+\left(\sigma(\rho^1)-\sigma(\rho^2)\right)^2
\\ & \quad +\abs{\sigma(\rho^1)\sigma'(\rho^1)+\sigma(\rho^2)\sigma'(\rho^2)-2\sigma'(\rho^1)\sigma(\rho^2)}
\\ & \quad \quad +\abs{\sigma(\rho^1)\sigma'(\rho^1)+\sigma(\rho^2)\sigma'(\rho^2)-2\sigma(\rho^1)\sigma'(\rho^2)} \leq c\mathbf{1}_{\{0<\abs{\rho^1(x,s)-\rho^2(x',s)}<c\d\}}\d.
\end{align*}
Here we are using the fact that the local Lipschitz regularity of $\Phi'$, that $\Phi'>0$ on $(0,\infty)$, and $\d\in(0,\nicefrac{\beta}{4})$ implies that the square root is $\nicefrac{1}{2}$-H\"older continuous on the support of $\varphi_\beta\zeta_M$.  The final two terms are bounded using the triangle inequality, the local boundedness of $\sigma$ and $\sigma'$, and the local Lipschitz regularity of $\sigma$ and $\sigma'$.  Returning to \eqref{3_0010000}, it follows from the boundedness of the $F_i$, H\"older's inequality, and Young's inequality that there exists $c\in(0,\infty)$ depending on $\beta$ and $M$ such that, almost surely for every $t\in[0,T]$,
\begin{equation}\label{3_11}
\limsup_{\ve\rightarrow 0}\abs{I^{\textrm{err}}_t} \leq c\int_0^T\int_\R\int_{\TT^d}\mathbf{1}_{\{0<\abs{\rho^1(y,s)-\rho^2(y,s)}<c\d\}}(1+\abs{\nabla\rho^1}^2+\abs{\nabla\rho^2}^2)(\d\overline{\kappa}^\d_{s,1})\overline{\kappa}^\d_{s,2}\varphi_\beta\zeta_M.
\end{equation}
It follows from the uniform boundedness of $(\d\overline{\kappa}^\d_{s,1})$ in $\delta\in(0,\nicefrac{\beta}{4})$, the definitions of $\overline{\kappa}^\d_{s,2}$, $\varphi_\beta$, and $\zeta_M$, property \eqref{def_2500000} of the solutions, the dominated convergence theorem, and \eqref{3_11} that, almost surely for every $t\in[0,T]$,
\begin{equation}\label{3_12}\limsup_{\d\rightarrow 0}\left(\limsup_{\ve\rightarrow 0}\abs{I^{\textrm{err}}_t}\right)=0.\end{equation}

\textbf{The martingale term}.  For the analysis of the martingale terms we will repeatedly use the fact that if $F^\ve_t,F\in L^2(\O\times[0,T])$ are $\F_t$-progressively measurable processes that satisfy $F^\ve_t\rightarrow F$ in $L^2(\O\times[0,T])$ as $\ve\rightarrow 0$ then, after passing to a subsequence, the It\^o integrals satisfy, almost surely for every $t\in[0,T]$,
\begin{equation}\label{lem_dct} \lim_{\ve\rightarrow 0}\int_0^tF^\ve_s\dd B_s = \int_0^t F_s\dd B_s.\end{equation}
The proof is a consequence of the Burkholder-Davis-Gundy inequality (see, for example, \cite[Chapter~4, Theorem~4.1]{RevYor1999}).  It follows from property \eqref{def_2500000} of the solutions, the definition of $\kappa^{\ve,\d}$, the boundedness of the kinetic functions, the definition of $\xi^F$, and \eqref{lem_dct} that, after passing to a subsequence $\ve\rightarrow 0$, almost surely for every $t\in[0,T]$,
\begin{align}\label{3_23}
\lim_{\ve\rightarrow 0} I^{\textrm{mart}}_t & = \int_0^t\int_\R\int_{\TT^d}\overline{\kappa}^\d_{s,1}(2\chi^\d_{s,2}-1)\varphi_\beta(\eta)\zeta_M(\eta)\nabla\cdot(\sigma(\rho^1)\dd\xi^F)\dy\deta
\\ \nonumber & \quad + \int_0^t\int_\R\int_{\TT^d}\overline{\kappa}^\d_{s,2}(2\chi^\d_{s,1}-1)\varphi_\beta(\eta)\zeta_M(\eta)\nabla\cdot(\sigma(\rho^2)\dd\xi^F)\dy\deta,
\end{align}
for $\chi^\d_{s,i}(y,\eta)=(\chi^i_s(y,\cdot)*\kappa^\d_1)(\eta)$.  It follows from $\varphi_\beta,\zeta_M\in\C^\infty_c((0,\infty))$, the definition of $\overline{\kappa}^\d_{s,1}$, the boundedness of the kinetic function, property \eqref{def_2500000} of the solutions, the definition of $\xi^F$, and the Burkholder-Davis-Gundy inequality (see, for example, \cite[Chapter~4, Theorem~4.1]{RevYor1999}) that there exists $c\in(0,\infty)$ such that
\begin{align*}
& \E\left[\sup_{t\in[0,T]}\abs{\int_0^t\int_\R\int_{\TT^d}\overline{\kappa}^\d_{s,1}(2\chi^\d_{s,2}-1)(\varphi_\beta(\eta)\zeta_M(\eta)-\varphi_\beta(\rho^1)\zeta_M(\rho^1))\nabla\cdot(\sigma(\rho^1)\dd\xi^F)\dy\deta}\right]
\\ & \leq c\E\left[\int_0^T\left(\int_\R\int_{\TT^d}\overline{\kappa}^\d_{s,1}(\varphi_\beta(\eta)\zeta_M(\eta)-\varphi_\beta(\rho^1)\zeta_M(\rho^1))\left(\sigma'(\rho^1)\abs{\nabla\rho^1}+\abs{\sigma(\rho^1)}\right)\dy\deta \right)^2\ds\right]^\frac{1}{2}
\\ & \leq c\d\E\left[\int_0^T\int_{\TT^d}\left([\sigma'(\rho^1)]^2\abs{\nabla\rho^1}^2+\abs{\sigma(\rho^1)}^2\right)\mathbf{1}_{\{\nicefrac{\beta}{2}-\d<\rho^1<M+1+\d\}}\dy\ds\right]^\frac{1}{2}.
\end{align*}
It follows from the local regularity of $\rho^1$ and the dominated convergence theorem that, after passing to a subsequence $\d\rightarrow 0$, almost surely for every $t\in[0,T]$,
\begin{equation}\label{3_24}\lim_{\d\rightarrow 0}\abs{\int_0^t\int_\R\int_{\TT^d}\overline{\kappa}^\d_{s,1}(2\chi^\d_{s,2}-1)(\varphi_\beta(\eta)\zeta_M(\eta)-\varphi_\beta(\rho^1)\zeta_M(\rho^1))\nabla\cdot(\sigma(\rho^1)\dd\xi^F)\dy\deta}=0,\end{equation}
and similarly for the symmetric term coming from \eqref{3_23}. An explicit calculation proves that the convolution $\abs{(\kappa^\d*(\kappa^\d*\chi^2_s))}\leq 1$ and that, whenever $2\d<\rho^2(y,s)$,
\begin{equation}\label{3_25}
\int_{\R^2}\kappa^\d(\xi-\eta)\kappa^\d(\eta-\xi')\chi^2_s(y,\xi')\deta\dxip =\left\{\begin{aligned} & 0 && \textrm{if}\;\;\xi\leq -2\d\;\;\textrm{or}\;\;\xi\geq \rho^2(y,s)+2\d, \\ & \nicefrac{1}{2} && \textrm{if}\;\;\xi = 0\;\;\textrm{or}\;\;\xi = \rho^2(y,s), \\ & 1 && \textrm{if}\;\;2\d<\xi<\rho^2(y,s)-2\d.\end{aligned}\right.
\end{equation}
It follows from \eqref{3_25} and $\varphi_\beta(0)=0$ that pointwise
\[\lim_{\d\rightarrow 0}\left(\int_\R\overline{\kappa}^\d_{s,1}(2\chi^\d_{s,2}-1)\deta\right)\varphi_\beta(\rho^1) =  \left(\mathbf{1}_{\{\rho^1=\rho^2\}}+2\mathbf{1}_{\{0\leq \rho^1<\rho^2\}}-1\right)\varphi_\beta(\rho^1).\]
In combination \eqref{3_23}, \eqref{3_24}, and \eqref{3_25} prove with \eqref{def_2500000}, the definition of $\xi^F$, the fact that $\varphi_\beta(0)=0$, and \eqref{lem_dct} that, after passing to a subsequence $\d\rightarrow 0$, almost surely for every $t\in[0,T]$,
\begin{align*}
\lim_{\d\rightarrow 0}\left(\lim_{\ve\rightarrow 0} I^\textrm{mart}_t\right) & =  \int_0^t\int_{\TT^d}\left(\mathbf{1}_{\{\rho^1=\rho^2\}}+2\mathbf{1}_{\{\rho^1<\rho^2\}}-1\right)\varphi_\beta(\rho^1)\zeta_M(\rho^1)\nabla\cdot(\sigma(\rho^1)\dd\xi^F)\dy
\\ \nonumber & \quad + \int_0^t\int_{\TT^d}\left(\mathbf{1}_{\{\rho^1=\rho^2\}}+2\mathbf{1}_{\{\rho^2<\rho^1\}}-1\right)\varphi_\beta(\rho^2)\zeta_M(\rho^2)\nabla\cdot(\sigma(\rho^2)\dd\xi^F)\dy.
\end{align*}
Since $\mathbf{1}_{\{\rho^2<\rho^1\}}=1-\mathbf{1}_{\{\rho^1=\rho^2\}}-\mathbf{1}_{\{\rho^1<\rho^2\}}$ and since $\sgn(\rho^2-\rho^1)=\mathbf{1}_{\{\rho^1=\rho^2\}}+2\mathbf{1}_{\{\rho^1<\rho^2\}}-1$, along subsequences $\ve,\d\rightarrow 0$, almost surely for every $t\in[0,T]$,
\begin{equation}\label{3_987}
\lim_{\d,\ve\rightarrow 0}\left(I^\textrm{mart}_t\right)  =  \int_0^t\int_{\TT^d}\sgn(\rho^2-\rho^1)\left(\varphi_\beta(\rho^1)\zeta_M(\rho^1)\nabla\cdot(\sigma(\rho^1)\dd\xi^F)-\varphi_\beta(\rho^2)\zeta_M(\rho^2)\nabla\cdot(\sigma(\rho^2)\dd\xi^F)\right),
\end{equation}
where we observe that there is no ambiguity in interpreting the value of $\sgn(0)$ since by Stampacchia's lemma (see \cite[Chapter~5, Exercises~17,18]{Eva2010}) we have that
\[\int_0^T\int_{\TT^d}\mathbf{1}_{\{\rho^1=\rho^2\}}\left(\varphi_\beta(\rho^1)\zeta_M(\rho^1)\nabla\cdot(\sigma(\rho^1)\dd\xi^F)-\varphi_\beta(\rho^2)\zeta_M(\rho^2)\nabla\cdot(\sigma(\rho^2)\dd\xi^F)\right)\dy=0.\]
For every $\beta\in(0,1)$ and $M\in\N$ let $\Theta_{\beta,M}\colon[0,\infty)\rightarrow[0,\infty)$ be the unique function that satisfies $\Theta_{\beta,M}(0)=0$ and $\Theta_{\beta,M}'(\xi)=\varphi_\beta(\xi)\zeta_M(\xi)\sigma'(\xi)$.  Returning to \eqref{3_987}, it follows that, along subsequences,
\begin{align}\label{3_988}
\lim_{\d,\ve\rightarrow 0}\left(I^\textrm{mart}_t\right)  & =  \int_0^t\int_{\TT^d}\sgn(\rho^2-\rho^1)\nabla\cdot\left(\left(\Theta_{\beta,M}(\rho^1)-\Theta_{\beta,M}(\rho^2)\right)\dd\xi^F\right)
\\ \nonumber &  \quad +\int_0^t\int_{\TT^d}\sgn(\rho^2-\rho^1)\left(\varphi_\beta(\rho^1)\zeta_M(\rho^1)\sigma(\rho^1)-\Theta_{\beta,M}(\rho^1)\right)\nabla\cdot\dd\xi^F
\\ \nonumber & \quad - \int_0^t\int_{\TT^d}\sgn(\rho^2-\rho^1)\left(\varphi_\beta(\rho^2)\zeta_M(\rho^2)\sigma(\rho^2)-\Theta_{\beta,M}(\rho^2)\right)\nabla\cdot\dd\xi^F.
\end{align}
For the first term on the righthand side of \eqref{3_988}, for $\sgn^\d=(\sgn*\kappa^\d_1)$ for every $\d\in(0,1)$, it follows from \eqref{lem_dct} that, along a subsequence $\d\rightarrow 0$, almost surely for every $t\in[0,T]$,
\begin{align}\label{3_98900}
& \int_0^t\int_{\TT^d}\sgn(\rho^2-\rho^1)\nabla\cdot\left(\left(\Theta_{\beta,M}(\rho^1)-\Theta_{\beta,M}(\rho^2)\right)\dd\xi^F\right)
 \\ \nonumber &  = \lim_{\d\rightarrow 0} \int_0^t\int_{\TT^d}\sgn^\d(\rho^1-\rho^2)\nabla\cdot\left(\left(\Theta_{\beta,M}(\rho^1)-\Theta_{\beta,M}(\rho^2)\right)\dd\xi^F\right)
\\ \nonumber & = -\lim_{\d\rightarrow 0}\int_0^T\int_{\TT^d}(\sgn^\d)'(\rho^1-\rho^2)\left(\Theta_{\beta,M}(\rho^1)-\Theta_{\beta,M}(\rho^2)\right)(\nabla\rho^1-\nabla\rho^2)\cdot \dd\xi^F.
\end{align}
It follows from Assumption~\ref{assume_1}, Stampacchia's lemma (see \cite[Chapter~5, Exercises~17,18]{Eva2010}), and the definition on $\Theta_{\beta,M}$ that $\Theta_{\beta,M}$ is Lipschitz continuous on $\R$, and that there exists $c\in(0,\infty)$ independent of $\d\in(0,1)$ but depending on the definition of the convolution kernel, $M\in\N$, and $\beta\in(0,1)$ such that, for all $\d\in(0,\nicefrac{\beta}{4})$,
\[\abs{(\sgn^\d)'(\rho^1-\rho^2)\left(\Theta_{\beta,M}(\rho^1)-\Theta_{\beta,M}(\rho^2)\right)}\leq c\mathbf{1}_{\{0<\abs{\rho^1-\rho^2}<c\d\;\;\textrm{and}\;\;\nicefrac{\beta}{4}<\rho^i<M+\d\;\;\textrm{for}\;\;i\in\{1,2\}\}}.\]
It then follows from the local regularity of the solutions, the dominated convergence theorem, \eqref{lem_dct}, and \eqref{3_98900} that, almost surely for every $t\in[0,T]$,
\begin{equation}\label{3_989}  \int_0^t\int_{\TT^d}\sgn(\rho^2-\rho^1)\nabla\cdot\left(\left(\Theta_{\beta,M}(\rho^1)-\Theta_{\beta,M}(\rho^2)\right)\dd\xi^F\right)=0.\end{equation}
For the second term on the righthand side of \eqref{3_988}, it follows from the $L^2$-integrability of the $\sigma(\rho^i)$ that
\[\lim_{M\rightarrow\infty}\left(\lim_{\beta\rightarrow 0}\varphi_\beta(\rho^i)\zeta_M(\rho^i)\sigma(\rho^i)\right)=\sigma(\rho^i)\;\;\textrm{strongly in}\;\;L^2(\TT^d\times[0,T]).\]
It follows from the definition of $\Theta_{\beta,M}$, from $\sigma(0)=0$, and an explicit computation using the integration by parts formula and the definitions of $\varphi_\beta$ and $\zeta_M$ that there exists $c\in(0,\infty)$ such that, for every $M\in[1,\infty)$ and $\beta\in(0,\nicefrac{1}{2})$,
\[\abs{\sigma(\rho^i)-\Theta_{\beta,M}(\rho^i)}\leq c\left(\sup_{\xi\in[0,\beta]}\abs{\sigma(\xi)}+\sigma(\rho^i)\mathbf{1}_{\{\rho^i>M\}}+\sup_{\xi\in[M,(M+1)\wedge\rho^i]}\abs{\sigma(\xi)}\mathbf{1}_{\{\rho^i>M\}}\right).\]
The first term converges to zero as $\beta\rightarrow 0$ using the continuity of $\sigma$ and $\sigma(0)=0$, the second term converges to zero strongly in $L^2(\O\times[0,T];L^2(\TT^d))$ as $M\rightarrow\infty$ using the $L^2$-integrability of $\sigma(\rho^i)$ and the $L^1$-integrability of $\rho^i$, and using \eqref{assume_f1} of Assumption~\ref{assume_1} the final term converges strongly to zero as $M\rightarrow\infty$ in $L^2(\O\times[0,T];L^2(\TT^d)$.  It then follows from \eqref{lem_dct} that, along subsequences $\beta\rightarrow 0$ and $M\rightarrow\infty$, almost surely for every $t\in[0,T]$,
\begin{equation}\label{3_890}\lim_{M\rightarrow\infty}\left(\lim_{\beta\rightarrow 0}\int_0^t\int_{\TT^d}\sgn(\rho^2-\rho^1)\left(\varphi_\beta(\rho^1)\zeta_M(\rho^1)\sigma(\rho^1)-\Theta_{\beta,M}(\rho^1)\right)\nabla\cdot\dd\xi^F\right)=0,\end{equation}
and similarly for the analogous term defined by $\rho^2$.  In combination \eqref{3_988}, \eqref{3_989}, and \eqref{3_890} prove that, along subsequences, almost surely for every $t\in[0,T]$,
\begin{equation}\label{3_29}  \lim_{M\rightarrow\infty}\left(\lim_{\beta\rightarrow 0}\left(\lim_{\d\rightarrow 0}\left(\lim_{\ve\rightarrow 0} I^\textrm{mart}_t\right)\right)\right) =0.\end{equation}

\textbf{The conservative term.}  The conservative term is given for every $t\in[0,T]$ by
\[I^{\textrm{cons}}_t=\int_0^t\int_\R\int_{(\TT^d)^2}\overline{\kappa}^{\ve,\d}_{s,1}\nabla\cdot\nu(\rho^1)(1-2\chi^{\ve,\d}_{s,2})\varphi_\beta(\eta)\zeta_M(\eta)+\overline{\kappa}^{\ve,\d}_{s,2}\nabla\cdot\nu(\rho^2)(1-2\chi^{\ve,\d}_{s,1})\varphi_\beta(\eta)\zeta_M(\eta),\]
and is treated using the argument leading from \eqref{3_987} to \eqref{3_890}, since we may assume without loss of generality that $\nu(0)=0$.  In this case, we use the fact that $\nu(\rho)$ is $L^1$-integrable and use \eqref{assume_f2} to apply the dominated convergence theorem, whereas in the previous argument we used the $L^2$-integrability of $\sigma(\rho)$ to deal with the stochastic integral.  This proves that, along subsequences, almost surely for every $t\in[0,T]$,
\begin{equation}\label{3_29292929}  \lim_{M\rightarrow\infty}\left(\lim_{\beta\rightarrow 0}\left(\lim_{\d\rightarrow 0}\left(\lim_{\ve\rightarrow 0} I^\textrm{cons}_t\right)\right)\right) =0.\end{equation}

\textbf{The cutoff term.}  The cutoff term is defined for every $t\in[0,T]$, $\ve,\beta\in(0,1)$, $\d\in(0,\nicefrac{\beta}{4})$, and $M\in\N$ by
\begin{align*}
I^{\textrm{cut}}_t & = \int_0^t\int_{\R^2}\int_{(\TT^d)^2}\kappa^{\ve,\d}(x,y,\xi,\eta)(2\chi^{\ve,\d}_{s,2}-1)\partial_\eta(\varphi_\beta(\eta)\zeta_M(\eta))\dd q^1(x,\xi,s)
 \\  & +\frac{1}{2}\int_0^t\int_\R\int_{(\TT^d)^2}\left(F_3(x)\sigma^2(\rho^1(x,s))+\sigma(\rho^1)\sigma'(\rho^1)\nabla\rho^i\cdot F_2(x)\right)(2\chi^{\ve,\d}_{s,2}-1)\overline{\kappa}^{\ve,\d}_{s,1}\partial_\eta(\varphi_\beta(\eta)\zeta_M(\eta))
 \\ & + \int_0^t\int_{\R^2}\int_{(\TT^d)^2}\kappa^{\ve,\d}(x,y,\xi,\eta)(2\chi^{\ve,\d}_{1,2}-1)\partial_\eta(\varphi_\beta(\eta)\zeta_M(\eta))\dd q^2(x,\xi,s)
 \\  & +\frac{1}{2}\int_0^t\int_\R\int_{(\TT^d)^2}\left(F_3(x)\sigma^2(\rho^2(x,s))+\sigma(\rho^2)\sigma'(\rho^2)\nabla\rho^i\cdot F_2(x)\right)(2\chi^{\ve,\d}_{s,1}-1)\overline{\kappa}^{\ve,\d}_{s,2}\partial_\eta(\varphi_\beta(\eta)\zeta_M(\eta)).
\end{align*}
For the terms involving the kinetic measures, it follows from Definition~\ref{def_smooth} and the boundedness of the kinetic function that there exists $c\in(0,\infty)$ such that
\begin{align*}
& \limsup_{\ve,\d\rightarrow 0}\abs{\int_0^t\int_{\R^2}\int_{(\TT^d)^2}\kappa^{\ve,\d}(x,y,\xi,\eta)(2\chi^{\ve,\d}_{s,2}-1)\partial_\eta(\varphi_\beta(\eta)\zeta_M(\eta))\dd q^1(x,\xi,s)}
\\ & \leq c\left(\beta^{-1}q^1(\TT^d\times[\nicefrac{\beta}{2},\beta]\times[0,T])+q^1(\TT^d\times[M,M+1]\times[0,T])\right),
\end{align*}
and similarly for the term involving $q^2$.  It then follows from property \eqref{def_5353} of Definition~\ref{def_sol}, Proposition~\ref{prop_measure}, the nonnegativity of the measures, and Fatou's lemma that there almost surely exist subsequences $\beta\rightarrow 0$ and $M\rightarrow \infty$ such that
\[\lim_{M\rightarrow\infty}\left(\lim_{\beta\rightarrow0}\left(\beta^{-1}q^1(\TT^d\times[\nicefrac{\beta}{2},\beta]\times[0,T])+q^1(\TT^d\times[M,M+1]\times[0,T])\right)\right)=0,\]
and similarly for the term involving $q^2$.  The terms involving $F_3$ are treated identically to how they were treated in Proposition~\ref{prop_measure} using the $L^2$-integrability of the $\sigma(\rho^i)$, Assumption~\ref{assume_1}, and the boundedness of $F_3$.  Finally, for the terms involving $F_2$, we pass to the limit $\ve,\d\rightarrow 0$ exactly as in the arguments leading from \eqref{3_23} to \eqref{3_987} using Stampacchia's lemma (see \cite[Chapter~5, Exercises~17,18]{Eva2010}), Definition~\ref{def_smooth}, and property \eqref{def_2500000} of Definition~\ref{def_sol} to see that these terms become
\begin{align*}
& \beta^{-1}\int_0^T\int_{\TT^d}\sgn(\rho^2-\rho^1)\nabla \left(\left(\sigma^2((\rho^1\wedge \beta)\vee\nicefrac{\beta}{2})-\sigma^2(\nicefrac{\beta}{2})\right)-\left(\sigma^2((\rho^2\wedge \beta)\vee\nicefrac{\beta}{2})-\sigma^2(\nicefrac{\beta}{2})\right)\right)\cdot F_2
\\ &+\frac{1}{2}\int_0^T\int_{\TT^d}\sgn(\rho^2-\rho^1)\nabla\left(\sigma^2((\rho^1\wedge (M+1))\vee M)-\sigma^2(M)\right)\cdot F_2
\\  & -\frac{1}{2}\int_0^T\int_{\TT^d}\sgn(\rho^2-\rho^1)\nabla\left(\sigma^2((\rho^2\wedge (M+1))\vee M)-\sigma^2(M)\right)\cdot F_2.
\end{align*}
We then pass to the limits $\beta\rightarrow 0$ and $M\rightarrow\infty$ exactly as in Proposition~\ref{prop_measure} and \eqref{3_989}, using the local regularity of $\sigma$ and either that $\nabla\cdot F=0$ or that $(\sigma\sigma')\in\C([0,\infty))$ with $(\sigma\sigma')(0)=0$.  In combination these estimates prove that there almost surely exist subsequences such that, for every $t\in[0,T]$,
\begin{equation}\label{3_18} \lim_{M\rightarrow\infty}\left(\lim_{\beta\rightarrow 0}\left(\lim_{\d\rightarrow 0}\left(\lim_{\ve\rightarrow 0}I^{\textrm{cut}}_t\right)\right)\right)=0.\end{equation}

\textbf{Conclusion.} Properties of the kinetic function \eqref{3_0000} and estimates \eqref{3_21}, \eqref{3_22}, \eqref{3_12}, \eqref{3_29}, \eqref{3_29292929}, and \eqref{3_18} prove that there almost surely exist random subsequences $\ve,\d,\beta\rightarrow 0$ and $M\rightarrow\infty$ such that, for every $t\in[0,T]$,
\begin{align*}
\left.\int_\R\int_{\TT^d}\abs{\chi^1_t-\chi^2_t}^2\right|_{s=0}^{s=t} & = \lim_{M\rightarrow\infty}\left(\lim_{\beta\rightarrow 0}\left(\lim_{\d\rightarrow 0}\left(\lim_{\ve\rightarrow 0}\left.\int_\R\int_{\TT^d}\abs{\chi^{\ve,\d}_{t,1}-\chi^{\ve,\d}_{t,2}}^2\varphi_\beta\zeta_M\right|_{s=0}^{s=t}\right)\right)\right)
\\ & \leq \lim_{M\rightarrow\infty}\left(\lim_{\beta\rightarrow 0}\left(\lim_{\d\rightarrow 0}\left(\lim_{\ve\rightarrow 0}\left(-2I^{\textrm{err}}_t-2I^{\textrm{meas}}_t+I^{\textrm{mart}}_t+I^{\textrm{cut}}_t+I^{\textrm{cons}}_t\right)\right)\right)\right)
\\ & =0.
\end{align*}
Properties of the kinetic function \eqref{3_0000} then prove that
\[\int_{\TT^d}\abs{\rho^1(\cdot,t)-\rho^2(\cdot,t)}=\int_\R\int_{\TT^d}\abs{\chi^1_t-\chi^2_t}^2\leq \int_\R\int_{\TT^d}\abs{\overline{\chi}(\rho^1_0)-\overline{\chi}(\rho^2_0)}^2=\int_{\TT^d}\abs{\rho^1_0-\rho^2_0},\]
which completes the proof. \end{proof}

\section{Existence of stochastic kinetic solutions}\label{sec_whole}  In this section, we construct a stochastic kinetic solution of \eqref{1_0} in the sense of Definition~\ref{def_sol}.  The section is split into three subsections.  In Section~\ref{sec_a_priori} we obtain stable priori estimates for  solutions to a regularized version of \eqref{1_0}, and in Section~\ref{sec_exist_approx} we construct using these a priori estimates a solution to the regularization of \eqref{1_0}.  Finally, in Section~\ref{sec_exist}, we pass to the limit with respect to these regularizations and construct a solution of \eqref{1_0}.

\subsection{A priori estimates for \eqref{1_0}}\label{sec_a_priori}

In this section, we establish a priori estimates for the approximate equation
\begin{align}\label{4_0} \dd\rho=\Delta\Phi(\rho)\dt+\a\Delta\rho\dt-\nabla\cdot(\sigma(\rho)\dd\xi^F+\nu(\rho)\dt)+\frac{1}{2}\nabla\cdot\left(F_1[\sigma'(\rho)]^2\nabla\rho+\sigma(\rho)\sigma'(\rho)F_2\right)\dt,\end{align}
for $\a\in(0,\infty)$, and for nonlinearities $\Phi$, $\sigma$, and $\nu$ satisfying Assumptions~\ref{assume} and \ref{assume_3} below.  These estimates provide the foundation for our existence theory.

The most important estimates of the section are proven in Proposition~\ref{prop_approx_est}, which are based on the auxiliary function introduced in Lemma~\ref{lem_aux} and Assumptions~\ref{assume} and \ref{assume_3}.  In Lemma~\ref{lem_frac} and Corollary~\ref{cor_frac_sob} we show that the estimates of Proposition~\ref{prop_approx_est} imply fractional Sobolev regularity for the solution.  In Proposition~\ref{prop_approx_time}, we obtain stable $W^{\beta,1}_tH^{-s}_x$-estimates for nonlinear functions $\Psi_\d(\rho)$ of the solution defined in Definition~\ref{def_cutoff} that localize the solution away from its zero set.  Finally, in Proposition~\ref{ent_dis_est} we prove an entropy dissipation estimate for initial data with finite entropy.  This estimate will be used in Corollary~\ref{cor_exist} to extend the existence proof for $L^p$-initial data of Theorem~\ref{thm_exist}, for $p\in[2,\infty)$, to initial data with finite entropy.
\begin{lem}\label{lem_aux}  Let $\Phi\in\C([0,\infty))\cap\C^1_{\textrm{loc}}((0,\infty))$ be strictly increasing with $\Phi(0)=0$, let $p\in[2,\infty)$, and let $\Theta_{\Phi,p}\in\C([0,\infty))\cap\C^1_{\textrm{loc}}((0,\infty))$ be the unique function satisfying $\Theta_{\Phi,p}(0)=0$ and $\Theta_{\Phi,p}'(\xi)=\xi^{\nicefrac{(p-2)}{2}}[\Phi'(\xi)]^{\nicefrac{1}{2}}$.  Assume that there exists $c\in(0,\infty)$ and $m\in\N$ such that $\Phi(\xi)\leq c(1+\xi^m)$ for every $\xi\in[0,\infty)$.  Then, for every $p\in[2,\infty)$ there exists $c\in(0,\infty)$ such that
\[0\leq \Theta_{\Phi,p}(\xi)\leq c(1+\xi^{\frac{m+p-1}{2}})\;\;\textrm{for every}\;\;\xi\in[0,\infty).\]
\end{lem}

\begin{proof}  Since $\Phi$ is strictly increasing, it follows from the fundamental theorem of calculus, $\Phi(0)=0$, and H\"older's inequality that
\[0\leq \Theta_{\Phi,p}(\xi)=\int_0^\xi(\xi')^{\nicefrac{p-2}{2}}[\Phi'(\xi')]^\frac{1}{2}\dxp\leq (\nicefrac{1}{p})^\frac{1}{2}\xi^\frac{p-1}{2}\left(\int_0^\xi\Phi'(\xi')\dxp\right)^\frac{1}{2}=(\nicefrac{1}{p})^\frac{1}{2}\xi^\frac{p-1}{2}\Phi(\xi)^\frac{1}{2}.\]
The claim now follows from $\Phi(\xi)\leq c(1+\xi^m)$ and Young's inequality.  \end{proof}
\begin{assumption}\label{assume}  Let $\Phi,\sigma\in\C([0,\infty))\cap\C^1_{\textrm{loc}}((0,\infty)))$, $\nu\in\C([0,\infty);\R^d)\cap\C^1_{\textrm{loc}}((0,\infty);\R^d)$, and $p\in[2,\infty)$ satisfy the following seven assumptions.

\begin{enumerate}[(i)]
\item We have $\Phi(0)=\sigma(0)=0$ and $\Phi'>0$ on $(0,\infty)$.
\item There exists $m\in[1,\infty)$ and $c\in(0,\infty)$ such that
\begin{equation}\label{aa_0} \Phi(\xi)\leq c(1+\xi^m)\;\;\textrm{for every}\;\;\xi\in[0,\infty).\end{equation}
\item There exists $c\in(0,\infty)$ such that, for $\Theta_{\Phi,p}$ defined in Lemma~\ref{lem_aux},
\begin{equation}\label{11_11_11}\abs{\nu(\xi)}+\Phi'(\xi)\leq c(1+\xi+\Theta^2_{\Phi,p}(\xi))\;\;\textrm{for every}\;\;\xi\in(0,\infty).\end{equation}
\item Either there exists $c\in(0,\infty)$ and $\theta\in[0,\nicefrac{1}{2}]$ such that
\begin{equation}\label{5_00}\xi^{-\left(\frac{p-2}{2}\right)}[\Phi'(\xi)]^{-\nicefrac{1}{2}}\leq c\xi^\theta\;\;\textrm{for every}\;\;\xi\in(0,\infty),\end{equation}
or for $\Theta_{\Phi,p}$ defined in Lemma~\ref{lem_aux}, there exists $c\in(0,\infty)$ and $q\in[1,\infty)$ such that
\begin{equation}\label{5_0}\abs{\xi-\xi'}^q\leq c\abs{\Theta_{\Phi,p}(\xi)-\Theta_{\Phi,p}(\xi')}^2\;\;\textrm{for every}\;\;\xi,\xi'\in[0,\infty).\end{equation}
\item For $\Theta_{\Phi,2}$ and $\Theta_{\Phi,p}$ defined in Lemma~\ref{lem_aux}, there exists $c\in(0,\infty)$ such that
\begin{equation}\label{aa_5050}\sigma^2(\xi)\leq c(1+\xi+\Theta^2_{\Phi,2}(\xi))\;\;\textrm{and}\;\;\xi^{p-2}\sigma^2(\xi)\leq c(1+\xi+\Theta^2_{\Phi,p}(\xi))\;\;\textrm{for every}\;\;\xi\in[0,\infty).\end{equation}
\item Either $\nabla\cdot F_2=0$ or for the unique function $\Psi_{\sigma,p}\in\C([0,\infty))\cap\C^1_\textrm{loc}((0,\infty))$ defined by $\Psi_{\sigma,p}(0)=0$ and $\Psi_{\sigma,p}'(\xi)=\xi^{p-2}\sigma(\xi)\sigma'(\xi)$ there exists $c\in(0,\infty)$ such that
\begin{equation}\label{aa_7000}\abs{\Psi_{\sigma,p}(\xi)}\leq c(1+\xi+\Theta^2_{\Phi,p}(\xi))\;\;\textrm{for every}\;\;\xi\in[0,\infty).\end{equation}
\item For every $\d\in(0,1)$ there exists $c_\d\in(0,\infty)$ such that, for every $\xi\in(\d,\infty)$,
\begin{equation}\label{aa_121212}\frac{[\sigma'(\xi)]^4}{\Phi'(\xi)}+(\sigma(\xi)\sigma'(\xi))^2+\Phi'(\xi)\leq c_\d(1+\xi+\Theta^2_{\Phi,p}(\xi)).\end{equation}
\end{enumerate}
\end{assumption}

\begin{example}\label{example_1}  In the model case that $\Phi(\xi)=\xi^m$ the function $\Theta_{\Phi,p}$ defined in Lemma~\ref{lem_aux} is given for a constant $c_{p,m}\in(0,\infty)$ by
\[\Theta_{\Phi,p}(\xi)=c_{p,m}\xi^{\frac{m+p-1}{2}}.\]
An explicit computation proves that $\Theta_{\Phi,p}$ satisfies condition (2) of Assumption~\ref{assume} for every $m\in(0,\infty)$  and $p\in[2,\infty)$.  In the case $p=2$, the function $\Theta_{\Phi,2}$ satisfies \eqref{5_00} if $m\in(0,1]$ and $\Theta_{\Phi,2}$ satisfies \eqref{5_0} for $q=m+1$ if $m\in[1,\infty)$.  Concerning condition \ref{aa_7000}, in the model case that $\Phi(\xi)=\xi^m$ and $\sigma=\Phi^{\nicefrac{1}{2}}$, it follows for some $\tilde{c}_{p,m}\in(0,\infty)$ that $\Psi_{\sigma,p}(\xi)=\tilde{c}_{p,m}\xi^{m+p-2}$.  Assumption \eqref{aa_7000} is therefore satisfied for every $m\in(0,\infty)$ and $p\in[2,\infty)$.  Finally, since in every model case $\Phi'$ is uniformly bounded away from zero on $[\d,\infty)$, condition \eqref{aa_121212} amounts to a mild growth condition on $\sigma$ and its derivative that is satisfied by every example, for every $m\in(0,\infty)$.\end{example}

\begin{lem}  Assume $\Psi\in\C([0,\infty))$ is strictly increasing with $\Psi(0)=0$, assume that there exists $p\in[1,\infty)$ such that
\begin{equation}\label{50_1}\Psi^2(\xi)\leq (1+\xi^p)\;\;\textrm{for every}\;\;\xi\in[0,\infty),\end{equation}
and assume that $z\colon\TT^d\rightarrow\R$ is measurable with $\Psi(z)\in H^1(\TT^d)$.  Then for every $\ve\in(0,1)$ there exists $c\in(0,\infty)$ depending on $\ve$ such that
\begin{equation}\label{50_2}\norm{\Psi(z)}^2_{L^2(\TT^d)}\leq c(1+\norm{z}^p_{L^1(\TT^d)})+\ve\norm{\nabla\Psi(z)}^2_{L^2(\TT^d)}.\end{equation}
\end{lem}

\begin{proof}  It follows from \eqref{50_1} that there exists $c\in(0,\infty)$ such that
\begin{equation}\label{4_30}\norm{\Psi(z)^\frac{2}{p}}^p_{L^1(\TT^d)}\leq c(1+\norm{z}^p_{L^1(\TT^d)}).\end{equation}
Let $2_*=6$ if $d\in\{1,2,3\}$ and let $\nicefrac{1}{2_*}=\nicefrac{1}{2}-\nicefrac{1}{d}$ if $d\geq 3$.  The triangle inequality, H\"older's inequality, and the Sobolev inequality prove that there exists $c\in(0,\infty)$ such that
\begin{align}\label{4_030} & \norm{\Psi(z)^\frac{2}{p}}^p_{L^\frac{2_*p}{2}(\TT^d)} =\norm{\Psi(z)}^2_{L^{2_*}(\TT^d)}
\\ \nonumber & \leq c\left(\langle\Psi(z)\rangle^2+\left(\int_{\TT^d}\left(\Psi(z)-\langle\Psi(z)\rangle\right)^{2_*}\right)^\frac{2}{2_*}\right) \leq c\left(\E\left[\norm{\Psi(z)}^2_{L^2(\TT^d)}\right]+\E\left[\int_{\TT^d}\abs{\nabla\Psi(z)}^2\right]\right),
\end{align}
for $\langle\Psi(z)\rangle = \int_{\TT^d}\Psi(z)$.  Since $1\leq p\leq \frac{2_*p}{2}$, interpolating between estimates \eqref{4_30} and \eqref{4_030} yields for $\nicefrac{1}{p}=\theta+\nicefrac{2(1-\theta)}{2_*p}$ that
\[\norm{\Psi(z)}^2_{L^2(\TT^d)}=\norm{\Psi(z)^\frac{2}{p}}^p_{L^p(\TT^d)}\leq \norm{\Psi(z)^\frac{2}{p}}^{\theta p}_{L^1(\TT^d)}\norm{\Psi(z)^\frac{2}{p}}^{(1-\theta)p}_{L^\frac{2_*p}{2}(\TT^d)},\]
and hence by H\"older's inequality, Young's inequality, \eqref{4_30}, and \eqref{4_030} for every $\d\in(0,1)$ there exists $c\in(0,\infty)$ such that
\[\norm{\Psi(z)}^2_{L^2(\TT^d)}\leq \d\norm{\Psi(z)}^2_{L^2(\TT^d)}+\d\int_{\TT^d}\abs{\nabla\Psi(z)}^2+c\left(1+\norm{z}^p_{L^1(\TT^d)}\right).\]
After choosing $\d(1-\d)^{-1}=\ve$ this completes the proof.\end{proof}

\begin{remark}The following assumption is only used to obtain the a priori estimates below and to construct solutions on an approximate level in Section~\ref{sec_exist_approx}, where the purpose of this assumption is to avoid singularities like \eqref{eq:singular_term} in the It\^o-correction.  We dispense with this assumption and treat general nonlinearities $\sigma$ in Section~\ref{sec_exist}.\end{remark}

\begin{assumption}\label{assume_3}  Let $\sigma\in\C([0,\infty))\cap\C^\infty((0,\infty))$ with $\sigma(0)=0$ and with $\sigma'\in \C^\infty_c([0,\infty))$.\end{assumption}

\begin{definition}\label{def_approx} Let $\xi^F$, $\Phi$, $\sigma$, and $\nu$ satisfy Assumptions~\ref{assumption_noise}, \ref{assume}, and \ref{assume_3} for some $p\in[2,\infty)$, let $\a\in(0,1)$, and let $\rho_0\in L^{p+m-1}(\O;L^1(\TT^d))\cap L^p(\O;L^p(\TT^d))$ be nonnegative and $\F_0$-measurable.  A solution of \eqref{4_0} with initial data $\rho_0$ is a continuous $L^p(\TT^d)$-valued, nonnegative, $\F_t$-predictable process $\rho$ such that almost surely $\rho$ and $\Theta_{\Phi,2}(\rho)$ are in $L^2([0,T];H^1(\TT^d))$ and such that for every $\psi\in\C^\infty(\TT^d)$, almost surely for every $t\in[0,T]$,
\begin{align*}
& \int_{\TT^d}\rho(x,t)\psi(x)\dx  = \int_{\TT^d}\rho_0\psi\dx -\int_0^t\int_{\TT^d}\Phi'(\rho)\nabla\rho\cdot\nabla\psi - \a\int_0^t\int_{\TT^d}\nabla\rho\cdot\nabla\psi+\int_0^t\int_{\TT^d}\nu(\rho)\cdot\nabla\psi
\\ & \quad +\int_0^t\int_{\TT^d}\sigma(\rho)\nabla\psi\cdot\dd\xi^F -\frac{1}{2}\int_0^t\int_{\TT^d}F_1[\sigma'(\rho)]^2\nabla\rho\cdot\nabla\psi-\frac{1}{2}\int_0^t\int_{\TT^d}\sigma(\rho)\sigma'(\rho)F_2\cdot\nabla\psi.
\end{align*}  \end{definition}

\begin{remark}  In Definition~\ref{def_approx} the distributional equality $\Phi'(\rho)\nabla\rho = [\Phi'(\rho)]^{\nicefrac{1}{2}}\nabla\Theta_{\Phi,2}(\rho)$, the assumption that $\Theta_{\Phi,2}(\rho)$ is an $L^2_tH^1_x$-function, \eqref{11_11_11}, and H\"older's inequality prove that the term $\Phi'(\rho)\nabla\rho$ is integrable. The integrability of $\nu(\rho)$ follows from \eqref{11_11_11} and the estimates of Proposition~\ref{prop_approx_est} below. \end{remark}

\begin{prop}\label{prop_approx_est}  Let $\xi^F$, $\Phi$, $\sigma$, and $\nu$ satisfy Assumptions~\ref{assumption_noise}, \ref{assume}, and \ref{assume_3} for some $p\in[2,\infty)$, let $\a\in(0,1)$, let $T\in[1,\infty)$, let $\rho_0\in L^{m+p-1}(\O;L^1(\TT^d))\cap L^p(\O;L^p(\TT^d))$ be nonnegative and $\F_0$-measurable, and let $\rho$ be a solution of \eqref{4_0} in the sense of Definition~\ref{def_approx}.  Then, almost surely for every $t\in[0,T]$,
\begin{equation}\label{4_00019}\norm{\rho(\cdot,t)}_{L^1(\TT^d)}=\norm{\rho_0}_{L^1(\TT^d)}.\end{equation}
For $\Theta_{\Phi,p}$ defined in Lemma~\ref{lem_aux}, there exists $c\in(0,\infty)$ depending on $p$ but independent of $\a$ and $T$ such that
\begin{align}\label{4_19}
&\sup_{t\in[0,T]}\E\left[\int_{\TT^d}\rho^p(x,t)\right]+\E\left[\int_0^T\int_{\TT^d}\abs{\nabla\Theta_{\Phi,p}(\rho)}^2\right]+\E\left[\a\int_0^T\int_{\TT^d}\abs{\rho}^{p-2}\abs{\nabla\rho}^2\right]
\\ \nonumber & \leq  cT\left(1+\E\left[\norm{\rho_0}^{m+p-1}_{L^1(\TT^d)}+\norm{\rho_0}^p_{L^p(\TT^d)}\right]\right).
\end{align}
For every $M_1<M_2\in(0,\infty)$ there exists $c\in(0,\infty)$ independent of $M_1$ and $M_2$ such that
\begin{align}\label{50_7070}
& \E\left[\int_0^T\int_{\TT^d}\mathbf{1}_{\{M_1<\rho<M_2\}}\Phi'(\rho)\abs{\nabla\rho}^2+\a\int_0^T\int_{\TT^d}\mathbf{1}_{\{M_1<\rho<M_2\}}\abs{\nabla\rho}^2\right]
\\ \nonumber & \leq  c\E\left[\int_{\TT^d}(\rho_0-M_1)_++\int_0^T\int_{\TT^d}\mathbf{1}_{\{\rho\geq M_1\}}\sigma^2(\rho\wedge M_2)\right].
\end{align}
\end{prop}

\begin{proof}  The $L^1$-estimate follows from the nonnegativity of $\rho$, after choosing $\psi=1$ in Definition~\ref{def_approx}.  For the energy estimate \eqref{4_19}, it follows from Assumption~\ref{assume} and It\^o's formula---which is justified similarly to \cite[Proposition~7.7]{FehGes2020} using the version of It\^o's formula proven in Krylov~\cite[Theorem~3.1]{Kry2013}---that, for $\Theta_{\Phi,p}$ defined in Lemma~\ref{lem_aux}, almost surely for every $t\in[0,T]$,
\begin{align}\label{4_00}
& \left.\frac{1}{p(p-1)}\int_{\TT^d}\rho^2(x,s)\right|_{s=0}^{s=t} =-\int_0^t\int_{\TT^d}\abs{\nabla\Theta_{\Phi,p}(\rho)}^2-\a\int_0^t\int_{\TT^d}\abs{\rho}^{p-2}\abs{\nabla\rho}^2
\\ \nonumber & \quad +\int_0^t\int_{\TT^d}\abs{\rho}^{p-2}\nu(\rho)\cdot\nabla\rho+\int_0^t\int_{\TT^d}\sigma(\rho)\abs{\rho}^{p-2}\nabla\rho\cdot\dd\xi^F-\frac{1}{2}\int_0^t\int_{\TT^d}F_1[\sigma'(\rho)]^2\abs{\rho}^{p-2}\abs{\nabla\rho}^2
\\ \nonumber & \quad -\frac{1}{2}\int_0^t\int_{\TT^d}\sigma(\rho)\sigma'(\rho)\abs{\rho}^{p-2}F_2\cdot\nabla\rho +\frac{1}{2}\sum_{k=1}^\infty\int_0^t\int_{\TT^d}\abs{\rho}^{p-2}\left(\nabla\cdot(\sigma(\rho)f_k)\right)^2.
\end{align}
The definitions of the coefficients $F_i$ prove that
\begin{equation}\label{4_01} \frac{1}{2}\sum_{k=1}^\infty\left(\nabla\cdot(\sigma(\rho)f_k)\right)^2=\frac{1}{2}\left(F_1[\sigma'(\rho)]^2\abs{\nabla\rho}^2+2\sigma(\rho)\sigma'(\rho)F_2\cdot\nabla\rho+F_3\sigma^2(\rho)\right).\end{equation}
Returning to \eqref{4_00}, it follows from \eqref{4_01} that, almost surely for every $t\in[0,T]$,
\begin{align}\label{4_02}
& \left.\frac{1}{p(p-1)}\int_{\TT^d}\rho^p(x,s)\right|_{s=0}^{s=t} =-\int_0^t\int_{\TT^d}\abs{\nabla\Theta_{\Phi,p}(\rho)}^2-\a\int_0^t\int_{\TT^d}\abs{\rho}^{p-2}\abs{\nabla\rho}^2+\int_0^t\int_{\TT^d}\abs{\rho}^{p-2}\nu(\rho)\cdot\nabla\rho
\\ \nonumber & +\int_0^t\int_{\TT^d}\sigma(\rho)\abs{\rho}^{p-2}\nabla\rho\cdot\dd\xi^F+\frac{1}{2}\int_0^t\int_{\TT^d}F_3\abs{\rho}^{p-2}\sigma^2(\rho)+\frac{1}{2}\int_0^t\int_{\TT^d}\abs{\rho}^{p-2}\sigma(\rho)\sigma'(\rho)F_2\cdot\nabla\rho.
\end{align}
For the third term on the righthand side of \eqref{4_02}, it follows that
\begin{equation}\label{15_1}\abs{\int_0^T\int_{\TT^d}\abs{\rho}^{p-2}\nabla\rho\cdot\nu(\rho)}=\abs{\sum_{i=1}^d\int_0^T\int_{\TT^d}\partial_i\left(\Theta_{\nu,p,i}(\rho)\right)}=0,\end{equation}
for $\Theta_{\nu,p,i}\in\C([0,\infty))\cap\C^1_{\textrm{loc}}((0,\infty))$ the unique functions satisfying $\Theta_{\nu,p,i}(0)=0$ and $\Theta'_{\nu,p,i}(\xi)=\abs{\xi}^{p-2}\nu_i(\xi)$.  After integrating by parts in the final term of \eqref{4_02},
\[\int_0^t\int_{\TT^d}\sigma(\rho)\sigma'(\rho)\abs{\rho}^{p-2}F_2\cdot\nabla\rho=-\int_0^t\int_{\TT^d}\Psi_{\sigma,p}(\rho)\nabla\cdot F_2,\]
for $\Psi_{\sigma,\rho}\in\C([0,\infty))\cap\C^1_{\textrm{loc}}((0,\infty))$ defined in \eqref{aa_7000}.   It follows that this term is zero if $\nabla\cdot F_2=0$ or using the boundedness of $\nabla\cdot F_2$ that it is bounded by \eqref{aa_7000}.  After taking the expectation of \eqref{4_02}, it follows from $T\in[1,\infty)$, \eqref{aa_5050}, \eqref{4_00019}, the boundedness of $\nabla\cdot F_2$ and $F_3$, H\"older's inequality, and Young's inequality that there exists $c\in(0,\infty)$ such that
\begin{align*}
&\sup_{t\in[0,T]}\E\left[\int_{\TT^d}\rho^p(x,t)\right]+\E\left[\int_0^T\int_{\TT^d}\abs{\nabla\Theta_{\Phi,p}(\rho)}^2\right]+\E\left[\a\int_0^T\int_{\TT^d}\abs{\rho}^{p-2}\abs{\nabla\rho}^2\right]
\\ \nonumber & \leq  cT\E\left[1+\norm{\rho_0}_{L^1(\TT^d)}+\int_{\TT^d}\rho^p_0+\int_0^T\int_{\TT^d}\Theta_{\Phi,p}^2(\rho)\right].
\end{align*}
Lemma~\ref{lem_aux}, the interpolation estimate \eqref{50_2} with $\Psi=\Theta_{\Phi,p}$ integrated in time, the $L^1$-estimate \eqref{4_00019}, $T\in[1,\infty)$, and Young's inequality prove that there exists $c\in(0,\infty)$ such that
\begin{align*}
&\sup_{t\in[0,T]}\E\left[\int_{\TT^d}\rho^p(x,t)\right]+\E\left[\int_0^T\int_{\TT^d}\abs{\nabla\Theta_{\Phi,p}(\rho)}^2\right]+\E\left[\a\int_0^T\int_{\TT^d}\abs{\rho}^{p-2}\abs{\nabla\rho}^2\right]
\\ \nonumber & \leq  cT\left(1+\E\left[\norm{\rho_0}^{m+p-1}_{L^1(\TT^d)}+\norm{\rho_0}^p_{L^p(\TT^d)}\right]\right),
\end{align*}
which completes the proof of \eqref{4_19}.

To prove \eqref{50_7070}, following an approximation argument that justifies applying It\^o's formula \cite[Theorem~3.1]{Kry2013} as in \cite[Proposition~7.7]{FehGes2020} to the unique function $S_M\colon[0,\infty)\rightarrow[0,\infty)$ satisfying $S''_M(\xi)=\mathbf{1}_{\{M_1<\xi<M_2\}}$, after observing the cancellation between the It\^o- and It\^o-to-Stratonovich corrections as above, we have almost surely that
\begin{align*}
& \int_{\TT^d}S_M(\rho(x,T))  = \int_{\TT^d}S_M(\rho_0)-\int_0^T\int_{\TT^d}\mathbf{1}_{\{M_1<\rho<M_2\}}\Phi'(\rho)\abs{\nabla\rho}^2-\a\int_0^T\int_{\TT^d}\mathbf{1}_{\{M_1<\rho<M_2\}}\abs{\nabla\rho}^2
\\ & \quad +\int_0^T\int_{\TT^d}S''_M(\rho)\nu(\rho)\cdot\nabla\rho+\int_0^T\int_{\TT^d}S''_M(\rho)\sigma(\rho)\nabla\rho\cdot\dd\xi^F
\\ & \quad +\frac{1}{2}\int_0^T\int_{\TT^d}\mathbf{1}_{\{M_1<\rho<M_2\}}\left(\sigma(\rho)\sigma'(\rho)\nabla\rho\cdot F_2+F_3\sigma^2(\rho)\right).
\end{align*}
The fourth term on the righthand side vanishes similarly to \eqref{15_1}.  Therefore, after taking the expectation and using the distributional equality
\[\mathbf{1}_{\{M_1<\rho<M_2\}}\sigma(\rho)\sigma'(\rho)\nabla\rho=\frac{1}{2}\nabla\left[\left(\sigma^2((\rho\wedge M_2)\vee M_1)-\sigma^2(M_1)\right)\right],\]
we have using the definition of $S_M$ that
\begin{align*}
& \E\left[\int_0^T\int_{\TT^d}\mathbf{1}_{\{M_1<\rho<M_2\}}\Phi'(\rho)\abs{\nabla\rho}^2+\a\int_0^T\int_{\TT^d}\mathbf{1}_{\{M_1<\rho<M_2\}}\abs{\nabla\rho}^2\right] \leq \E\left[\int_{\TT^d}(\rho_0-M_1)_+\right]
\\ &\quad -\frac{1}{2}\E\left[\int_0^T\int_{\TT^d}\left(\frac{1}{2}\left(\sigma^2((\rho\wedge M_2)\vee M_1)-\sigma^2(M_1)\right)\nabla\cdot F_2-\mathbf{1}_{\{M_1<\rho<M_2\}}F_3\sigma^2(\rho)\right)\right].
\end{align*}
The boundedness of $\nabla\cdot F_2$ and $F_3$ prove that there exists $c\in(0,\infty)$ such that
\begin{align*}
& \E\left[\int_0^T\int_{\TT^d}\mathbf{1}_{\{M_1<\rho<M_2\}}\Phi'(\rho)\abs{\nabla\rho}^2+\a\int_0^T\int_{\TT^d}\mathbf{1}_{\{M_1<\rho<M_2\}}\abs{\nabla\rho}^2\right]
\\ & \leq c\E\left[\int_{\TT^d}(\rho_0-M_1)_++\int_0^T\int_{\TT^d}\mathbf{1}_{\{\rho\geq M_1\}}\sigma^2(\rho\wedge M_2)\right],
\end{align*}
which completes the proof of \eqref{50_7070} and the proof.\end{proof}

\begin{remark}\label{rem_rem} Observe that Assumption~\ref{assume} guarantees that the estimates of Proposition~\ref{prop_approx_est} are always satisfied in the case $p=2$, even when $p\in(2,\infty)$.  The case $p=2$ appears in the definition of the approximate kinetic measures \eqref{4_700} below. \end{remark}

\begin{lem}\label{lem_frac}  Let $\Phi$ satisfy Assumption~\ref{assume} for some $p\in[2,\infty)$ and let $z\in H^1(\TT^d)$ be nonnegative. If $\Phi$ satisfies \eqref{5_00}, then
\begin{equation}\label{5_000}\norm{\nabla z}_{L^1(\TT^d;\R^d)}\leq \norm{z}_{L^1(\TT^d)}^{\theta}\norm{\nabla\Theta_{\Phi,p}(z)}_{L^2(\TT^d)}.\end{equation}
If $\Phi$ satisfies \eqref{5_0}, then for every $\beta\in (0,\nicefrac{2}{q}\wedge 1)$, for some $c\in(0,\infty)$ depending on $\beta$,
\[\norm{z}_{W^{\beta,1}(\TT^d)}\leq c\left(\norm{z}_{L^1(\TT^d)}+\norm{\nabla\Theta_{\Phi,p}(z)}^\frac{2}{q}_{L^2(\TT^d;\R^d)}\right).\]
\end{lem}

\begin{proof}  It follows from an approximation argument using \eqref{5_00}, $\Phi\in \C^1_{\textrm{loc}}((0,\infty))$, and the chain rule (see \cite[Chapter~5, Exercises~17,18]{Eva2010}) that $\nabla z = \abs{\rho}^{2-p}[\Phi'(z)]^{-\nicefrac{1}{2}}\nabla\Theta_{\Phi,p}(z)$.  It then follows from \eqref{5_00} and H\"older's inequality that, for $\theta\in[0,\nicefrac{1}{2}]$ as in \eqref{5_00} and for some $c\in(0,\infty)$,
\begin{align*}
\norm{\nabla z}_{L^1(\TT^d)}\leq c\norm{z^\theta\nabla\Theta_{\Phi,p}(z)}_{L^2(\TT^d)} & \leq c\norm{z^{2\theta}}^\frac{1}{2}_{L^1(\TT^d)}\norm{\nabla\Theta_{\Phi,p}(z)}_{L^2(\TT^d)}
\\ & \leq c\norm{z}^\theta_{L^1(\TT^d)}\norm{\nabla\Theta_{\Phi,p}(z)}_{L^2(\TT^d)},
\end{align*}
which completes the proof of \eqref{5_000}.  If $\Phi$ satisfies \eqref{5_0}, then for $q\in[1,\infty)$ as in \eqref{5_0} the fractional Sobolev semi-norm satisfies, for every $\beta\in(0,\nicefrac{2}{q}\wedge 1)$, for some $c\in(0,\infty)$ depending on $\beta$,
\begin{align*}
& \int_{\TT^d}\int_{\TT^d}\frac{\abs{z(x)-z(y)}}{\abs{x-y}^{d+\beta}}\dx\dy\leq c\int_{\TT^d}\int_{\TT^d}\frac{\abs{\Theta_{\Phi,p}(z)(x)-\Theta_{\Phi,p}(z)(y)}^\frac{2}{q}}{\abs{x-y}^{d+\beta}}\dx\dy
\\ & \leq c\int_{\TT^d}\int_{\TT^d}\abs{\int_0^1\nabla\Theta_{\Phi,p}(z)(y+s(x-y))\cdot(x-y)\ds}^\frac{2}{q}\abs{x-y}^{-(d+\beta)}\dx\dy
\\ & \leq c\int_{\TT^d}\int_{\TT^d}\int_0^1\abs{\nabla\Theta_{\Phi,p}(z)(y+s(x-y))}^\frac{2}{q}\abs{x-y}^{-(d+\beta-\nicefrac{2}{q})}\ds\dx\dy.
\end{align*}
It then follows from H\"older's inequality, $q\in[1,\infty)$, and $d+\beta-\nicefrac{2}{q} < d$ that there exists $c\in(0,\infty)$ depending on $\beta$ such that
\[\int_{\TT^d}\int_{\TT^d}\frac{\abs{z(x)-z(y)}}{\abs{x-y}^{d+\beta}}\dx\dy\leq c\left(\int_{\TT^d}\abs{\nabla\Theta_{\Phi,p}(z)}^2\dx\right)^\frac{1}{q},\]
which, together with the definition of the fractional Sobolev norm, completes the proof.\end{proof}

\begin{cor}\label{cor_frac_sob}   Let $\xi^F$, $\Phi$, $\sigma$, and $\nu$ satisfy Assumptions~\ref{assumption_noise}, \ref{assume}, and \ref{assume_3} for some $p\in[2,\infty)$, let $T\in[1,\infty)$, let $\a\in(0,1)$, let $\rho_0\in L^{m+p-1}(\O;L^1(\TT^d))\cap L^p(\O;L^p(\TT^d))$ be nonnegative and $\F_0$-measurable, and let $\rho$ be a solution of \eqref{4_0} in the sense of Definition~\ref{def_approx}.  Then, if $\Phi$ satisfies \eqref{5_00}, for some $c\in(0,\infty)$ independent of $\a$ and $T$,
\[\E\left[\norm{\rho}_{L^1([0,T];W^{1,1}(\TT^d))}\right]\leq cT\left(1+\E\left[\norm{\rho_0}^{m+p-1}_{L^1(\TT^d)}+\int_{\TT^d}\rho_0^p\right]\right).\]
If $\Phi$ satisfies \eqref{5_0} then for every $\beta\in(0,\nicefrac{2}{q}\wedge 1)$ there exists $c\in(0,\infty)$ depending on $\beta$ but independent of $\a$ and $T$ such that
\[\E\left[\norm{\rho}_{L^1([0,T];W^{\beta,1}(\TT^d))}\right]\leq cT\left(1+\E\left[\norm{\rho_0}^{m+p-1}_{L^1(\TT^d)}+\int_{\TT^d}\rho_0^p\right]\right).\]
\end{cor}

\begin{proof}  The proof is a consequence of Proposition~\ref{prop_approx_est}, Lemma~\ref{lem_frac}, $T\in[1,\infty)$, $\theta\in[0,\nicefrac{1}{2}]$, $q\in[1,\infty)$, and Young's inequality.  \end{proof}

\begin{definition}\label{def_cutoff}  For every $\d\in(0,1)$ let $\psi_\d\in\C^\infty([0,\infty))$ be a smooth nondecreasing function satisfying $0\leq\psi_\d\leq 1$, $\psi_\d(\xi)=1$ if $\xi\geq \d$, $\psi_\d(\xi)=0$ if $\xi\leq\nicefrac{\d}{2}$, and $\abs{\psi_\d'(\xi)}\leq\nicefrac{c}{\d}$ for some $c\in(0,\infty)$ independent of $\d$.  For every $\d\in(0,1)$ let $\Psi_\d\in\C^\infty([0,\infty))$ be defined by
\[\Psi_\d(\xi)=\psi_\d(\xi)\xi\;\;\textrm{for every}\;\;\xi\in[0,\infty).\]
\end{definition}

\begin{prop}\label{prop_approx_time}  Let $\xi^F$, $\Phi$, $\sigma$, and $\nu$ satisfy Assumptions~\ref{assumption_noise}, \ref{assume}, and \ref{assume_3} for some $p\in[2,\infty)$, let $T\in[1,\infty)$, let $\a\in(0,1)$, let $\rho_0\in L^{m+p-1}(\O;L^1(\TT^d))\cap L^p(\O;L^p(\TT^d))$ be nonnegative and $\F_0$-measurable, and let $\rho$ be a solution of \eqref{4_0} in the sense of Definition~\ref{def_approx}.  Then, for every $\beta\in(0,\nicefrac{1}{2})$ and $s>\frac{d}{2}+1$ there exists $c\in(0,\infty)$ depending on $\d$, $\beta$, and $s$ but independent of $\a$ and $T$ such that
\[\E\left[\norm{\Psi_\d(\rho)}_{W^{\beta,1}([0,T];H^{-s}(\TT^d))}\right]\leq cT\left(1+\E\left[\norm{\rho_0}^{p+m-1}_{L^1(\TT^d)}+\int_{\TT^d}\rho_0^p\right]\right).\]
\end{prop}

\begin{proof}  It follows from It\^o's formula---which is justified similarly to \cite[Proposition~7.7]{FehGes2020} using \cite[Theorem~3.1]{Kry2013}---the compact support of $\Psi_\d$ in $(0,\infty)$, and the distributional equality \eqref{4_02525} that, for every $\d\in(0,1)$, as distributions on $\TT^d$, we have that $\Psi_\d(\rho(x,t)) = \Psi_\d(\rho_0)+ I^{\textrm{f.v.}}_t+ I^{\textrm{mart}}_t$ for the finite variation part
\small
\begin{align*}
& I^{\textrm{f.v.}}_t  =\int_0^t\nabla\cdot\left(\Psi'_\d(\rho)[\Phi'(\rho)]^\frac{1}{2}\nabla\Theta_{\Phi,2}(\rho)\right)-\int_0^t\Psi''_\d(\rho)\abs{\nabla\Theta_{\Phi,2}(\rho)}^2+\a\int_0^t\nabla\cdot\left(\Psi'_\d(\rho)\nabla\rho\right)
\\ & \quad -\a\int_0^t\Psi_\d''(\rho)\abs{\nabla\rho}^2 +\frac{1}{2}\int_0^t\nabla\cdot\left(\Psi'_\d(\rho)F_1(x)\frac{[\sigma'(\rho)]^2}{[\Phi'(\rho)]^\frac{1}{2}}\nabla\Theta_{\Phi,2}(\rho)\right)
\\ &\quad  +\frac{1}{2}\int_0^t\nabla\cdot\left(\Psi'_\d(\rho)\sigma(\rho)\sigma'(\rho)F_2\right)+\frac{1}{2}\int_0^t\Psi''_\d(\rho)\frac{\sigma(\rho)\sigma'(\rho)}{[\Phi'(\rho)]^\frac{1}{2}}\nabla\Theta_{\Phi,2}(\rho)\cdot F_2+\frac{1}{2}\int_0^t\Psi_\d''(\rho)F_3(x)\sigma^2(\rho)
\\ & \quad -\nabla\cdot\left(\int_0^t\Psi'_\d(\rho)\nu(\rho)\right)+\int_0^t\Psi''_\d(\rho)\nu(\rho)\cdot\nabla\rho,
\end{align*}
\normalsize and for the martingale part
\[I ^{\textrm{mart}}_t = -\int_0^t\nabla\cdot\left(\Psi'_\d(\rho)\sigma(\rho)\dd\xi^F\right)+\int_0^t\Psi_\d''(\rho)[\Phi'(\rho)]^{-\nicefrac{1}{2}}\nabla\Theta_{\Phi,2}(\rho)\cdot\dd\xi^F.\]
Since $s>\frac{d}{2}+1$, it follows from the Sobolev embedding theorem that there exists $c\in(0,\infty)$ such that $\norm{f}_{L^\infty(\TT^d)}+\norm{\nabla f}_{L^\infty(\TT^d;\R^d)}\leq c\norm{f}_{H^s(\TT^d)}$.  It then follows from the facts that $\Psi_\d'$ is supported on $[\nicefrac{\d}{2},\infty)$, that $\Psi_\d''$ is supported on $[\nicefrac{\d}{2},\d]$, and that $\Phi'>0$ on $(0,\infty)$, the facts that $\a\in(0,1)$ and $T\in[1,\infty)$, and from Assumption~\ref{assume} and particularly \eqref{aa_121212}, H\"older's inequality, and Young's inequality that there exists $c\in(0,\infty)$ depending on $\d\in(0,1)$ such that
\begin{align}\label{4_400}
& \norm{I^\textrm{f.v.}_\cdot}_{W^{1,1}([0,T];H^{-s}(\TT^d))} \leq cT\left(1+\norm{\rho_0}_{L^1(\TT^d)}+\norm{\Phi'(\rho)\mathbf{1}_{\{\rho\geq\nicefrac{\d}{2}\}}}^2_{L^2(\TT^d)}+\norm{\sigma(\rho)\mathbf{1}_{\{\rho\geq\nicefrac{\d}{2}\}}}^2_{L^2(\TT^d)}\right)
\\ \nonumber &+ cT\left(\norm{\Theta_{\Phi,2}(\rho)}^2_{L^2(\TT^d\times[0,T])}+\norm{\nabla\Theta_{\Phi,2}(\rho)}^2_{L^2(\TT^d\times[0,T];\R^d)}+\a\norm{\nabla\rho}^2_{L^2(\TT^d\times[0,T];\R^d)}\right).
\end{align}
It then follows from $\a\in(0,1)$, $T\in[1,\infty)$, Remark~\ref{rem_rem}, the interpolation estimate \eqref{50_2} applied to $\Psi=\Theta_{\Phi,2}$, Proposition~\ref{prop_approx_est}, $p\in[2,\infty)$, H\"older's inequality, Young's inequality, and \eqref{4_400} that there exists $c\in(0,\infty)$ depending on $\d$ such that
\begin{equation}\label{4_402}
\E\left[\norm{I^\textrm{f.v.}_\cdot}_{W^{1,1}([0,T];H^{-s}(\TT^d))}\right] \leq cT\left(1+\E\left[\norm{\rho_0}^{p+m-1}_{L^1(\TT^d)}+\int_{\TT^d}\rho_0^p\right]\right).
\end{equation}
It remains to treat the martingale part.  Following a computation similar to Flandoli and Gatarek \cite[Lemma~2.1]{FlaGat1995}, the Burkholder-Davis-Gundy inequality (see, for example, \cite[Chapter~4, Theorem~4.1]{RevYor1999}), $s>\frac{d}{2}+1$, the fact that  $\Psi_\d''$ is supported on $[\nicefrac{\d}{2},\d]$, H\"older's inequality, Assumption~\ref{assume} and particularly \eqref{aa_5050} and \eqref{aa_121212}, and \eqref{4_00019} prove that for any $\beta\in(0,\nicefrac{1}{2})$ there exists $c\in(0,\infty)$ depending on $\d$ and $\beta$ such that the fractional Sobolev seminorm satisfies
\begin{align}\label{4_403}
& \E\left[\int_0^T\int_0^T\abs{s-t}^{-(1+2\beta)}\norm{I^\textrm{mart}_t-I^\textrm{mart}_s}_{H^{-s}(\TT^d)}^{2}\ds\dt\right]
\\ \nonumber & \leq c\E\left[\int_0^T\norm{\nabla\Theta_{\Phi,2}(\rho)\mathbf{1}_{\{\nicefrac{\d}{2}\leq \rho\leq \d\}}}^2_{L^2(\TT^d;\R^d)}+\norm{\sigma(\rho)}^2_{L^2(\TT^d)}\ds\right]
\\ \nonumber & \leq c\E\left[T+T\norm{\rho_0}_{L^1(\TT^d)}+\int_0^T\int_{\TT^d}\abs{\nabla\Theta_{\Phi,2}(\rho)}^2+\abs{\Theta_{\Phi,2}(\rho)}^2\right].
\end{align}
It follows similarly that there exists $c\in(0,\infty)$ such that
\begin{equation}\label{4_404}E\left[\norm{I^\textrm{mart}_\cdot}^2_{L^2([0,T];H^{-s}(\TT^d))}\right]\leq c\E\left[T+T\norm{\rho_0}_{L^1(\TT^d)}+\int_0^T\int_{\TT^d}\abs{\nabla\Theta_{\Phi,2}(\rho)}^2+\abs{\Theta_{\Phi,2}(\rho)}^2\right].\end{equation}
The interpolation estimate \eqref{50_2}, Proposition~\ref{prop_approx_est}, \eqref{4_403}, \eqref{4_404}, $p\in[2,\infty)$, H\"older's inequality, Young's inequality, and $T\in[1,\infty)$ then prove that
\begin{equation}\label{4_405}
\E\left[\norm{I^\textrm{mart}_\cdot}^2_{W^{\beta,2}([0,T];H^{-s}(\TT^d))}\right]  \leq cT\left(1+\E\left[\norm{\rho_0}^{p+m-1}_{L^1(\TT^d)}+\int_{\TT^d}\rho_0^p\right]\right).
\end{equation}
In combination \eqref{4_402}, \eqref{4_405}, and the embeddings $W^{\beta,2},W^{1,1}\hookrightarrow W^{\beta,1}$ for every $\beta\in(0,\nicefrac{1}{2})$ complete the proof.  \end{proof}

\begin{remark} The following definition and assumption are only used to treat initial data with finite entropy in the sense of Definition~\ref{def_entropy} below.  In Definition~\ref{def_sol}, the role of these assumptions is to guarantee the integrability of the stochastic flux $\sigma$.  In the model case $\sigma=\Phi^{\nicefrac{1}{2}}$ for $\Phi(\xi)=\xi^m$ this integrability amounts to $L^m_tL^m_x$-integrability for the solution, and in general a function with finite entropy is $L^m$-integrable only if $m=1$.  We emphasize that this assumption is satisfied by the model case $\sigma=\Phi^{\nicefrac{1}{2}}$ for $\Phi(\xi)=\xi^m$ for every $m\in[1,\infty)$.

\end{remark}

\begin{definition}\label{def_entropy} The space of nonnegative, $L^1(\TT^d)$-functions with finite entropy is the space
\[\Ent(\TT^d)=\left\{\rho\in L^1(\TT^d)\colon \rho\geq0\;\;\textrm{almost everywhere with}\;\;\int_{\TT^d}\rho\log(\rho)<\infty\right\}.\]
A function $\rho\colon\O\rightarrow L^1(\TT^d)\cap\Ent(\TT^d)$ satisfies $\rho\in L^1(\O;\Ent(\TT^d))$ if $\rho$ is $\F_0$-measurable with $\E\left[\norm{\rho}_{L^1(\TT^d)}\right]<\infty$ and $\E\left[\int_{\TT^d}\rho\log(\rho)\right]<\infty$.
\end{definition}

\begin{assumption}\label{assumption_10}  Let $\Phi,\sigma\in\C([0,\infty))$ and $\nu\in\C([0,\infty);\R^d)$ satisfy the following four assumptions.
\begin{enumerate}[(i)]
\item There exists $c\in(0,\infty)$ such that $\abs{\sigma(\xi)}\leq c\Phi^\frac{1}{2}(\xi)$ for every $\xi\in [0,\infty)$.
\item There exists $c\in(0,\infty)$ such that
\begin{equation}\label{50_5}\abs{\nu(\xi)}+\Phi'(\xi)\leq c(1+\xi+\Phi(\xi))\;\;\textrm{for every}\;\;\xi\in[0,\infty).\end{equation}
\item We have $\nabla\cdot F_2 = 0$.
\item We have that $\log(\Phi)$ is locally integrable on $[0,\infty)$.
\end{enumerate}
\end{assumption}

\begin{prop}\label{ent_dis_est}  Let $\xi^F$, $\Phi$, $\sigma$, and $\nu$ satisfy Assumptions~\ref{assumption_noise}, \ref{assume}, \ref{assume_3}, and \ref{assumption_10} for some $p\in[2,\infty)$, let $\a\in(0,1)$, let $T\in[1,\infty)$, let $\rho_0\in L^m(\O;L^1(\TT^d))\cap L^1(\O;\Ent(\TT^d))$ be $\F_0$-measurable, and let $\rho$ be a solution of \eqref{4_0} in the sense of Definition~\ref{def_approx}.  Then there exists $c\in(0,\infty)$ independent of $\a$ and $T$ such that, for $\Psi_\Phi\colon[0,\infty)\rightarrow\R$ the unique function satisfying $\Psi_\Phi(0)=0$ with $\Psi_\Phi'(\xi)=\log(\Phi(\xi))$,
\begin{align}\label{4_019}
& \E\left[\sup_{t\in[0,T]}\int_{\TT^d}\Psi_\Phi(\rho(x,t))\right] +\E\left[\int_0^T\int_{\TT^d}\abs{\nabla\Phi^\frac{1}{2}(\rho)}^2\right]+\a\E\left[\int_0^T\int_{\TT^d}\frac{\Phi'(\rho)}{\Phi(\rho)}\abs{\nabla\rho}^2\right]
 \\ \nonumber & \leq cT\left(1+\E\left[\norm{\rho_0}^m_{L^1(\TT^d)}+\int_{\TT^d}\Psi_\Phi(\rho_0)\right]\right).
\end{align}

\end{prop}

\begin{proof}  For every $\d\in(0,1)$ let $\Psi_{\Phi,\d}\colon[0,\infty)\rightarrow\R$ be the unique smooth function satisfying $\Psi_{\Phi,\d}(0)=0$ and $\Psi_{\Phi,\d}'(\xi)=\log(\Phi(\xi)+\d)$.  It follows from It\^o's formula---which is justified similarly to \cite[Proposition~7.7]{FehGes2020} using \cite[Theorem~3.1]{Kry2013}---the nonnegativity of $\rho$, and the distributional equality
\begin{equation}\label{4_02525}\nabla\Phi^\frac{1}{2}(\rho)=\frac{\Phi'(\rho)}{2\Phi^\frac{1}{2}(\rho)}\nabla\rho\;\;\textrm{on}\;\;\TT^d\times(0,\infty)\times[0,T],\end{equation}
that, almost surely for every $t\in[0,T]$,
\begin{align}\label{4_025}
& \left.\int_{\TT^d}\Psi_{\Phi,\d}(\rho(x,s))\right|_{s=0}^{s=t}  = -\int_0^t\int_{\TT^d}\frac{4\Phi(\rho)}{\Phi(\rho)+\d}\abs{\nabla\Phi^\frac{1}{2}(\rho)}^2-\a\int_0^t\int_{\TT^d}\frac{\Phi'(\rho)}{\Phi(\rho)+\d}\abs{\nabla\rho}^2
\\ \nonumber & \quad +\int_0^t\int_{\TT^d}\frac{\Phi'(\rho)}{\Phi(\rho)+\d}\nu(\rho)\cdot\nabla\rho+ \int_0^t\int_{\TT^d}\frac{2\Phi^\frac{1}{2}(\rho)\sigma(\rho)}{\Phi(\rho)+\d}\nabla\Phi^\frac{1}{2}(\rho)\cdot\dd\xi^F
\\ \nonumber & \quad + \frac{1}{2}\int_0^t\int_{\TT^d}\frac{\Phi'(\rho)}{\Phi(\rho)+\d}\left(\sigma(\rho)\sigma'(\rho)F_2\cdot \nabla\rho+F_3\sigma^2(\rho)\right).
\end{align}
The third term on the righthand side of \eqref{4_025} vanishes in exact analogy with \eqref{15_1}.    It follows from $\sigma\leq c\Phi^\frac{1}{2}$, the Burkholder-Davis-Gundy inequality (see, for example, \cite[Chapter~4, Theorem~4.1]{RevYor1999}), H\"older's inequality, $\nicefrac{\xi}{1+\xi}\leq 1$ for every $\xi\in[0,\infty)$, and the definition of $\xi^F$ that there exists $c\in(0,\infty)$ such that
\begin{align}\label{4_05}
\E\left[\sup_{t\in[0,T]}\abs{ \int_0^t\int_{\TT^d}\frac{2\Phi^\frac{1}{2}(\rho)\sigma(\rho)}{\Phi(\rho)+\d}\nabla\Phi^\frac{1}{2}(\rho)\cdot\dd\xi^F}\right] & \leq c\E\left[\int_0^T\left(\int_{\TT^d}\frac{2\Phi^\frac{1}{2}(\rho)\sigma(\rho)}{\Phi(\rho)+\d}\abs{\nabla\Phi^\frac{1}{2}(\rho)}\right)^2\right]^\frac{1}{2}
\\ \nonumber & \leq c \E\left[\int_0^T\int_{\TT^d}\frac{4\Phi(\rho)}{\Phi(\rho)+\d}\abs{\nabla\Phi^\frac{1}{2}(\rho)}^2\right]^\frac{1}{2}.
\end{align}
For the final two term on the righthand side of \eqref{4_025}, it follows from $\sigma\leq c\Phi^\frac{1}{2}$, \eqref{4_00019}, \eqref{50_5}, and the boundedness of $F_3$ that there exists $c\in(0,\infty)$ such that, for every $\d\in(0,1)$,
\begin{equation}\label{4_4646}\abs{\int_0^T\int_{\TT^d}\frac{\Phi'(\rho)\sigma^2(\rho)}{\Phi(\rho)+\d}F_3}\leq c\int_0^T\int_{\TT^d}\Phi'(\rho)\leq c\left(T+T\norm{\rho_0}_{L^1(\TT^d)}+\int_0^T\int_{\TT^d}\Phi(\rho)\right),\end{equation}
and for $\Theta_{\sigma,\d}\colon[0,\infty)\rightarrow[0,\infty)$ the unique function satisfying
\[\Theta_{\sigma,\d}(0)=0\;\;\textrm{and}\;\;\Theta_{\sigma,\d}'(\xi)=\frac{\Phi'(\xi)\sigma(\xi)\sigma'(\xi)}{\Phi(\xi)+\d},\]
it follows from $\nabla\cdot F_2=0$ that
\begin{equation}\label{4_4444}  \int_0^T\int_{\TT^d}\frac{\Phi'(\rho)\sigma(\rho)\sigma'(\rho)}{\Phi(\rho)+\d}F_2\cdot\nabla\rho=-\int_0^T\int_{\TT^d}\Theta_{\sigma,\d}(\rho)\nabla\cdot F_2=0. \end{equation}
Returning to \eqref{4_025}, it follows from \eqref{4_05}, \eqref{4_4646}, \eqref{4_4444}, Young's inequality, and $T\in[1,\infty)$ that, for some $c\in(0,\infty)$ independent of $\d\in(0,1)$, $\a\in(0,1)$, and $T\in[1,\infty)$,
\begin{align}\label{4_07}
& \E\left[\sup_{t\in[0,T]}\int_{\TT^d}\Psi_{\Phi,\d}(\rho(x,t))\right] +\E\left[\int_0^T\int_{\TT^d}\frac{4\Phi(\rho)}{\Phi(\rho)+\d}\abs{\nabla\Phi^\frac{1}{2}(\rho)}^2\right]+\a\E\left[\int_0^T\int_{\TT^d}\frac{\Phi'(\rho)}{\Phi(\rho)+\d}\abs{\nabla\rho}^2\right]
\\ \nonumber & \leq cT\left(1+\E\left[\norm{\rho_0}_{L^1(\TT^d)}+\int_{\TT^d}\Psi_{\Phi,\d}(\rho_0)+\int_0^T\int_{\TT^d}\Phi(\rho)\right]\right).
\end{align}
Since $\{\Phi(\rho)=0\}=\{\Phi^\frac{1}{2}(\rho)=0\}$, it follows from Stampacchia's lemma (see \cite[Chapter~5, Exercises~17,18]{Eva2010}) that
\[\int_0^T\int_{\TT^d}\mathbf{1}_{\{\Phi(\rho)=0\}}\abs{\nabla\Phi^\frac{1}{2}(\rho)}^2 = 0.\]
Fatou's lemma proves that, after passing to the limit $\d\rightarrow 0$ in \eqref{4_07}, for some $c\in(0,\infty)$,
\begin{align*}
& \E\left[\sup_{t\in[0,T]}\int_{\TT^d}\Psi_\Phi(\rho(x,t))\right] +\a\E\left[\int_0^T\int_{\TT^d}\abs{\nabla\Phi^\frac{1}{2}(\rho)}^2\right]+\a\E\left[\int_0^T\int_{\TT^d}\frac{\Phi'(\rho)}{\Phi(\rho)}\abs{\nabla\rho}^2\right]
\\ \nonumber & \leq cT\left(1+\E\left[\norm{\rho_0}_{L^1(\TT^d)}+\int_{\TT^d}\Psi_{\Phi,\d}(\rho_0)+\int_0^T\int_{\TT^d}\Phi(\rho)\right]\right).
\end{align*}
The interpolation estimate \eqref{50_2} applied to $\Psi=\Phi^\frac{1}{2}$ integrated in time, \eqref{aa_0}, and $T\in[1,\infty)$ prove that there exists $c\in(0,\infty)$ such that
\begin{align*}
& \E\left[\sup_{t\in[0,T]}\int_{\TT^d}\Psi_\Phi(\rho(x,t))\right] +\E\left[\int_0^T\int_{\TT^d}\abs{\nabla\Phi^\frac{1}{2}(\rho)}^2\right]+\a\E\left[\int_0^T\int_{\TT^d}\frac{\Phi'(\rho)}{\Phi(\rho)}\abs{\nabla\rho}^2\right]
 \\ \nonumber & \leq cT\left(1+\E\left[\norm{\rho_0}_{L^1(\TT^d)}+\norm{\rho_0}^m_{L^1(\TT^d)}+\int_{\TT^d}\Psi_\Phi(\rho_0)\right]\right).
\end{align*}
Estimate \eqref{4_019} now follows from $T\in[1,\infty)$, $m\in[1,\infty)$, and Young's inequality, which completes the proof.  \end{proof}

\begin{remark}  The condition that $\nabla\cdot F_2=0$ in Proposition~\ref{ent_dis_est} is not necessary.  Returning to \eqref{4_4444}, it is only necessary that for all $\d\in(0,1)$ sufficiently small the functions $\Theta_{\sigma,\d}$ satisfy, for every $\d\in(0,1)$ sufficiently small, for some $c\in(0,\infty)$,
\begin{equation}\label{eq_cond_1}\abs{\Theta_{\sigma,\d}(\xi)\nabla\cdot F_2(x)}\leq c(1+\xi+\Phi(\xi))\;\;\textrm{for every}\;\;\xi\in[0,\infty)\;\;\textrm{and}\;\;x\in\TT^d.\end{equation}
The term \eqref{4_4444} can then be estimated identically to \eqref{4_4646}.  However, in the model case that $\Phi(\xi)=\xi^m$ and $\sigma(\xi)=\Phi^{\nicefrac{1}{2}}(\xi)$ condition \eqref{eq_cond_1} requires $m\in(1,\infty)$, and therefore excludes the Dean--Kawasaki case unless $\nabla\cdot F_2=0$.  \end{remark}

\subsection{Existence of solutions to \eqref{1_0} for a smooth and bounded $\sigma$}\label{sec_exist_approx}

In this section, we establish the existence of solutions to the equation
\begin{align}\label{8_0} \dd\rho = \Delta\Phi(\rho)\dt-\nabla\cdot(\sigma(\rho)\dd\xi^F+\nu(\rho)\dt)+\frac{1}{2}\nabla\cdot\left(F_1(x)[\sigma'(\rho)]^2\nabla\rho+\sigma(\rho)\sigma'(\rho)\nabla\rho\cdot F_2\right)\dt\end{align}
for nonnegative initial data in $L^{m+p-1}(\O;L^1(\TT^d))\cap L^p(\O;L^p(\TT^d))$ and for nonlinearities $\Phi$, $\sigma$, and $\nu$ that satisfy Assumptions~\ref{assume} and \ref{assume_3}.  We prove the existence in Proposition~\ref{prop_approx_exist}.  In Proposition~\ref{prop_approx_kinetic}, we derive the kinetic formulation of \eqref{8_0},  and we show that the solution constructed in Proposition~\ref{prop_approx_exist} is a stochastic kinetic solution of \eqref{8_0} in the sense of Definition~\ref{def_sol}.

\begin{prop}\label{prop_approx_exist}  Let $\xi^F$, $\Phi$, $\sigma$, and $\nu$ satisfy Assumptions~\ref{assumption_noise}, \ref{assume}, and \ref{assume_3} for some $p\in[2,\infty)$, let $\a\in(0,1)$, and let $\rho_0\in L^{m+p-1}(\O;L^1(\TT^d))\cap L^p(\O;L^p(\TT^d))$ be nonnegative and $\F_0$-measurable.  Then there exists a solution of \eqref{4_0} in the sense of Definition~\ref{def_approx}.  Furthermore, the solution satisfies the estimates of Proposition~\ref{prop_approx_est}.\end{prop}

\begin{proof}  Let $\{\Phi_n\}_{n\in\N}$ be smooth, bounded, and nondecreasing with $\Phi_n(0)=0$, and let $\Phi_n$ and $\Phi'_n$ converge locally uniformly as $n\rightarrow\infty$ to $\Phi$ and $\Phi'$ on $[0,\infty)$, and let $\{\nu_n\}_{n\in\N}$ be smooth approximations of $\nu$ that converge locally uniformly to $\nu$ as $n\rightarrow\infty$.   Let $\{e_k\}_{k\in\N}$ be an orthonormal $L^2(\TT^d)$-basis that is an orthogonal $H^1(\TT^d)$-basis.  For every $K\in\N$ let $\xi^{F,K}$ denote the finite-dimensional noise $\xi^{F,K}=\sum_{k=1}^Kf_k(x)B^k_t$, and for every $M\in\N$ let $\Pi_M\colon L^2(\TT^d\times[0,T])\rightarrow L^2(\TT^d\times[0,T])$ be the projection map defined by $\Pi_Mg(x,t)=\sum_{k=1}^Mg_k(t)e_k$ for $g_k(t)=\int_{\TT^d}g(x,t)e_k(x)\dx$, and let $L^2_M = \Pi_M(L^2(\TT^d\times[0,T]))$.  Then, the projected equation
\begin{align*}
 \dd\rho & =\Pi_M\left(\Delta\Phi_n(\rho)\dt+\a\Delta\rho\dt-\nabla\cdot(\sigma(\rho)\dd\xi^{F,K}-\nabla\cdot\nu_n(\rho)\dt)\right)
 \\ & \quad +\Pi_M\left(\frac{1}{2}\nabla\cdot\left(F^K_1[\sigma'(\rho)]^2\nabla\rho+\sigma(\rho)\sigma'(\rho)\nabla\rho\cdot F^K_2\right)\dt\right),
\end{align*}
posed in $L^2(\O;L^2_M)$ for $F^K_1=\sum_{k=1}^Kf_k^2$ and $F_2^K=\sum_{k=1}^Kf_k\nabla f_k$ is equivalent to a finite-dimensional system of It\^o equations.  Since the $\Phi_n$, $\sigma$, and $\nu_n$ are smooth, bounded functions, the system has a unique strong solution (see, for example, \cite[Chapter~9, Theorem~2.1]{RevYor1999}).  We then pass first to the limit $M\rightarrow\infty$, then $K\rightarrow\infty$, and then $n\rightarrow\infty$ using simplified version of Theorem~\ref{thm_exist} below, relying on simplified versions of Proposition~\ref{prop_approx_est}, Proposition~\ref{prop_approx_time}, and the Aubin-Lions-Simon Lemma \cite{Aub1963,Lio1969,Sim1987}.  The $L^p$-continuity is a consequence of It\^o's formula \cite[Theorem~3.1]{Kry2013}, and the nonnegativity is a consequence of applying It\^o's formula to the negative part $\min(0,\rho)$ of the solution like was done for the positive part in Proposition~\ref{prop_approx_est}.  For a similar argument in this simplified setting see, for example, \cite[Proposition~5.4]{DarGes2020}.  \end{proof}

\begin{prop}\label{prop_approx_kinetic}  Let $\xi^F$, $\Phi$, $\sigma$, and $\nu$ satisfy Assumptions~\ref{assumption_noise}, \ref{assume}, and \ref{assume_3} for some $p\in[2,\infty)$, let $\a\in(0,1)$, and let $\rho_0\in L^{m+p-1}(\O;L^1(\TT^d))\cap L^p(\O;L^p(\TT^d))$ be nonnegative and $\F_0$-measurable.  Let $\rho$ be a solution of \eqref{4_0} in the sense of Definition~\ref{def_approx} and let $\chi\colon\TT^d\times\R\times[0,T]\rightarrow\{0,1\}$ be the kinetic function $\chi(x,\xi,t)=\mathbf{1}_{\{0<\xi<\rho(x,t)\}}$.  Then $\rho$ is a stochastic kinetic solution in the sense that, almost surely for every $\psi\in\C^\infty_c(\TT^d\times\R)$ and $t\in[0,T]$,\small
\begin{align}\label{4_100}
& \int_{\TT^d}\int_\R\chi(x,\xi,t)\psi = \int_{\TT^d}\overline{\chi}(\rho_0)\psi-\int_0^t\int_{\TT^d}\Phi'(\rho)\nabla\rho\cdot (\nabla\psi)(x,\rho)-\a\int_0^t\int_{\TT^d}\nabla\rho\cdot(\nabla\psi)(x,\rho)
\\ \nonumber & \quad -\frac{1}{2}\int_0^t\int_{\TT^d}F_1[\sigma'(\rho)]^2\nabla\rho\cdot(\nabla\psi)(x,\rho) -\frac{1}{2}\int_0^t\int_{\TT^d}\sigma(\rho)\sigma'(\rho)F_2\cdot(\nabla\psi)(x,\rho)
\\ & \nonumber \quad -\int_0^t\int_{\TT^d}(\partial_\xi\psi)(x,\rho)\Phi'(\xi)\abs{\nabla\rho}^2-\a\int_0^t\int_{\TT^d}(\partial_\xi\psi)(x,\rho)\abs{\nabla\rho}^2
\\ \nonumber & \quad +\frac{1}{2}\int_0^t\int_{\TT^d}(\partial_\xi\psi)(x,\rho)\sigma(\rho)\sigma'(\rho)\nabla\rho\cdot F_2+\frac{1}{2}\int_0^T\int_{\TT^d} (\partial_\xi\psi)(x,\rho)F_3\sigma^2(\rho)
\\ & \nonumber \quad -\int_0^t\int_{\TT^d}\psi(x,\rho)\nabla\cdot\nu(\rho)-\int_0^t\int_{\TT^d}\psi(x,\rho)\nabla\cdot\left(\sigma(\rho)\dd\xi^F\right),
\end{align}
\normalsize where the derivatives are interpreted according to Remark~\ref{rem_derivative} and where the kinetic measure is defined by \eqref{4_700} below.  \end{prop}

\begin{proof}  Let $S\colon \R\rightarrow\R$ be a smooth, bounded function and let $\psi\in\C^\infty(\TT^d)$.  It\^o's formula--- which is justified similarly to \cite[Proposition~7.7]{FehGes2020} using \cite[Theorem~3.1]{Kry2013}---applied to the composition $S(\rho)$ implies that, almost surely for every $t\in[0,T]$,\small
\begin{align*}
& \int_{\TT^d}S(\rho(x,t))\psi(x) = \int_{\TT^d}S(\rho_0)\psi-\int_0^t\int_{\TT^d}S'(\rho)\Phi'(\rho)\nabla\rho\cdot \nabla\psi-\a\int_0^t\int_{\TT^d}S'(\rho)\nabla\rho\cdot\nabla\psi
\\ & \nonumber \quad  -\frac{1}{2}\int_0^t\int_{\TT^d}S'(\rho)[\sigma'(\rho)]^2\nabla\rho\cdot\nabla\psi-\frac{1}{2}\int_0^t\int_{\TT^d}S'(\rho)\sigma(\rho)\sigma'(\rho)F_2\cdot\nabla\psi
\\ & \nonumber \quad -\int_0^t\int_{\TT^d}S''(\rho)\psi\Phi'(\rho)\abs{\nabla\rho}^2-\a\int_0^t\int_{\TT^d}S''(\rho)\psi\abs{\nabla\rho}^2-\int_0^t\int_{\TT^d}S'(\rho)\nabla\cdot\nu(\rho)
\\ \nonumber & \quad -\int_0^t\int_{\TT^d}S'(\rho)\psi(x)\nabla\cdot\left(\sigma(\rho)\dd\xi^F\right)+\frac{1}{2}\int_0^t\int_{\TT^d}S''(\rho)\psi\sigma\sigma'(\rho)\nabla\rho\cdot F_2+\frac{1}{2}\int_0^t\int_{\TT^d}S''(\rho)\psi F_3\sigma^2(\rho).
\end{align*}
\normalsize After defining the smooth function $\Psi_S(x,\xi)=\psi(x)S'(\xi)$, almost surely for every $t\in[0,T]$,\small
\begin{align*}
& \int_{\TT^d}\int_\R\chi(x,\xi,t)\Psi_S = \int_{\TT^d}\overline{\chi}(\rho_0)\Psi_S-\int_0^t\int_{\TT^d}\Phi'(\rho)\nabla\rho\cdot (\nabla_x\Psi_S)(x,\rho)-\a\int_0^t\int_{\TT^d}\nabla\rho\cdot(\nabla_x\Psi_S)(x,\rho)
\\ \nonumber & -\frac{1}{2}\int_0^t\int_{\TT^d}[\sigma'(\rho)]^2\nabla \rho\cdot(\nabla_x\Psi_S)(x,\rho)-\frac{1}{2}\int_0^t\int_{\TT^d}\sigma(\rho)\sigma'(\rho)F_2\cdot (\nabla_x\Psi_S)(x,\rho)
\\ & \nonumber -\int_0^t\int_{\TT^d}(\partial_\xi\Psi_S)(x,\rho)\Phi'(\rho)\abs{\nabla\rho}^2-\a\int_0^t\int_{\TT^d}(\partial_\xi\Psi_S)(x,\rho)\abs{\nabla\rho}^2-\int_0^t\int_{\TT^d}\Psi_S(x,\rho)\nabla\cdot\nu(\rho)
\\ \nonumber & -\int_0^t\int_{\TT^d}\Psi_S(x,\rho)\nabla\cdot\left(\sigma(\rho)\dd\xi^F\right)+\frac{1}{2}\int_0^t\int_{\TT^d} (\partial_\xi\Psi_S)(x,\rho)\sigma(\rho)\sigma'(\rho)\nabla\rho\cdot F_2+\frac{1}{2}\int_0^t\int_{\TT^d}(\partial_\xi\Psi_S)(x,\rho)F_3\sigma^2(\rho).
\end{align*}
\normalsize Since linear combinations of functions of the type $\Psi_S$ are dense in the space $\C^\infty_c(\TT^d\times\R)$, this completes the proof that $\rho$ satisfies \eqref{4_100}.  The kinetic measure corresponding to the solution $\rho$ in Proposition~\ref{prop_approx_kinetic} is the measure
\begin{equation}\label{4_700}q = \delta_0(\xi-\rho)\Phi'(\xi)\abs{\nabla\rho}^2+\a\delta_0(\xi-\rho)\abs{\nabla\rho}^2=\delta_0(\xi-\rho)\left(\abs{\nabla\Theta_{\Phi,2}(\rho)}^2+\a\abs{\nabla\rho}^2\right).\end{equation}
It follows from the definitions, Assumption~\ref{assume}, and Proposition~\ref{prop_approx_est} that $q$ is a finite measure that satisfies the conditions of Definition~\ref{def_measure}, which completes the proof.\end{proof}

\subsection{Existence of stochastic kinetic solutions to \eqref{1_0}}\label{sec_exist}

In this section, we will prove the existence of stochastic kinetic solutions to the equation
\begin{align}\label{6_0}\dd \rho = \Delta\Phi(\rho)\dt-\nabla\cdot\left(\sigma(\rho)\dd\xi^F+\nu(\rho)\dt\right)+\frac{1}{2}\nabla\cdot\left(F_1[\sigma'(\rho)]^2\nabla\rho+\sigma(\rho)\sigma'(\rho) F_2\right)\dt,\end{align}
with initial data in the space $L^{m+p-1}(\O;L^1(\TT^d))\cap L^p(\O;L^p(\TT^d))$ and nonlinearities $\Phi$, $\sigma$, and $\nu$ satisfying Assumptions~\ref{assume_1} and \ref{assume}.  The essential difficulty is that we do not have a stable $W^{\beta,1}_tH^{-s}_x$-estimate for the approximate solutions constructed in Proposition~\ref{prop_approx_exist}.  We have only the stable $W^{\beta,1}_tH^{-s}_x$-estimate of Proposition~\ref{prop_approx_time} for the functions $\Psi_\d(\rho)$ defined in Definition~\ref{def_cutoff}.  It is for this reason that we introduce in Definition~\ref{def_metric_1} a new metric on $L^1([0,T];L^1(\TT^d))$ based on the nonlinear approximations $\Psi_\d$.   The corresponding metric topology is equal to the usual strong norm topology on $L^1([0,T];L^1(\TT^d))$ (see  Lemma~\ref{lem_equivalent_top}), and with respect to this metric we prove that the tightness of the $\Psi_\d(\rho)$ in law implies the tightness of the approximate solutions $\rho$ in law (see Definition~\ref{def_approx_sol} and Proposition~\ref{prop_tight}).  In Proposition~\ref{prop_mart_tight} we prove the tightness of the martingale terms of the equation.  In Theorem~\ref{thm_exist} we prove the existence of a probabilistically weak solution, and we then use the pathwise uniqueness of Theorem~\ref{thm_unique} and Lemma~\ref{lem_weak_conv} to prove the existence of a probabilistically strong solution.  Corollaries~\ref{cor_exist} and \ref{cor_exist_1} extend these results to initial data with finite entropy and $L^1$-initial data respectively.

\begin{lem}\label{lem_nonlinear}  Let $\sigma$ satisfy Assumption~\ref{assume}.  Then there exists a sequence $\{\sigma_n\}_{n\in\N}$ that satisfies Assumption~\ref{assume_3} for every $n\in\N$, that satisfies Assumption~\ref{assume} uniformly in $n\in\N$, and that satisfies $\sigma_n\rightarrow\sigma$ in $\C^1_{\textrm{loc}}((0,\infty))$ as $n\rightarrow\infty$.  \end{lem}

\begin{proof}  The proof follows by constructing smooth and bounded approximations by convolution using Assumption~\ref{assume}.  \end{proof}

\begin{definition}\label{def_metric_1}  For every $\d\in(0,1)$ let $\Psi_\d$ be as in Definition~\ref{def_cutoff}.  Let $D\colon L^1([0,T];L^1(\TT^d))\times L^1([0,T];L^1(\TT^d))\rightarrow[0,\infty)$ be defined by
\begin{equation}\label{def_metric}D(f,g) = \sum_{k=1}^\infty 2^{-k}\left(\frac{\norm{\Psi_{\nicefrac{1}{k}}(f)-\Psi_{\nicefrac{1}{k}}(g)}_{L^1([0,T];L^1(\TT^d))}}{1+\norm{\Psi_{\nicefrac{1}{k}}(f)-\Psi_{\nicefrac{1}{k}}(g)}_{L^1([0,T];L^1(\TT^d))}}\right).\end{equation}

\end{definition}

\begin{lem}\label{lem_equivalent_top}  The function $D$ defined by \eqref{def_metric} is a metric on $L^1([0,T];L^1(\TT^d))$.  Furthermore, the metric topology determined by $D$ is equal to strong norm topology on $L^1([0,T];L^1(\TT^d))$.  \end{lem}

\begin{proof}  The fact that $D$ is a metric follows from the fact that $\Psi_{\nicefrac{1}{k}}(g)=\Psi_{\nicefrac{1}{k}}(f)$ for every $k\in\N$ if and only if $f=g$, and from the fact that $f(t)=\nicefrac{t}{1+t}$ is concave.  To see that the two topologies coincide, it is sufficient to prove that they determine the same convergent sequences.  This follows from the fact that there exists $c\in(0,\infty)$ such that, for every $k\in\N$,
\begin{align*}
\norm{\Psi_{\nicefrac{1}{k}}(f)-\Psi_{\nicefrac{1}{k}}(g)}_{L^1([0,T];L^1(\TT^d))} & \leq \norm{f-g}_{L^1([0,T];L^1(\TT^d))}
\\ & \leq \norm{\Psi_{\nicefrac{1}{k}}(f)-\Psi_{\nicefrac{1}{k}}(g)}_{L^1([0,T];L^1(\TT^d))}+\frac{cT}{k},
\end{align*}
which completes the proof.  \end{proof}

\begin{definition}\label{def_approx_sol}  Let $\xi^F$, $\Phi$, $\sigma$, and $\nu$ satisfy Assumptions~\ref{assumption_noise} and \ref{assume} for some $p\in[2,\infty)$, let $T\in[1,\infty)$, let $\{\sigma_n\}_{n\in\N}$ be as in Lemma~\ref{lem_nonlinear}, and let $\rho_0\in  L^{m+p-1}(\O;L^1(\TT^d))\cap L^p(\O;L^p(\TT^d))$ be nonnegative and $\F_0$-measurable.  For every $n\in\N$ and $\a\in(0,1)$ let $\rho^{\a,n}$ be the stochastic kinetic solution of
\begin{align*}
\dd\rho^{\a,n} & =\D\Phi(\rho^{\a,n})\dt+\a\Delta\rho^{\a,n}\dt-\nabla\cdot(\sigma_n(\rho^{\a,n})\dd \xi^F+\nu(\rho)\dt)
\\ & \quad +\frac{1}{2}\nabla\cdot\left([\sigma_n'(\rho^{\a,n})]^2\nabla\rho^{\a,n}+\sigma_n(\rho)\sigma_n'(\rho)F_2\right)\dt
\end{align*}
in $\TT^d\times(0,T)$ with initial data $\rho_0$ constructed in Proposition~\ref{prop_approx_kinetic}.  \end{definition}

\begin{prop}\label{prop_tight}  Let $\xi^F$, $\Phi$, $\sigma$, and $\nu$ satisfy Assumptions~\ref{assumption_noise} and \ref{assume} for some $p\in[2,\infty)$, let $\{\sigma_n\}_{n\in\N}$ be as in Lemma~\ref{lem_nonlinear}, and let $\rho_0\in  L^{m+p-1}(\O;L^1(\TT^d))\cap L^p(\O;L^p(\TT^d))$ be nonnegative and $\F_0$-measurable.  Then the laws of the $\{\rho^{\a,n}\}_{\a\in(0,1),n\in\N}$ are tight on $L^1([0,T];L^1(\TT^d))$ in the strong norm topology.  \end{prop}

\begin{proof}  According to Lemma~\ref{lem_equivalent_top}, it is equivalent to prove that the laws are tight on the space $L^1([0,T];L^1(\TT^d))$ equipped with the metric $D$ defined in \eqref{def_metric}.  Since it follows from Definition~\ref{def_cutoff} that there exists $c\in(0,\infty)$ independent of $k\in\N$ such that $\Psi'_{\nicefrac{1}{k}}(\xi)\leq c$ for every $\xi\in[0,\infty)$, it follows from Corollary~\ref{cor_frac_sob} that if $\Phi$ satisfies \eqref{5_00} then there exists $c\in(0,\infty)$ such that, for every $\a\in(0,1)$, $n,k\in\N$,
\begin{equation}\label{6_21} \E\left[\norm{\Psi_{\nicefrac{1}{k}}(\rho^{\a,n})}_{L^1([0,T];W^{1,1}(\TT^d))}\right]\leq cT\left(1+\E\left[\norm{\rho_0}^{p+m-1}_{L^1(\TT^d)}+\int_{\TT^d}\rho_0^p\right]\right),\end{equation}
and if $\Phi$ satisfies \eqref{5_0} then for every $\beta\in(0,\nicefrac{2}{q}\wedge 1)$ there exists $c\in(0,\infty)$ such that, for every $\a\in(0,1)$, $n,k\in\N$,
\begin{equation}\label{6_22}\E\left[\norm{\Psi_{\nicefrac{1}{k}}(\rho^{\a,n})}_{L^1([0,T];W^{\beta,1}(\TT^d))}\right]\leq cT\left(1+\E\left[\norm{\rho_0}^{p+m-1}_{L^1(\TT^d)}+\int_{\TT^d}\rho_0^p\right]\right).\end{equation}
Finally, it follows from Proposition~\ref{prop_approx_est} and $\Psi_{\nicefrac{1}{k}}(\xi)\leq \xi$ for every $\xi\in[0,\infty)$ and $k\in\N$ that, for every $\a\in(0,1)$ and $n,k\in\N$,
\begin{equation}\label{6_23} \E\left[\norm{\Psi_{\nicefrac{1}{k}}(\rho^{\a,n})}_{L^\infty([0,T];L^1(\TT^d))}\right]\leq\norm{\rho_0}_{L^1(\TT^d)}.\end{equation}
Let $s>\frac{d}{2}+1$.  The compact embeddings of $W^{1,1}(\TT^d),W^{\beta,1}(\TT^d)\hookrightarrow L^1(\TT^d)$, the continuous embedding of $L^1(\TT^d)\hookrightarrow H^{-s}(\TT^d)$, the Aubins-Lions-Simon lemma \cite{Aub1963,Lio1969,Sim1987} and specifically Simon \cite[Corollary~5]{Sim1987}, and estimates \eqref{6_21}, \eqref{6_22}, and \eqref{6_23} and Proposition~\ref{prop_approx_time} prove that, for every $k\in\N$, the laws $\{\Psi_{\nicefrac{1}{k}}(\rho^{\a,n})\}_{\a\in(0,1),n\in\N}$ are tight on $L^1([0,T];L^1(\TT^d))$ in the strong topology.

It follows from Definition~\ref{def_cutoff} that the maps $F_k\colon L^1([0,T];L^1(\TT^d))\rightarrow L^1([0,T];L^1(\TT^d))$ defined for every $k\in\N$ by $F_k(\rho)=\Psi_{\nicefrac{1}{k}}(\rho)$ are continuous in the strong topology.  Let $\ve\in(0,1)$ and for every $k\in\N$ the tightness in law of the $\{\Psi_{\nicefrac{1}{k}}(\rho^{\a,n})\}_{\a\in(0,1),n\in\N}$ proves that there exists a compact set $C_k\subseteq L^1([0,T];L^1(\TT^d))$ such that for every $\a\in(0,1)$ and $n\in\N$ we have $\P[\Psi_{\nicefrac{1}{k}}(\rho^{\a,n})\notin C_k]<\nicefrac{\ve}{2^k}$.   Let $D_k=F^{-1}_k(C_k)$ and let $K=\cap_{k=1}^\infty D_k$.    Then for every $\a\in(0,1)$ and $n\in\N$ we have $\P[\rho^{\a,n}\notin K]<\ve$, and it follows from Definition~\ref{def_metric_1} that $K$ is sequentially compact and hence compact in $L^1([0,T];L^1(\TT^d))$ with respect to the topology determined by the metric $D$ defined in \eqref{def_metric}.  Lemma~\ref{lem_equivalent_top} completes the proof, since this topology is equivalent to the strong norm topology.  \end{proof}

\begin{prop}\label{prop_mart_tight}  Let $\xi^F$, $\Phi$, $\sigma$, and $\nu$ satisfy Assumptions~\ref{assumption_noise} and \ref{assume} for some $p\in[2,\infty)$, let $\rho_0\in  L^{m+p-1}(\O;L^1(\TT^d))\cap L^p(\O;L^p(\TT^d))$ be nonnegative and $\F_0$-measurable, and let the solutions $\{\rho^{\a,n}\}_{\a\in(0,1),n\in\N}$ be as in Definition~\ref{def_approx_sol}.  Then for every $\psi\in\C^\infty_c(\TT^d\times(0,\infty))$ and $\gamma\in(0,\nicefrac{1}{2})$ the laws of the martingales
\[M^{\a,n,\psi}_t = \int_0^t\int_{\TT^d}\psi(x,\rho^{\a,n})\nabla\cdot\left(\sigma_n(\rho^{\a,n})\dd\xi^F\right),\]
are tight on $\C^\gamma([0,T])$.
\end{prop}

\begin{proof}  For every $n\in\N$ let $\Psi_n$ be the unique smooth and bounded function on $\TT^d\times(0,\infty)$ satisfying $\Psi_n(x,0)=0$ for every $x\in\TT^d$ and $(\partial_\xi\Psi_n)(x,\xi)=\psi(x,\xi)\sigma_n'(\xi)$.  We then observe using the definition of $M^{\a,n,\psi}_t$, the $H^1$-regularity of $\rho^{\a,n}$, and the chain rule that, almost surely for every $t\in[0,T]$,
\begin{align}\label{6_24}
M^{\a,n,\psi}_t & = \int_0^t\int_{\TT^d}\nabla\left(\Psi_n(x,\rho^{\a,n})\right)\cdot\dd\xi^F-\int_0^t\int_{\TT^d}(\nabla_x\Psi_n)(x,\rho^{\a,n})\cdot\dd\xi^F
\\ \nonumber & \quad +\int_0^t\int_{\TT^d}\psi(x,\rho^{\a,n})\sigma_n(\rho^{\a,n})\nabla\cdot\dd\xi^F.
\end{align}
After integrating by parts in the first term on the righthand side of \eqref{6_24},
\begin{align}\label{6_25}
M^{\a,n,\psi}_t & = -\int_0^t\int_{\TT^d}\Psi_n(x,\rho^{\a,n})\nabla\cdot\dd\xi^F-\int_0^t\int_{\TT^d}(\nabla_x\Psi_n)(x,\rho^{\a,n})\cdot\dd\xi^F
\\ \nonumber & \quad +\int_0^t\int_{\TT^d}\psi(x,\rho^{\a,n})\sigma_n(\rho^{\a,n})\nabla\cdot\dd\xi^F,
\end{align}
and it follows from Lemma~\ref{lem_nonlinear} and $\psi\in\C^\infty_c(\TT^d\times(0,\infty))$ that the $\Psi_n$ and $(\nabla_x\Psi_n)$ are uniformly bounded in $n\in\N$.  Furthermore, since Lemma~\ref{lem_nonlinear} proves that the $\sigma_n$ are uniformly bounded on the support of $\psi$, it follows from the Burkholder-Davis-Gundy inequality (see, for example, \cite[Chapter~4, Theorem~4.1]{RevYor1999}), Proposition~\ref{prop_approx_est}, Lemma~\ref{lem_nonlinear}, the definition of $\xi^F$, and \eqref{6_25} that for every $r\in(0,\infty)$ there exists $c\in(0,\infty)$ depending on $\psi$ and $r$ such that, for every $s\leq t\in[0,T]$, $\a\in(0,1)$, and $n\in\N$,
\[\E\left[\abs{M^{\a,n,\psi}_t-M^{\a,n,\psi}_s}^r\right]\leq c\abs{s-t}^\frac{r}{2}.\]
After choosing $r\in(2,\infty)$ such that $\gamma<\nicefrac{1}{2}-\nicefrac{1}{r}$, the claim follows from the quantified version of Kolmogorov's continuity criterion.  See, for example, Friz and Victoir \cite[Corollary A.11]{FriVic2010}.  \end{proof}

\begin{lem}\label{lem_weak_conv}  Let $(\O,\F,\P)$ be a probability space and let $\overline{X}$ be a complete separable metric space.  Then a sequence $\{X_n\colon\O\rightarrow \overline{X}\}$ of $\overline{X}$-valued random variables on $\O$ converges in probability, as $n\rightarrow\infty$, if and only if for every pair of sequences $\{(n_k,m_k)\}_{k=1}^\infty$ satisfying $n_k,m_k\rightarrow\infty$ as $k\rightarrow\infty$ there exists a further subsequence $\{(n_{k'},m_{k'})\}_{k'=1}^\infty$ satisfying $n_{k'},m_{k'}\rightarrow\infty$ as $k'\rightarrow\infty$ such that the joint laws of $(X_{n_{k'}},X_{m_{k'}})_{k'\in\N}$ converge weakly, as $k'\rightarrow\infty$, to a probability measure $\mu$ on $\overline{X}\times\overline{X}$ satisfying $\mu(\{(x,y)\in\overline{X}\times \overline{X}\colon x=y\})=1$.
\end{lem}

\begin{proof} The proof can be found in Gy\"ongy and Krylov \cite[Lemma~1.1]{GyoKry1996}. \end{proof}

\begin{thm}\label{thm_exist} Let $\xi^F$, $\Phi$, $\sigma$, and $\nu$ satisfy Assumptions~\ref{assumption_noise} and \ref{assume} for some $p\in[2,\infty)$ and let $\rho_0\in L^{m+p-1}(\O;L^1(\TT^d))\cap L^p(\O;L^p(\TT^d))$ be nonnegative and $\F_0$-measurable.    Then there exists a stochastic kinetic solution of \eqref{6_0} in the sense of Definition~\ref{def_sol}.  Furthermore, the solution satisfies the estimates of Proposition~\ref{prop_approx_est}.\end{thm}

\begin{proof}  The proof is organized as follows.  We will first use the tightness and Skorokhod representation theorem to extract an almost surely convergent subsequence of solutions on an auxiliary probability space.  We will then characterize the martingale terms appearing in the limit as stochastic integrals, which will allow us to prove that the limiting function is a stochastic kinetic solution after we prove that the limit is strongly $L^1$-continuous in time.  To conclude the existence of a probabilistically strong solution, we use the uniqueness of Theorem~\ref{thm_unique} and Lemma~\ref{lem_weak_conv}.

\textbf{Tightness and the Skorokhod representation theorem.}  Let $\rho_0\in L^{m+p-1}(\O;L^1(\TT^d))\cap L^p(\O;L^p(\TT^d))$ be nonnegative and $\F_0$-measurable, let the solutions $\{\rho^{\a,n}\}_{\a\in(0,1),n\in\N}$ be defined by Definition~\ref{def_approx_sol}, for every $\psi\in\C^\infty_c(\TT^d\times(0,\infty))$ let the martingales $\{M^{\a,n,\psi}\}_{\a\in(0,1), n\in\N}$ be defined in Proposition~\ref{prop_mart_tight}, and let the measures $\{q^{\a,n}\}_{\a\in(0,1),n\in\N}$ be defined by
\[q^{\a,n}=\delta_0(\xi-\rho^{\a,n})\left(\abs{\nabla\Theta_{\Phi,2}(\rho^{\a,n})}^2+\a\abs{\nabla\rho^{\a,n}}^2\right).\]
It follows from Proposition~\ref{prop_approx_est}, Remark~\ref{rem_rem}, and Proposition~\ref{prop_approx_exist} that the $\{q^{\a,n}\}_{\a\in(0,1),n\in\N}$ are finite kinetic measures in the sense of Definition~\ref{def_measure}.  It then follows from \eqref{2_5000} that, for every $\a\in(0,1)$, $n\in\N$, $\psi\in\C^\infty_c(\TT^d\times(0,\infty))$, and $t\in[0,T]$, for the kinetic function $\chi^{\a,n}$ of $\rho^{\a,n}$,\small
\begin{align}\label{6_29}
& M^{\a,n,\psi}_t = -\left.\int_\R\int_{\TT^d}\chi^{\a,n}(x,\xi,r)\psi(x,\xi)\right|_{r=0}^{r=t}-\int_0^t\int_{\TT^d}\Phi'(\rho^{\a,n})\nabla\rho^{\a,n}\cdot(\nabla\psi)(x,\rho^{\a,n})
\\ \nonumber &-\a\int_0^t\int_{\TT^d}\nabla\rho^{\a,n}\cdot (\nabla\psi)(x,\rho^{\a,n})-\frac{1}{2}\int_0^t\int_{\TT^d}[\sigma_n'(\rho^{\a,n})]^2\nabla\rho^{\a,n}\cdot(\nabla\psi)(x,\rho^{\a,n})
\\ \nonumber & -\frac{1}{2}\int_0^t\int_{\TT^d}\sigma_n(\rho^{\a,n})\sigma_n'(\rho^{\a,n})F_2\cdot(\nabla\psi)(x,\rho^{\a,n})+\frac{1}{2}\int_0^t\int_{\TT^d}(\partial_\xi\psi)(x,\rho^{\a,n})\sigma(\rho^{\a,n})\sigma'(\rho^{\a,n})\nabla\rho^{\a,n}\cdot F_2
\\ \nonumber & +\frac{1}{2}\int_0^t\int_{\TT^d}F_3\sigma_n^2(\rho^{\a,n})(\partial_\xi\psi)(x,\rho^{\a,n})-\int_0^t\int_\R\int_{\TT^d}\partial_\xi\psi(x,\xi)\dd q^{\a,n}-\int_0^t\int_{\TT^d}\Psi(x,\rho)\nabla\cdot\nu(\rho).
\end{align}
\normalsize Let $s>\frac{d}{2}+1$ and fix a countable sequence $\{\psi_j\}_{j\in\N}$ that is dense in $\C^\infty_c(\TT^d\times(0,\infty))$ in the strong $H^s(\TT^d\times(0,\infty))$-topology.

For every $\a\in(0,1)$ and $n\in\N$ define the random variables
\[X^{\a,n} = (\rho^{\a,n},\nabla\Theta_{\Phi,p}(\rho^{\a,n}),\a\nabla\rho^{\a,n},q^{\a,n}, (M^{\a,n,\psi_j})_{j\in\N}),\]
taking values in the space
\[\overline{X} =  L^1(\TT^d\times[0,T])\times L^2(\TT^d\times[0,T];\R^d)^2\times \mathcal{M}(\TT^d\times\R\times[0,T])\times\C([0,T])^{\N},\]
where $\overline{X}$ is equipped with the product metric topology induced by the strong topology on $L^1(\TT^d\times[0,T])$, the weak topology on $L^2(\TT^d\times[0,T];\R^d)$, the weak topology on the space of nonnegative Borel measures $\mathcal{M}(\TT^d\times\R\times[0,T])$---which is the dual space of the space of continuous functions on $\TT^d\times\R\times[0,T]$ that vanish at infinity in the supremum norm---and the topology of component-wise convergence in the strong norm on $\C([0,T])^\N$ induced by the metric
\[D_C((f_k)_{k\in\N},(g_k)_{k\in\N})=\sum_{k=1}^\infty2^{-k}\frac{\norm{f_k-g_k}_{\C([0,T])}}{1+\norm{f_k-g_k}_{\C([0,T])}}.\]
We aim to apply Lemma~\ref{lem_weak_conv}.  Let $(\alpha_k,n_k)_{k\in\N}$ and $(\beta_k,m_k)_{k\in\N}$ be two subsequences satisfying $\alpha_k,\beta_k\rightarrow 0$ and $n_k,m_k\rightarrow\infty$ as $k\rightarrow\infty$, and consider the laws of
\[(X^{\alpha_k,n_k},X^{\beta_k,m_k},B)\;\;\textrm{on}\;\;\overline{Y}=\overline{X}\times\overline{X}\times\C([0,T])^\N,\]
for $B=(B^j)_{j\in\N}$.  It follows from Propositions~\ref{prop_approx_est}, \ref{prop_tight}, and \ref{prop_mart_tight} that the laws of $(X^{\a,n})_{\a\in(0,1),n\in\N}$ are tight on $\overline{X}$ and therefore by Prokhorov's theorem (see, for example, Billingsley \cite[Chapter~1, Theorem~5.1]{Bil1999}), after passing to a subsequence still denoted $k\rightarrow\infty$, there exists probability measure $\mu$ on $\overline{Y}$ such that, as $k\rightarrow\infty$, $(X^{\alpha_k,n_k},X^{\beta_k,m_k},B)\rightarrow\mu $ in law.  Since the space $\overline{X}$ is separable and hence so too is $\overline{Y}$, it follows from the Skorokhod representation theorem (see, for example, \cite[Chapter~1, Theorem~6.7]{Bil1999}) that there exists a probability space $(\tilde{\O},\tilde{\F},\tilde{\P})$ and $\overline{Y}$-valued random variables $(\tilde{Y}^k,\tilde{Z}^k,\tilde{\beta}^k)_{k\in\N}$ and $(\tilde{Y},\tilde{Z},\tilde{\beta})$ on $\tilde{\O}$ such that, for every $k$,
\begin{equation}\label{6_32}(\tilde{Y}^k,\tilde{Z}^k,\tilde{\beta}^k)=(X^{\alpha_k,n_k},X^{\beta_k,m_k},B)\;\;\textrm{in law on}\;\;\overline{Y},\end{equation}
such that
\begin{equation}\label{6_34}(\tilde{Y},\tilde{Z},\tilde{\beta})=\mu\;\;\textrm{in law on $\overline{Y}$,}\end{equation}
and such that, as $k\rightarrow\infty$,
\begin{equation}\label{6_35}\tilde{Y}^k\rightarrow \tilde{Y},\;\;\tilde{Z}^k\rightarrow \tilde{Z},\;\;\textrm{and}\;\;\tilde{\beta}^k\rightarrow\tilde{\beta}\;\;\textrm{almost surely in $\overline{X}$ and in $\C([0,T];\R^{2N_Kd})$ respectively.}\end{equation}
We will now show that $\tilde{Y}=\tilde{Z}$ almost surely on $\overline{X}$ with respect to $\tilde{\P}$.

It follows from \eqref{6_32} that for every $k$ there exists $\tilde{\rho}^k\in L^\infty(\O\times[0,T];L^1(\TT^d))$, $\tilde{G}_1^k, \tilde{G}^k_2\in L^2(\O\times[0,T];L^2(\TT^d;\R^d))$, $\tilde{q}^k\in L^1(\O;\mathcal{M}(\TT^d\times\R\times[0,T]))$, and $(\tilde{M}^{k,\psi_j})_{j\in\N})\in L^1(\O;\C([0,T])^\N)$  such that
\[\tilde{Y}^k = (\tilde{\rho}^k,\tilde{G}^k_1,\tilde{G}^k_2,\tilde{q}^k,(\tilde{M}^{k,\psi_j})_{j\in\N}).\]
We will first identify the vector fields $\tilde{G}^k_i$.  Since for every smooth and bounded function $\varphi:\R\rightarrow\R$ the map that takes $g\in L^1(\TT^d\times[0,T])$ to $\varphi(g)\in L^1(\TT^d\times[0,T])$ is continuous, it follows from \eqref{6_32} that for every smooth and bounded function $\varphi$ we have
\begin{equation}\label{20_0}\tilde{\E}\left[\int_0^T\int_{\TT^d}\varphi(\tilde{\rho}^k)\right]=\E\left[\int_0^T\int_{\TT^d}\varphi(\rho^{\alpha_k,n_k})\right].\end{equation}
In analogy with Lemma~\ref{lem_nonlinear}, after choosing an increasing sequence $\varphi_n$ that approximate $\Theta_{\Phi,p}$, it follows that, uniformly in $k$,
\[\tilde{\E}\left[\int_0^T\int_{\TT^d}\Theta_{\Phi,p}(\tilde{\rho}^k)\right]<\infty.\]
It then follows from \eqref{6_32} and \eqref{20_0} that, for every $\psi\in\C^\infty(\TT^d)$, $\tilde{A}\in\tilde{\F}$, and $A\in\F$,
\begin{equation}\label{20_1} \tilde{\E}\left[\int_0^T\int_{\TT^d}\left(\Theta_{\Phi,p}(\tilde{\rho}^k)\nabla\psi + \psi \tilde{G}^k_1\right)\mathbf{1}_{\tilde{A}}\right] =  \E\left[\int_0^T\int_{\TT^d}\left(\Theta_{\Phi,p}(\rho^{\alpha_k,n_k})\nabla\psi + \psi \nabla\Theta_{\Phi,p}(\rho^{\alpha_k,n_k})\right)\mathbf{1}_A\right]=0,\end{equation}
from which it follows $\tilde{\P}$-almost surely that $\tilde{G}^k_1=\nabla\Theta_{\Phi,p}(\tilde{\rho}^k)$.  In exact analogy with \eqref{20_1}, it follows that $\tilde{\P}$-almost surely $\tilde{G}^k_2=\a_k\nabla\tilde{\rho}^k$.  A virtually identical argument proves that the continuous paths $(\tilde{M}^{k,\psi_j})_{j\in\N}$ are $\tilde{\P}$-almost surely defined by $\tilde{\rho}^k$ as in \eqref{6_29}.  Finally, it follows from \eqref{6_32} that, for every $\psi\in\C^\infty_c(\TT^d\times[0,T])$ with
\[\norm{\nabla\psi}_{L^2([0,T];L^2(\TT^d;\R^d))}\leq \norm{\nabla\Theta_{\Phi,p}(\tilde{\rho}^k)}_{L^2([0,T];L^2(\TT^d;\R^d))},\]
and for every nonnegative $\phi\in\C^\infty_c(\TT^d\times[0,T]\times(0,\infty))$ and $\tilde{A}\in\tilde{F}$,
\begin{equation}\label{20_2}\tilde{\E}\left[\left(\int_0^T\int_\R\int_{\TT^d}\left(\nabla\Theta_{\Phi,p}(\tilde{\rho}^k)\cdot\nabla\psi\right)\phi\dx\dxi\dt-\int_0^T\int_\R\int_{\TT^d}\phi\dd \tilde{q}^k\right)\mathbf{1}_{\tilde{A}} \right]\leq 0,\end{equation}
from which it follows that the measure $\tilde{q}^k$ is $\tilde{\P}$-almost surely is a kinetic measure for $\tilde{\rho}^k$ in the sense of Definition~\ref{def_measure} that satisfies \eqref{2_500}.  Since Proposition~\ref{prop_approx_est} and \eqref{6_32} prove that $\tilde{\P}$-almost surely $\a_k\nabla\tilde{\rho}^k\rightharpoonup 0$ weakly in $L^2([0,T];L^2(\TT^d;\R^d))$ as $k\rightarrow\infty$, it follows that there exists $\tilde{\rho}\in L^1(\O\times[0,T];L^1(\TT^d))$ such that
\[\tilde{Y} = (\tilde{\rho},\nabla\Theta_{\Phi,p}(\tilde{\rho}),0,\tilde{q},(\tilde{M}^j)_{j\in\N}),\]
where a repetition of the arguments leading to \eqref{20_2} proves that the measure $\tilde{q}$ is a kinetic measure for $\tilde{\rho}$ in the sense of Definition~\ref{def_measure} that satisfies \eqref{2_500} (see \eqref{6_045} below).  It remains to characterize the paths $\tilde{M}^j$.

\textbf{The path $\tilde{\beta}$ is a Brownian motion.}  Let $\tilde{\beta}^k=(\tilde{\beta}^{k,j})_{j\in\N}$ and let $\tilde{\beta}=(\tilde{\beta}^j)_{j\in\N}$.  Let $F\colon\overline{Y}\rightarrow\R$ be a continuous function.  It follows from \eqref{6_32} that, for every $s\leq t\in[0,T]$, $j\in\N$, and $k$,
\begin{align}\label{6_37}
& \tilde{\E}\left[F\left(\tilde{Y}^k|_{[0,s]},\tilde{Z}^k|_{[0,s]}, \tilde{\beta}^k|_{[0,s]}\right)\left(\tilde{\beta}^{k,j}_t-\tilde{\beta}^{k,j}_s\right)  \right]
\\ \nonumber & =\E\left[F\left(X^{\alpha_k,n_k}|_{[0,s]},X^{\beta_k,m_k}|_{[0,s]}, B|_{[0,s]}\right)\left(B^j_t-B^j_s\right)\right]=0.
\end{align}
After passing to the limit $k\rightarrow\infty$ using the uniform integrability of the paths $\tilde{\beta}^{k,j}_t-\tilde{\beta}^{k,j}_s$ implied by Proposition~\ref{prop_mart_tight} and \eqref{6_32} in $k$, it follows from \eqref{6_35} and \eqref{6_37} that, for every $s\leq t\in[0,T]$ and $j\in\N$,
\begin{equation}\label{6_38}\tilde{\E}\left[F\left(\tilde{Y}|_{[0,s]},\tilde{Z}|_{[0,s]}, \tilde{\beta}|_{[0,s]}\right)\left(\tilde{\beta}^j_t-\tilde{\beta}^j_s\right)\right]=0.\end{equation}
An identical argument proves that, for every $i,j\in\N$ and $s\leq t\in[0,T]$,
\begin{equation}\label{6_39}\tilde{\E}\left[F\left(\tilde{Y}|_{[0,s]},\tilde{Z}|_{[0,s]}, \tilde{\beta}|_{[0,s]}\right)\left(\tilde{\beta}^i_t\tilde{\beta}^j_t-\tilde{\beta}^i_s\tilde{\beta}^j_s-\delta_{ij}(t-s)\right)\right]=0,\end{equation}
for the Kronecker delta $\delta_{ij}$ that is $1$ if $i=j$ and $0$ otherwise.  Since $\tilde{\P}$-almost surely $\tilde{\beta}^j\in\C([0,T])$ for every $j\in\N$, it follows from \eqref{6_38}, \eqref{6_39}, and Levy's characterization of Brownian motion (see, for example, \cite[Chapter~4, Theorem~3.6]{RevYor1999}) that $\tilde{\beta}^j$ is for every $j\in\N$ a one-dimensional Brownian motion with respect to the filtration $\mathcal{G}_t=\sigma(\tilde{Y}|_{[0,t]},\tilde{Z}|_{[0,t]}, \tilde{\beta}|_{[0,t]})$.  It follows from the continuity and uniform integrability of the Brownian motion in time that $\tilde{\beta}$ is a Brownian motion with respect to the augmented filtration $(\overline{\mathcal{G}}_t)_{t\in[0,T]}$ of $(\mathcal{G}_t)_{t\in[0,T]}$, by which we mean that $(\overline{\mathcal{G}}_t)_{t\in[0,T]}$ is the smallest complete, right-continuous filtration containing $(\mathcal{G}_t)_{t\in[0,T]}$.

\textbf{The paths $(\tilde{M}^j)_{j\in\N}$ are $\overline{\mathcal{G}}_t$-martingales.}  Let $F\colon\overline{Y}\rightarrow\R$ be a continuous function and let $j\in\N$.  It follows from \eqref{6_32} that, for every $s\leq t\in[0,T]$ and $k$,
\begin{align*}
& \tilde{\E}\left[F\left(\tilde{Y}^k|_{[0,s]},\tilde{Z}^k|_{[0,s]}, \tilde{\beta}^k|_{[0,s]}\right)\left(\tilde{M}^{k,\psi_j}_t-\tilde{M}^{k,\psi_j}_s\right)  \right]
\\ &=\E\left[F\left(X^{\alpha_k,n_k}|_{[0,s]},X^{\beta_k,m_k}|_{[0,s]}, B|_{[0,s]}\right)\left(M^{n_k,\alpha_k,\psi_j}_t-M^{n_k,\alpha_k,\psi_j}_s\right)\right]=0.
\end{align*}
After passing to the limit $k\rightarrow\infty$, using that Proposition~\ref{prop_mart_tight} and \eqref{6_32} prove that the $M^{n_k,\alpha_k,\psi_j}$ are uniformly bounded in $L^p(\O\times[0,T])$ for every $p\in[1,\infty)$ and hence uniformly integrable,
\begin{equation}\label{6_40}\tilde{\E}\left[F\left(\tilde{Y}|_{[0,s]},\tilde{Z}|_{[0,s]}, \tilde{\beta}|_{[0,s]}\right)\left(\tilde{M}^j_t-\tilde{M}^j_s\right)  \right]=0.\end{equation}
This proves that $(\tilde{M}^j)_{t\in[0,\infty)}$ satisfies the martingale property with respect to $(\mathcal{G}_t)_{t\in[0,T]}$.  It then follows from the continuity and uniform integrability of the $\tilde{M}^j$ that the $\tilde{M}^j$ are continuous martingales with respect to the augmentation $(\overline{\mathcal{G}}_t)_{t\in[0,T]}$, with respect to which the $\tilde{M}^j$ are measurable.

\textbf{The $(\tilde{M}^j)_{j\in\N}$ are stochastic integrals with respect to $\tilde{\beta}$.}  Let $F\colon\overline{Y}\rightarrow\R$ be a continuous function.  It follows from \eqref{6_32} that, for every $s\leq t$, $i\in\N$, and $k$,
\begin{align*}
& \tilde{\E}\left[F\left(\tilde{Y}^k|_{[0,s]},\tilde{Z}^k|_{[0,s]}, \tilde{\beta}^k|_{[0,s]}\right)\left(\tilde{M}^{k,\psi_j}_t\tilde{\beta}^{k,i}_t-\tilde{M}^{k,\psi_j}_s\tilde{\beta}^{k,i}_s-\int_s^t\int_{\TT^d}\psi_j(x,\tilde{\rho}^k)\nabla\cdot(\sigma_{n_k}(\tilde{\rho}^k)f_i)\right)\right]
\\ & =\E\left[F\left(X^{\alpha_k,n_k}|_{[0,s]},X^{\beta_k,m_k}|_{[0,s]}, B|_{[0,s]}\right)\cdot\ldots\right.
\\ &  \quad\quad\left.\ldots\cdot \left(M^{\alpha_k,n_k,\psi_j}_tB^i_t-M^{\alpha_k,n_k,\psi_j}_sB^i_s-\int_s^t\int_{\TT^d}\psi_j(x,\rho^{\alpha_k,n_k})\nabla\cdot(\sigma_{n_k}(\rho^{\alpha_k,n_k})f_i)\right)\right]=0.
\end{align*}
Since it follows from Proposition~\ref{prop_mart_tight} and \eqref{6_32} that the $\tilde{M}^{k,\psi_j}_t\tilde{\beta}^i_t$  are uniformly integrability in $k$ for every time $t\in[0,T]$, after passing to the limit $k\rightarrow\infty$,
\begin{equation}\label{6_41}\tilde{\E}\left[F\left(\tilde{Y}|_{[0,s]},\tilde{Z}|_{[0,s]}, \tilde{\beta}|_{[0,s]}\right)\left(\tilde{M}^j_t\tilde{\beta}^i_t-\tilde{M}^j_s\tilde{\beta}^i_s-\int_s^t\int_{\TT^d}\psi_j(x,\tilde{\rho})\nabla\cdot(\sigma(\tilde{\rho})f_i)\right)\right]=0.\end{equation}
We therefore conclude from \eqref{6_41} that, for every $i\in\N$,
\begin{equation}\label{6_42}\tilde{M}^j_t\tilde{\beta}^i_t-\int_0^t\int_{\TT^d}\psi_j(x,\tilde{\rho})\nabla\cdot(\sigma(\tilde{\rho})f_i)\;\;\textrm{is a $\mathcal{G}_t$-martingale.}\end{equation}
It then follows from the continuity of the process in time and the uniform integrability that the process \eqref{6_42} is also a continuous $\overline{\mathcal{G}}_t$-martingale.

A virtually identical proof shows that for every $j\in\N$ the process
\begin{equation}\label{6_43} (\tilde{M}^j_t)^2-\int_0^t\sum_{k=1}^\infty\left(\int_{\TT^d}\psi_j(x,\tilde{\rho})\nabla\cdot(\sigma(\tilde{\rho})f_k)\right)^2,\end{equation}
is a continuous $\overline{\mathcal{G}}_t$-martingale.  It then follows \eqref{6_42}, \eqref{6_43}, and an explicit calculation using the quadratic variation and covariation with the Brownian motion that, for every $j\in\N$ and $t\in[0,T]$,
\begin{equation}\label{6_44}\tilde{\E}\left[\left(\tilde{M}^j_t-\int_0^t\int_{\TT^d}\psi_j(x,\tilde{\rho})\nabla\cdot(\sigma(\tilde{\rho})\dd\tilde{\xi}^F)\right)^2                   \right]=0,\end{equation}
for the noise $\tilde{\xi}^F$ defined analogously to Assumption~\ref{assumption_noise} by the Brownian motion $\tilde{\beta}$ on $\tilde{\O}$.  It follows from Proposition~\ref{prop_mart_tight}, \eqref{6_40}, and \eqref{6_44} that the quadratic variation of the difference between the continuous $L^2$-bounded martingales $\tilde{M}^j$ and $\int_0^t\int_{\TT^d}\psi^j(x,\tilde{\rho})\nabla\cdot(\sigma(\tilde{\rho})\dd\tilde{\xi}^F)$ vanishes.  Hence, almost surely for every $j\in\N$ and $t\in[0,T]$,
\begin{equation}\label{6_45} \tilde{M}^j_t = \int_0^t\int_{\TT^d}\psi_j(x,\tilde{\rho})\nabla\cdot(\sigma(\tilde{\rho})\dd\tilde{\xi}^F).\end{equation}

\textbf{The kinetic measure.}  We will now show that the limiting measure $\tilde{q}$ is almost surely a kinetic measure for $\tilde{\rho}$.  It is a finite measure on $\TT^d\times\R\times[0,T]$, and therefore satisfies \eqref{def_5353}.  In Definition~\ref{def_measure}, the measurability is a consequence of the convergence.  The fact that almost surely
\begin{equation}\label{6_045}\delta_0(\xi-\tilde{\rho})\Phi'(\xi)\abs{\nabla\tilde{\rho}}^2=\delta_0(\xi-\tilde{\rho})\abs{\nabla\Theta_{\Phi,2}(\tilde{\rho})}^2\leq \tilde{q}\;\;\textrm{on}\;\;\TT^d\times(0,\infty)\times[0,T],\end{equation}
follows from the definition of $\tilde{q}$, weak lower-semicontinuity of the Sobolev norm, the strong convergence of the $\tilde{\rho}^k$ to $\tilde{\rho}$ in $L^1([0,T];L^1(\TT^d))$, and the weak convergence of the $\nabla\Theta_{\Phi,p}(\tilde{\rho}^k)$ to $\nabla\Theta_{\Phi,p}(\tilde{\rho})$ in $L^2([0,T];L^2(\TT^d;\R^d))$ from which it follows that the $\nabla\Theta_{\Phi,2}(\tilde{\rho}^k)$ also converge weakly to $\nabla\Theta_{\Phi,2}(\tilde{\rho})$ using Remark~\ref{rem_rem}.

\textbf{The integrability of $\sigma(\tilde{\rho})$}. The $L^2$-integrability of $\sigma(\tilde{\rho})$ is a consequence of the $L^2$-integrability of $\nabla\Theta_{\Phi,p}(\tilde{\rho})$, the $L^1$-integrability of $\tilde{\rho}$ using the fact that the $L^1_tL^1_x$-convergence and \eqref{4_00019} imply that $\norm{\tilde{\rho}(x,t)}_{L^1(\TT^d)}=\norm{\rho_0}_{L^1(\TT^d)}$ for almost every $t\in[0,T]$, \eqref{aa_5050}, and the interpolation estimate \eqref{50_2} applied to $\Psi=\Theta_{\Phi,p}$ and integrated in time.

\textbf{Recovering the equation almost everywhere in time.}  Let $\mathcal{A}\subseteq[0,T]$ denote the random set of atoms of the measure $\tilde{q}$ in time,
\[\mathcal{A} = \{t\in[0,T]\colon \tilde{q}(\{t\}\times\TT^d\times\R)\neq 0\}.\]
Since the measure $\tilde{q}$ is almost surely finite, the set $\mathcal{A}$ is almost surely at most countable.  It then follows from Lemma~\ref{lem_nonlinear}, \eqref{6_29}, \eqref{6_35}, \eqref{6_45}, the definition of the kinetic function and the kinetic function $\tilde{\chi}$ of $\tilde{\rho}$, and the compact support of $\psi_j$ on $\TT^d\times(0,\infty)$ that there almost surely exists a random set of full measure $\mathcal{C}\subseteq [0,T]\setminus\mathcal{A}$ such that, for every $t\in\mathcal{C}$ and $j\in\N$,
\begin{align}\label{6_46}
& \left.\int_\R\int_{\TT^d}\tilde{\chi}(x,\xi,r)\psi_j(x,\xi)\right|_{r=0}^{r=t} =-\int_0^t\int_{\TT^d}\Phi'(\tilde{\rho})\nabla\tilde{\rho}\cdot(\nabla\psi_j)(x,\tilde{\rho}) -\frac{1}{2}\int_0^t\int_{\TT^d}[\sigma'(\tilde{\rho})]^2\nabla\tilde{\rho}\cdot(\nabla\psi_j)(x,\tilde{\rho})
\\ \nonumber &-\frac{1}{2}\int_0^t\int_{\TT^d}\sigma(\tilde{\rho})\sigma'(\tilde{\rho})F_2\cdot(\nabla\psi_j)(x,\tilde{\rho})-\int_0^t\int_\R\int_{\TT^d}\partial_\xi\psi_j(x,\xi)\dd \tilde{q}- \int_0^t\int_{\TT^d}\psi_j(x,\tilde{\rho})\nabla\cdot(\sigma(\tilde{\rho})\dd\tilde{\xi}^F)
\\ \nonumber &-\int_0^t\int_{\TT^d}\psi_j(x,\tilde{\rho})\nabla\cdot\nu(\tilde{\rho})+\frac{1}{2}\int_0^t\int_{\TT^d}(\partial_\xi\psi_j)(x,\tilde{\rho})\sigma(\tilde{\rho})\sigma'(\tilde{\rho})\nabla\tilde{\rho}\cdot F_2+\frac{1}{2}\int_0^t\int_{\TT^d}(\partial_\xi\psi_j)(x,\tilde{\rho})F_3\sigma^2(\tilde{\rho}),
\end{align}
where there is no ambiguity in interpreting the integral with respect $\tilde{q}$ since $t\notin\mathcal{A}$.  We will now prove that the measure $\tilde{q}$ has no atoms in time, and that the function $\tilde{\rho}$ almost surely admits a representative taking values in $\C([0,T];L^1(\TT^d))$.

\textbf{Right- and left-continuous representatives of $\tilde{\rho}$.}  For every $j\in\N$ let
\[\langle \tilde{\chi},\psi_j\rangle_t = \int_\R\int_{\TT^d}\tilde{\chi}(x,\xi,t)\psi_j(x,\xi).\]
It then follows from \eqref{6_46} that, almost surely for every $t\in\mathcal{C}$ and $j\in\N$,
\begin{align}\label{6_47}
& \langle\tilde{\chi},\psi_j\rangle_t = \int_\R \int_{\TT^d}\overline{\chi}(\rho_0)\psi_j-\int_0^t\int_{\TT^d}\Phi'(\tilde{\rho})\nabla\tilde{\rho}\cdot(\nabla\psi_j)(x,\tilde{\rho}) -\frac{1}{2}\int_0^t\int_{\TT^d}[\sigma'(\tilde{\rho})]^2\nabla\tilde{\rho}\cdot(\nabla\psi_j)(x,\tilde{\rho})
\\ \nonumber &-\frac{1}{2}\int_0^t\int_{\TT^d}\sigma(\tilde{\rho})\sigma'(\tilde{\rho})F_2\cdot(\nabla\psi_j)(x,\tilde{\rho})-\int_0^t\int_\R\int_{\TT^d}\partial_\xi\psi_j(x,\xi)\dd \tilde{q}- \int_0^t\int_{\TT^d}\psi_j(x,\tilde{\rho})\nabla\cdot(\sigma(\tilde{\rho})\dd\tilde{\xi}^F)
\\ \nonumber &-\int_0^t\int_{\TT^d}\psi_j(x,\tilde{\rho})\nabla\cdot\nu(\tilde{\rho})+\frac{1}{2}\int_0^t\int_{\TT^d}(\partial_\xi\psi_j)(x,\tilde{\rho})\sigma(\tilde{\rho})\sigma'(\tilde{\rho})\nabla\tilde{\rho}\cdot F_2+\frac{1}{2}\int_0^t\int_{\TT^d}(\partial_\xi\psi_j)(x,\tilde{\rho})F_3\sigma^2(\tilde{\rho}).
\end{align}
Observe that every term on the righthand side of \eqref{6_47} is continuous in time except potentially the term involving the measure $\tilde{q}$, which may have discontinuities on the random set $\mathcal{A}$.  For every $j\in\N$ let $\tilde{q}_{\psi_j}$ be the measure on $\TT^d\times\R\times[0,T]$ defined by $\dd \tilde{q}_{\psi_j} = (\partial_\xi\psi_j)(x,\xi)\dd \tilde{q}$, and observe that by using standard properties of measures the functions $\tilde{Q}^{\pm}_{\psi_j}\colon[0,T]\rightarrow\R$ defined by
\begin{equation}\label{6_48} \tilde{Q}^+_{\psi_j}(t)=\tilde{q}_{\psi_j}(\TT^d\times\R\times[0,t])\;\;\textrm{and}\;\;\tilde{Q}^-_{\psi_j}(t)=\tilde{q}_{\psi_j}([0,t)\times\TT^d\times\R),\end{equation}
are almost surely right- and left-continuous and satisfy $\tilde{Q}^+_{\psi_j}(t)=\tilde{Q}^-_{\psi_j}(t)$ for every $t\in[0,T]\setminus\mathcal{A}$.  As a consequence, it follows from \eqref{6_47} and \eqref{6_48} that for every $j\in\N$ the functions $t\in\mathcal{C}\rightarrow\langle\tilde{\chi},\psi_j\rangle_t$ almost surely admit right- and left-continuous representatives $\langle\tilde{\chi},\psi_j\rangle^{\pm}_t$ defined on $[0,T]$.

Since the $\{\psi_j\}_{j\in\N}$ are dense on $L^2(\TT^d\times[0,\infty))$, it follows from the nonnegativity of the solutions and the definition of the kinetic function $\tilde{\chi}$ of $\tilde{\rho}$ that the functions $\tilde{\chi}^\pm$ defined by
\[\int_{\TT^d}\tilde{\chi}^{\pm}(x,\xi,t)\psi_j(x,\xi) = \langle \tilde{\chi},\psi_j\rangle^{\pm}_t\]
are almost surely weakly right- and left-continuous in $L^2(\TT^d\times\R)$ in time, and satisfy $\tilde{\chi}^{\pm}(x,\xi,t)=\tilde{\chi}(x,\xi,t)$ for every $t\in\mathcal{C}$.  Let $\tilde{Q}^{\pm}_\psi$ be defined analogously to \eqref{6_48} for every $\psi\in\C^\infty_c(\TT^d\times(0,\infty))$.  It then follows from the density of the $\{\psi_j\}_{j\in\N}$ in the $H^s$-norm, the Sobolev embedding theorem, and \eqref{6_47} that for every $\psi\in\C^\infty_c(\TT^d\times(0,\infty))$ there exists a subset of full probability such that, for every $t\in[0,T]$,\small
\begin{align}\label{6_49}
&\langle \tilde{\chi}^{\pm},\psi\rangle_t = \int_\R \int_{\TT^d}\overline{\chi}(\rho_0)\psi-\int_0^t\int_{\TT^d}\Phi'(\tilde{\rho})\nabla\tilde{\rho}\cdot(\nabla\psi)(x,\tilde{\rho}) -\frac{1}{2}\int_0^t\int_{\TT^d}[\sigma'(\tilde{\rho})]^2\nabla\tilde{\rho}\cdot(\nabla\psi)(x,\tilde{\rho})
\\ \nonumber &-\frac{1}{2}\int_0^t\int_{\TT^d}\sigma(\tilde{\rho})\sigma'(\tilde{\rho})F_2\cdot(\nabla\psi)(x,\tilde{\rho})-\tilde{Q}^{\pm}_{\psi}(t)- \int_0^t\int_{\TT^d}\psi(x,\tilde{\rho})\nabla\cdot(\sigma(\tilde{\rho})\dd\tilde{\xi}^F)
\\ \nonumber &-\int_0^t\int_{\TT^d}\psi_j(x,\tilde{\rho})\nabla\cdot\nu(\tilde{\rho})+\frac{1}{2}\int_0^t\int_{\TT^d}(\partial_\xi\psi)(x,\tilde{\rho})\sigma(\tilde{\rho})\sigma'(\tilde{\rho})\nabla\tilde{\rho}\cdot F_2+\frac{1}{2}\int_0^t\int_{\TT^d}(\partial_\xi\psi)(x,\tilde{\rho})F_3\sigma^2(\tilde{\rho}),
\end{align}
\normalsize and such that for every $t\in\mathcal{C}$ we have $\norm{\tilde{\rho}(\cdot,t)}_{L^1(\TT^d)}=\norm{\rho_0}_{L^1(\TT^d)}$.

\textbf{$L^1$-continuity in time.}  We aim to prove that there almost surely exist measurable functions $\tilde{\rho}^\pm\colon\TT^d\times[0,T]\rightarrow[0,\infty)$ such that $\tilde{\chi}^{\pm}=\mathbf{1}_{\{0<\xi<\tilde{\rho}^\pm\}}$.  To do this, we will first show that the $\tilde{\chi}^{\pm}$ are almost surely $\{0,1\}$-valued functions on $\TT^d\times\R\times[0,T]$.  After replacing \eqref{3_6} with the respectively right- and left-continuous versions of the kinetic functions and measures and using \cite[Chapter~0, Proposition~4.5]{RevYor1999} to justify differentiating the equality
\[\int_\R\int_{\TT^d}\tilde{\chi}^{\pm}+\tilde{\chi}^{\pm}-2(\tilde{\chi}^{\pm})^2\dx\dxi = 2\int_\R\int_{\TT^d}\tilde{\chi}^{\pm}(1-\tilde{\chi}^{\pm})\dx\dxi\]
after introducing the regularization used in the proof of Theorem~\ref{thm_unique}, it follows from Theorem~\ref{thm_unique}, the fact that $\tilde{\chi}^{\pm}(x,\xi,t)=\tilde{\chi}(x,\xi,t)$ almost surely for almost every time in $[0,T]$, the fact that the right- and left-continuity prove that $\tilde{\chi}^{\pm}$ preserve the $L^1$-norm, and \eqref{6_045} that, almost surely for every $t\in[0,T]$,
\[\int_\R\int_{\TT^d}\tilde{\chi}^{\pm}(x,\xi,t)(1-\tilde{\chi}^{\pm}(x,\xi,t))\dx\dxi\leq \int_\R\int_{\TT^d}\overline{\chi}(\rho_0)(1-\overline{\chi}(\rho_0))\dx\dxi=0.\]
Since the weak convergence implies that $0\leq \tilde{\chi}^{\pm}\leq 1$ almost everywhere, it follows that
\[\tilde{\chi}^{\pm}(x,\xi,t)(1-\tilde{\chi}^{\pm}(x,\xi,t))=0\;\;\textrm{almost everywhere.}\]
That is, almost surely,
\begin{equation}\label{6_50}\textrm{$\tilde{\chi}^{\pm}$ are $\{0,1\}$-valued on $\TT^d\times\R\times[0,T]$.}\end{equation}

We will now show that as a distribution $\partial_\xi\tilde{\chi}^{\pm}\leq 0$ on $\TT^d\times(0,\infty)\times[0,T]$, from which the claim follows.  Since for all times $t\in[0,T]$, with only a right or a left limit if $t=0$ or $t=T$ respectively, the fact that almost surely $\tilde{\chi}^\pm=\tilde{\chi}$ for almost every $t\in[0,T]$ implies that there almost surely exists a sequence $\{(t^+_k,t^-_k)\}_{k\in\N}$ of positive and negative numbers satisfying $t^+_k,t^-_k\rightarrow 0$ as $k\rightarrow\infty$ such that $\tilde{\chi}^\pm(x,\xi,t+t^\pm_k)=\tilde{\chi}(x,\xi,t+t^\pm_k)$ for every $k\in\N$.  It then follows from the respective right- and left-continuity of $\tilde{\chi}^\pm$ and properties of the kinetic function that, for all nonnegative $\alpha\in\C^\infty_c((0,\infty))$ and $\psi\in\C^\infty(\TT^d)$,
\begin{align*}
\int_\R\int_{\TT^d}\tilde{\chi}^{\pm}(x,\xi,t)\psi(x)\alpha'(\xi) & \geq \liminf_{k\rightarrow\infty}\int_\R\int_{\TT^d}\tilde{\chi}(x,\xi,t+t^{\pm}_k)\psi(x)\alpha'(\xi)
\\ & \geq \liminf_{k\rightarrow\infty}\int_\R\int_{\TT^d}\psi(x)\alpha(\tilde{\rho}(x,t+t^{\pm}_k)) \geq 0,
\end{align*}
from which we conclude from the density of linear combinations of functions of the type $\alpha(\xi)\psi(x)$ in $\C^\infty_c(\TT^d\times(0,\infty))$ that, as distributions,
\begin{equation}\label{6_51}\partial_\xi\tilde{\chi}^{\pm}(x,\xi,t)\leq 0\;\;\textrm{on}\;\;\TT^d\times(0,\infty)\times[0,T].\end{equation}
In combination \eqref{6_50} and \eqref{6_51} prove that the $\tilde{\chi}^\pm$ are kinetic functions in the sense that there exist $\tilde{\rho}^{\pm}\in L^1(\O\times[0,T];L^1(\TT^d))$ which almost surely satisfy $\tilde{\rho}^{\pm}(x,t)=\tilde{\rho}(x,t)$ for almost every $t\in[0,T]$ such that
\begin{equation}\label{6_52}\tilde{\chi}^{\pm}(x,\xi,t)=\mathbf{1}_{\{0<\xi<\tilde{\rho}^\pm(x,t)\}}.\end{equation}

To conclude that $\tilde{\rho}$ has an $L^1$-continuous representative, we will will prove that $\tilde{\rho}^+=\tilde{\rho}^-$.  Observe that it follows from \eqref{6_48} and \eqref{6_49} that, almost surely for every $\alpha\in\C^\infty_c((0,\infty))$ and $\psi\in\C^\infty(\TT^d)$,
\begin{equation}\label{5_53}\int_\R\int_{\TT^d}(\tilde{\chi}^+(x,\xi,t)-\tilde{\chi}^-(x,\xi,t))\psi(x)\alpha(\xi) = \tilde{Q}^-_{\alpha\psi}(t)-\tilde{Q}^+_{\alpha\psi}(t)=-\int_{\{t\}\times\TT^d\times\R}\alpha'(\xi)\psi(x)\dd \tilde{q}. \end{equation}
Fix a sequence of functions $\alpha_n\in\C^\infty_c((0,\infty))$ such that $0\leq \alpha_n\leq 1$, $\alpha_n(\xi)=1$ if $\nicefrac{1}{n}\leq \xi\leq n$, $\alpha_n(\xi)=0$ if $\xi<\nicefrac{1}{2n}$ or if $\xi>n+1$, and such that $\alpha_n'(\xi)\leq\nicefrac{c}{n}$ if $\nicefrac{1}{2n}<\xi<\nicefrac{1}{n}$ and $\alpha_n'(\xi)\leq c$ if $n<\xi<n+1$ for some $c\in(0,\infty)$ independent of $n$.  It follows from \eqref{6_52} and \eqref{5_53} that for every $\psi\in\C^\infty(\TT^d)$ there exists $c\in(0,\infty)$ such that, almost surely for every $t\in[0,T]$,
\[\abs{\int_{\TT^d}(p^+(x,t)-p^-(x,t))\psi(x)\dx}\leq c\liminf_{n\rightarrow\infty}\left(n\tilde{q}(\{t\}\times\TT^d\times[\nicefrac{1}{2n},\nicefrac{1}{n}])+\tilde{q}(\{t\}\times\TT^d\times[n,n+1])\right),\]
and therefore, using Proposition~\ref{prop_measure} and the finiteness of the measures $\tilde{q}$, almost surely for every $t\in[0,T]$ and $\psi\in\C^\infty(\TT^d)$,
\begin{equation}\label{6_54}
\abs{\int_{\TT^d}(p^+(x,t)-p^-(x,t))\psi(x)\dx}=0.
\end{equation}
It follows from \eqref{6_54} that, for every $j\in\N$, there exists a subset of full probability such that, for every $t\in[0,T]$,
\[\int_{\TT^d}(p^+(x,t)-p^-(x,t))\psi_j(x)\dx=0.\]
The density of the $\psi_j$ proves that almost surely $\tilde{\rho}^+=\tilde{\rho}^-$ in $L^1([0,T];L^1(\TT^d))$.  From \eqref{6_52} it follows that almost surely $\tilde{\chi}^+=\tilde{\chi}^-$ in $L^2(\TT^d\times\R\times[0,T])$, and we therefore conclude that $\tilde{\chi}^+=\tilde{\chi}^-$ is almost surely weakly $L^2$-continuous in time.

It remains only to show that the weak continuity of $\tilde{\chi}^+$ implies the strong continuity of $\tilde{\rho}^+$.  Let $t\in[0,T]$ and let $\{t_k\}_{k\in[0,\infty)}$ be a sequence in $[0,T]$ satisfying $t_k\rightarrow t$ as $k\rightarrow\infty$.  Properties of the kinetic function and the weak $L^2$-continuity of $\tilde{\chi}^+$ prove that
\begin{align*}
& \limsup_{k\rightarrow\infty}\int_{\TT^d}\abs{\tilde{\rho}^+(x,t)-\tilde{\rho}^+(x,t_k)}\dx = \limsup_{k\rightarrow\infty}\int_\R\int_{\TT^d}\abs{\tilde{\chi}^+(x,\xi,t)-\tilde{\chi}^+(x,\xi,t_k)}^2\dx\dxi
\\ & = \limsup_{k\rightarrow\infty}\int_\R\int_{\TT^d}\left(\tilde{\chi}^+(x,\xi,t)+\tilde{\chi}^+(x,\xi,t_k)-2\tilde{\chi}^+(x,\xi,t)\tilde{\chi}^+(x,\xi,t_k)\right)\dx\dxi=0,
\end{align*}
which completes the proof that $\tilde{\rho}^+$ is $L^1$-continuous in the strong topology, and therefore that $\tilde{\rho}$ has a representative taking values in $\C([0,T];L^1(\TT^d))$.  Furthermore, it follows from the continuity and \eqref{6_49} that the measure $\tilde{q}$ almost surely has no atoms in time, so that there is no ambiguity, for example, in interpreting integrals of the form $\int_0^t\int_\R\int_{\TT^d}(\partial_\xi\psi)(x,\xi)\dd \tilde{q}$.

\textbf{Conclusion.}  It follows from \eqref{6_045}, \eqref{6_49}, the weak $L^2$-continuity of $\tilde{\chi}^+$ and the strong $L^1$-continuity of $\tilde{\rho}^+$ that $\tilde{\rho}$ has a representative in $L^1(\O\times[0,T];L^1(\TT^d))$ that is a stochastic kinetic solution of \eqref{6_0} in the sense of Definition~\ref{def_sol} with respect to the Brownian motion $\tilde{\beta}$ and the filtration $(\overline{\mathcal{G}}_t)_{t\in[0,\infty)}$.  That is, we have shown the existence of a probabilistically weak solution.  It remains to show that there exists a probabilistically strong solution.

Returning to \eqref{6_32}, \eqref{6_34}, and \eqref{6_35}, it follows that there exists $\overline{\rho}\in L^\infty([0,T];L^1(\TT^d))$ such that
\[\tilde{Z} = (\overline{\rho},\nabla\Theta_{\Phi,p}(\overline{\rho}),0,\overline{q},(\overline{M}^j)_{j\in\N}).\]
A repetition of the above arguments proves that there almost surely exists a strongly $L^1$-continuous representative of $\overline{\rho}$ that is a stochastic kinetic solution of \eqref{6_0} with respect to the Brownian motion $\tilde{\beta}$ and the filtration $(\overline{\mathcal{G}}_t)_{t\in[0,\infty)}$ on $(\tilde{\O},\tilde{\F},\tilde{\P})$ in the sense of Definition~\ref{def_sol}.  The uniqueness of Theorem~\ref{thm_unique} proves that almost surely $\overline{\rho}=\tilde{\rho}$ in $L^1([0,T];L^1(\TT^d))$.  Returning to \eqref{6_32}, we conclude that along the subsequence $k\rightarrow\infty$ the joint laws of $(X^{\alpha_k,n_k},X^{\beta_k,m_k})$ restricted to $L^1([0,T];L^1(\TT^d))^2$ converge weakly to a measure $\mu$ on $L^1([0,T];L^1(\TT^d))^2$ satisfying the conditions of Lemma~\ref{lem_weak_conv}.

Returning to the original solutions $\{\rho^{\a,n}\}_{\a\in(0,1),n\in\N}$ defined on the original probability space $(\O,\F,\P)$, it follows from Lemma~\ref{lem_weak_conv} that, after passing to a subsequence $\alpha_k\rightarrow 0$, $n_k\rightarrow\infty$, there exists a random variable $\rho\in L^1(\O\times [0,T];L^1(\TT^d))$ such that the $\{\rho^{\alpha_k,n_k}\}_{k\in\N}$ converge to $\rho$ in probability.  Passing to a further subsequence, it follows that $\rho^{\a_k,n_k}\rightarrow\rho$ almost surely.  A simplified version of the above argument then proves that $\rho$ is a stochastic kinetic solution of \eqref{6_0} in the sense of Definition~\ref{def_sol} on $(\O,\F,\P)$.  Furthermore, since the $\{\rho^{\a,n}\}_{\a\in(0,1),n\in\N}$ are probabilistically strong solutions, it follows that $\rho$ is a probabilistically strong solution.  The estimates follow from the same argument and the weak lower semicontinuity of the Sobolev norm.  This completes the proof. \end{proof}

\begin{remark}  For $\F_0$-measurable initial data $\rho_0\in L^{m+p-1}(\O;L^1(\TT^d))\cap L^p(\O;L^p(\TT^d))$, as a consequence of Proposition~\ref{prop_approx_est}, Remark~\ref{rem_rem}, and Theorem~\ref{thm_exist}, the resulting measure $q$ is finite in the sense that
\[\E\left[q(\TT^d\times(0,\infty)\times[0,T])\right]<\infty.\]
In this case, therefore, condition \eqref{def_5353} can be replaced by the condition that the kinetic measure is a finite measure.\end{remark}

\begin{cor}\label{cor_exist}  Let $\xi^F$, $\Phi$, $\sigma$, and $\nu$ satisfy Assumptions~\ref{assumption_noise}, \ref{assume},and \ref{assumption_10} for some $p\in[2,\infty)$ and assume that $\rho_0\in L^m(\O;L^1(\TT^d))\cap L^1(\O;\Ent(\TT^d))$ is $\F_0$-measurable.  Then there exists a stochastic kinetic solution of \eqref{6_0} in the sense of Definition~\ref{def_sol}.  Furthermore, the solution satisfies the estimates of Proposition~\ref{ent_dis_est}.\end{cor}

\begin{proof} Let $\rho_0\in L^m(\O;L^1(\TT^d))\cap L^1(\O;\Ent(\TT^d))$ and for every $n\in\N$ let $\rho^n_0=\rho\wedge n$.  It then follows using the bounded entropy of $\rho_0$ in expectation and the fact that $\xi\log(\xi)$ is increasing on the set $[\nicefrac{1}{e},\infty)$ that
\begin{equation}\label{6_55}\sup_n\E\left[ \int_{\TT^d}\rho_0^n\log(\rho_0^n)\right]+\sup_n\E\left[\norm{\rho_0^n}^m_{L^1(\TT^d)}\right]<\infty\;\;\textrm{and}\;\;\lim_{n\rightarrow\infty}\norm{\rho^n_0-\rho_0}_{L^1(\TT^d)}=0.\end{equation}
Using Theorem~\ref{thm_exist} let $\{\rho^n\}_{n\in\N}$ be the unique stochastic kinetic solutions of \eqref{6_0} with initial data $\{\rho^n_0\}_{n\in\N}$.  The $L^1$-contraction of Theorem~\ref{thm_unique} proves that there exists a random variable $\rho\in L^1(\O\times[0,T];L^1(\TT^d))$ such that $\rho^n\rightarrow\rho$ strongly in $L^1(\O\times [0,T];L^1(\TT^d))$.  It follows from the $L^1$-estimate \eqref{4_00019}, Proposition~\ref{ent_dis_est}, the interpolation estimate \eqref{50_2}, and $\rho_0\in L^m(\O;L^1(\TT^d))\cap L^1(\O;\Ent(\TT^d))$ that $\Phi^\frac{1}{2}(\rho)\in L^2(\O\times[0,T];H^1(\TT^d))$.  It then follows from $\sigma\leq c\Phi^\frac{1}{2}$ that $\sigma(\rho)\in L^2(\O\times[0,T];L^2(\TT^d))$, and from $\abs{\nu}\leq (1+\xi+\Phi)$ that $\nu(\rho)\in L^1(\O;L^1(\TT^d\times[0,T]))$.  Finally, it follows from \eqref{50_7070} that the resulting kinetic measure---which is no longer globally integrable---decays at infinity in the sense of \eqref{def_5353}, it follows from the uniform boundedness of $\Phi'$ away from zero on compact subsets of $(0,\infty)$ and \eqref{50_7070} that $\rho$ satisfies the local regularity property \eqref{def_2500000}, and it follows from the weak continuity of the Sobolev norm that the kinetic measure satisfies \eqref{2_500}.  The proof that $\rho$ is $L^1$-continuous and a stochastic kinetic solution of \eqref{6_0} with initial data $\rho_0$ then follows from \eqref{6_55} and a simplified version of the above argument, since here everything is taking place on the original probability space.  The estimates follow from the same argument and the weak lower semicontinuity of the Sobolev norm.  This completes the proof. \end{proof}

\begin{cor}\label{cor_exist_1}  Let $\xi^F$, $\Phi$, $\sigma$, and $\nu$ satisfy Assumptions~\ref{assumption_noise} and \ref{assume} for some $p\in[2,\infty)$ and let $\rho_0\in L^1(\O;L^1(\TT^d))$ be nonnegative and $\F_0$-measurable.  Assume that, for some $c\in(0,\infty)$,
\begin{equation}\label{6_9191}\sigma^2(\xi)+\abs{\nu(\xi)}\leq c(1+\xi)\;\;\textrm{for every}\;\;\xi\in[0,\infty).\end{equation}
Then there exists a stochastic kinetic solution of \eqref{6_0} in the sense of Definition~\ref{def_sol}.  \end{cor}

\begin{proof} Let $\rho_0\in L^1(\O;L^1(\TT^d))$ and for every $n\in\N$ let $\rho^n_0=\rho\wedge n$.  Since it follows almost surely that
\[\lim_{n\rightarrow\infty}\norm{\rho^n_0-\rho_0}_{L^1(\TT^d)}=0,\]
the proof is now identical to Corollary~\ref{cor_exist}, since the entropy dissipation estimate of Proposition~\ref{ent_dis_est} was used only to prove the $L^2$-integrability of $\sigma(\rho)$, and the $L^1$-integrability of $\nu(\rho)$.  Precisely, it follows from \eqref{4_00019} and \eqref{6_9191} that there exists $c\in(0,\infty)$ such that, almost surely,
\[\norm{\sigma(\rho)}^2_{L^2([0,T];L^2(\TT^d))}+\norm{\nu(\rho)}_{L^1(\TT^d\times[0,T];\R^d)}\leq cT(1+\norm{\rho_0}_{L^1(\TT^d)}).\]
The local regularity property \eqref{def_2500000} and the vanishing of the kinetic measure at infinity \eqref{def_5353} are a consequence of estimate \eqref{50_7070}, which only requires the $L^1$-integrability of $\rho$ and the $L^2$-integrability of $\sigma(\rho)$.  This completes the proof. \end{proof}

\section{The well-posedness of \eqref{1_000}}\label{sec_gen}

In this section, we will extend the well-posedness theory to equations of the form
\begin{align}\label{9_0} \dd\rho =\Delta \Phi(\rho)\dt - \nabla\cdot\left(\sigma(\rho)\circ \dd\xi^F+\nu(\rho)\dt\right)+\phi(\rho)\dd\xi^G+\lambda(\rho)\dt\;\;\textrm{in}\;\;\TT^d\times(0,T),\end{align}
where the noise $\xi^G$ is of the form $\xi^G=\sum_{k=1}^\infty g_kW^k_t$ for continuous functions $g_k$ on $\TT^d$ and for independent Brownian motions $W^k$, and where the nonlinearities $\Phi$, $\sigma$, and $\nu$ satisfying Assumptions~\ref{assume_1} and \ref{assume} above.  The essential point in this section is that the pathwise (almost sure) contraction property of Theorem~\ref{thm_unique} will no longer be true if $\phi$ is nonzero or $\lambda$ is not monotone, and in general we only expect to obtain the $L^1$-contraction in expectation.

\textbf{The uniqueness of stochastic kinetic solutions to \eqref{9_0}}.  We  define a stochastic kinetic solution of \eqref{9_0} in Definition~\ref{def_gen_sol}.  We introduce the assumptions on $\phi$ and $\lambda$ in Assumption~\ref{assume_5}, and we prove the uniqueness of stochastic kinetic solutions in Theorem~\ref{thm_gen_unique}.  To do this, we control the kinetic measure at zero using Proposition~\ref{prop_measure_1} and control certain divergences at infinity using Lemma~\ref{lem_partition}.

\begin{assumption}\label{assumption_noise_1}  Assume that the noise $\xi^F$ and the initial data $\rho_0$ satisfy Assumption~\ref{assumption_noise}.  Let $\{W^k\}_{k\in\N}$ be independent one-dimensional Brownian motions that are independent of the $\{B^k\}_{k\in\N}$ and that are defined on the same probability space $(\O,\F,\P)$ with respect to the same filtration $(\F_t)_{t\in[0,\infty)}$ and let $g_k\in\C(\TT^d)$ for every $k\in\N$.  Assume that the sum $G_1=\sum_{k=1}^\infty g_k^2$ is continuous on $\TT^d$ and define $\xi^G = \sum_{k=1}^\infty g_k(x)W^k_t$. \end{assumption}

\begin{definition}\label{def_gen_sol}  Let $\rho_0\in L^1(\O;L^1(\TT^d))$ be nonnegative and $\F_0$-measurable.  A \emph{stochastic kinetic solution} of \eqref{9_0} is a nonnegative, almost surely continuous $L^1(\TT^d)$-valued, $\F_t$-predictable function $\rho\in L^1(\O\times[0,T];L^1(\TT^d))$ that satisfies the following three properties.
\begin{enumerate}[(i)]
\item \emph{Preservation of mass}:  for every $t\in[0,T]$,
\begin{equation}\label{defn_1}\E\left[\norm{\rho(\cdot,t)}_{L^1(\TT^d)}\right]=\E\left[\norm{\rho_0}_{L^1(\TT^d)}\right]+\E\left[\int_0^t\int_{\TT^d}\lambda(\rho)\right].\end{equation}
\item \emph{Integrability}:  we have that
\[\sigma(\rho), \phi(\rho)\in L^2(\O;L^2(\TT^d\times[0,T]))\;\;\textrm{and}\;\;\nu(\rho), \lambda(\rho)\in L^1(\O;L^1(\TT^d\times[0,T])).\]
\item \emph{Local regularity}:  for every $K\in\N$,
\[ [(\rho\wedge K)\vee\nicefrac{1}{K}]\in L^2(\O;L^2([0,T];H^1(\TT^d))).\]
\end{enumerate}
Furthermore, there exists a kinetic measure $q$ that satisfies the following three properties.
\begin{enumerate}[(i)]
\setcounter{enumi}{3}
\item \emph{Regularity}: almost surely as nonnegative measures,
\[\delta_0(\xi-\rho)\Phi'(\xi)\abs{\nabla\rho}^2\leq q\;\;\textrm{on}\;\;\TT^d\times(0,\infty)\times[0,T].\]
\item \emph{Vanishing at infinity}:  we have that
\[\lim_{M\rightarrow\infty}\E\left[q(\TT^d\times[M,M+1]\times[0,T])\right]=0.\]
\item \emph{The equation}: for every $\psi\in \C^\infty_c(\TT^d\times(0,\infty))$, almost surely for every $t\in[0,T]$,
\begin{align}\label{defn_6}
& \int_\R\int_{\TT^d}\chi(x,\xi,t)\psi= \int_\R\int_{\TT^d}\overline{\chi}(\rho_0)\psi-\int_0^t\int_{\TT^d}\Phi'(\rho)\nabla\rho\cdot(\nabla\psi)(x,\rho)+\int_0^t\int_{\TT^d}\lambda(\rho)\psi(x,\rho)
\\ \nonumber & -\frac{1}{2}\int_0^t\int_{\TT^d}[\sigma'(\rho)]^2\nabla\rho\cdot(\nabla\psi)(x,\rho)-\frac{1}{2}\int_0^t\int_{\TT^d}\sigma(\rho)\sigma'(\rho)F_2\cdot(\nabla\psi)(x,\rho)-\int_0^t\int_{\TT^d}\psi(x,\rho)\nabla\cdot\nu(\rho)
\\ \nonumber & -\int_0^t\int_\R\int_{\TT^d}\partial_\xi\psi(x,\xi)\dd q+\frac{1}{2}\int_0^t\int_{\TT^d}F_3\sigma^2(\rho)(\partial_\xi\psi)(x,\rho)+\frac{1}{2}\int_0^t\int_{\TT^d}(\partial_\xi\psi)(x,\rho)\sigma(\rho)\sigma'(\rho)\nabla\rho\cdot F_2
\\ \nonumber & +\frac{1}{2}\int_0^t\int_{\TT^d}G_1\phi^2(\rho)(\partial_\xi\psi)(x,\rho)-\int_0^t\int_{\TT^d}\psi(x,\rho)\nabla\cdot\left(\sigma(\rho)\dd\xi^F\right)+\int_0^t\int_{\TT^d}\phi(\rho)\psi(x,\rho)\dd\xi^G.
\end{align}
\end{enumerate}
\end{definition}

\begin{assumption}\label{assume_5}  Let $\Phi$, $\sigma$, and $\nu$ satisfy Assumption~\ref{assume_1}.  Let $\phi,\lambda\in\C([0,\infty))$ satisfy the following two assumptions.

\begin{enumerate}[(i)]

\item  We have that $\lambda(0)=0$ and that $\lambda\in\C([0,\infty))$ is Lipschitz continuous on $[0,\infty)$.
\item We have $\phi(0)=0$ and there exists $c\in(0,\infty)$ such that
\begin{equation}\label{9_1700} \abs{\phi(\xi)-\phi(\xi')}\leq c(\abs{\xi-\xi'}^{\nicefrac{1}{2}}\mathbf{1}_{\{\abs{\xi-\xi'}\leq 1\}}+\abs{\xi-\xi'}\mathbf{1}_{\{\abs{\xi-\xi'}\geq 1\}})\;\;\textrm{for every}\;\;\xi,\xi'\in[0,\infty).\end{equation}
\end{enumerate}
\end{assumption}

\begin{remark}  The following two properties follow from the above assumptions.
\begin{enumerate}[(i)]
\item There exists $c\in(0,\infty)$ such that
\begin{equation}\label{9_16}\limsup_{\xi\rightarrow 0^+}\frac{\phi^2(\xi)}{\xi}\leq c.\end{equation}
\item There exists $c\in(0,\infty)$ such that
\begin{equation}\label{9_17}\abs{\phi(\xi)}\leq c(1+\xi)\;\;\textrm{and}\;\;\abs{\lambda(\xi)}\leq c\xi\;\;\textrm{for every}\;\;\xi\in[0,\infty).\end{equation}
\end{enumerate}
\end{remark}

\begin{prop}\label{prop_measure_1}  Let $\xi^F$, $\xi^G$, $\Phi$, $\sigma$, $\nu$, $\phi$, and $\lambda$ satisfy Assumptions~\ref{assumption_noise_1} and \ref{assume_5} and let $\rho_0\in L^1(\O;L^1(\TT^d))$ be nonnegative and $\F_0$-measurable.  Then, if $\rho$ is a stochastic kinetic solution of \eqref{9_0} in the sense of Definition~\ref{def_gen_sol} with initial data $\rho_0$ and with kinetic measure $q$, it follows almost surely that
\[\liminf_{\beta\rightarrow 0}\left(\beta^{-1}q(\TT^d\times[\nicefrac{\beta}{2},\beta]\times[0,T])\right)= 0.\]
\end{prop}

\begin{proof}  The proof is identical to Proposition~\ref{prop_measure} using \eqref{9_17} and the fact that in this case \eqref{defn_1} is the correct notion of mass preservation, which is the formal estimate obtained by testing \eqref{defn_6} with $\psi=1$ and taking the expectation.\end{proof}

\begin{lem}\label{lem_partition}  Let $(X,\mathcal{S})$ be a measurable space, let $K\in\N$, let $\{\mu_k\}_{k\in\{1,2,\ldots,K\}}$ be finite nonnegative measures on $(X,\mathcal{S})$, and for every $k\in\{1,2,\ldots,K\}$ let $\{B_{n,k}\subseteq X\}_{n\in\N}\subseteq \mathcal{S}$ be disjoint subsets.  Then,
\[\liminf_{n\rightarrow\infty}\left(n\sum_{k=1}^K\mu_k(B_{n,k})\right)=0.\]
\end{lem}

\begin{proof}   Proceeding by contradiction, suppose that there exists $\ve\in(0,1)$ such that
\[\liminf_{n\rightarrow\infty}\left(n\sum_{k=1}^K\mu_k(B_{n,k})\right)\geq\ve.\]
Then, there exists $N\in\N$ such that, for every $n\geq N$,
\[n\sum_{k=1}^K\mu_k(B_{n,k})\geq \frac{\ve}{2}.\]
For every $k\in\{1,2,\ldots,K\}$ let $\mathcal{I}_{N,k}\subseteq[N,N+1,\ldots)$ be defined by
\[\mathcal{I}_{N,k}=\left\{n\in[N,N+1,\ldots)\colon \mu_k(B_{n,k})\geq \frac{\ve}{2Kn}\right\}.\]
Since by definition $[N,N+1,\ldots)=\cup_{k=1}^K\mathcal{I}_{N,k}$, and since $\sum_{n=N}^\infty\frac{1}{n}=\infty$, there exists $k_0\in\{1,2,\ldots,K\}$ such that $\sum_{n\in \mathcal{I}_{N,k_0}}\frac{1}{n}=\infty$.  This contradicts the assumption that $\mu_{k_0}$ is a finite measure, since the assumption that the $\{B_{n,k_0}\}_{n\in\N}$ are disjoint and the definition of $\mathcal{I}_{n,k_0}$ imply that
\[\infty=\sum_{n\in\mathcal{I}_{N,k_0}}\frac{1}{n}\leq \frac{2K}{\ve}\sum_{n\in\mathcal{I}_{N,k_0}}\mu_{k_0}(B_{n,k_0})\leq \frac{2K}{\ve}\mu_{k_0}(X)<\infty,\]
which completes the proof.  \end{proof}

\begin{thm}\label{thm_gen_unique} Let $\xi^F$, $\xi^G$, $\Phi$, $\sigma$, $\nu$, $\phi$, and $\lambda$ satisfy Assumptions~\ref{assumption_noise_1} and \ref{assume_5}, let $\rho^1_0,\rho^2_0\in L^1(\O;L^1(\TT^d))$ be nonnegative and $\F_0$-measurable and let $\rho^1,\rho^2$ be stochastic kinetic solutions of \eqref{9_0} in the sense of Definition~\ref{def_gen_sol} with initial data $\rho_0^1,\rho_0^2$.  Then there exists $c\in(0,\infty)$ such that, for every $t\in[0,T]$,
\begin{equation}\label{10_0}\E\left[\norm{\rho^1(\cdot,t)-\rho^2(\cdot,t)}_{L^1(\TT^d))}\right]\leq c\exp(ct)\E\left[\norm{\rho^1_0-\rho^2_0}_{L^1(\TT^d)}\right].\end{equation}
Furthermore, there exists $c\in(0,\infty)$ such that
\begin{equation}\label{10_1} \E\left[\sup_{t\in[0,T]}\norm{\rho^1(\cdot,t)-\rho^2(\cdot,t)}_{L^1(\TT^d)}\right]\leq c\exp(cT)\left(\left(\E\left[\norm{\rho^1_0-\rho^2_0}_{L^1(\TT^d)}\right]\right)^\frac{1}{2}+\E\left[\norm{\rho^1_0-\rho^2_0}_{L^1(\TT^d)}\right]\right).\end{equation}
\end{thm}

\begin{proof}  In comparison to Theorem~\ref{thm_exist}, since the expectation eliminates the martingale terms and since the noise terms $\xi^F$ and $\xi^G$ are independent, to obtain \eqref{10_0} it remains only to estimate the cutoff term
\begin{equation}\label{9_10}\E\left[\frac{1}{2}\int_0^T\int_\R\int_{(\TT^d)^2}G_1\phi^2(\rho^1)(1-2\chi^{\ve,\d}_{s,2})\overline{\kappa}^{\ve,\d}_{s,1}\partial_\eta(\varphi_\beta(\eta)\zeta_M(\eta))\dx\dy\deta\ds\right],\end{equation}
and the analogous term obtained by swapping the roles of $\rho^1$ and $\rho^2$, the error term (after passing $\ve\rightarrow 0$ as in \eqref{3_0010000})
\begin{equation}\label{9_14}\E\left[\frac{1}{2}\int_0^T\int_\R\int_{\TT^d}G_1\left(\phi(\rho^1)-\phi(\rho^2)\right)^2\overline{\kappa}^{\d}_{s,1}\overline{\kappa}^{\d}_{s,2}\varphi_\beta(\eta)\zeta_M(\eta)\dy\deta\ds\right],\end{equation}
and the error term
\begin{equation}\label{9_70}
\E\left[\int_0^T\int_\R\int_{(\TT^d)^2}\lambda(\rho^1)\overline{\kappa}^{\ve,\d}_{s,1}(1-2\chi^{\ve,\d}_{s,2})\varphi_\beta(\eta)\zeta_M(\eta)+ \int_0^T\int_\R\int_{(\TT^d)^2}\lambda(\rho^2)\overline{\kappa}^{\ve,\d}_{s,2}(1-2\chi^{\ve,\d}_{s,1})\varphi_\beta(\eta)\zeta_M(\eta)\right].
\end{equation}
Term \eqref{9_10} is treated analogously to \eqref{3_12}, where \eqref{9_16} and \eqref{9_17} are used to treat the limit $\beta\rightarrow 0$, and where the $L^1$-integrability of the $\rho^i$, Lemma~\ref{lem_partition} applied to the partitions $\O\times\{M<\rho^i<M+1\}\times[0,T]$ for $M\in\N$ and the measures $(1+\rho^i)\dx\ds\dd\P$, and assumption \eqref{9_17} are used to treat the limit $M\rightarrow\infty$.  The error term \eqref{9_14} is treated identically to \eqref{3_11} using the local $\nicefrac{1}{2}$-H\"older continuity of $\phi$ on $(0,\infty)$.  After passing to the limit $\ve\rightarrow 0$ exactly as in \eqref{3_23}, it follows that \eqref{9_70} becomes
\begin{equation}\label{9_700}
\E\left[\int_0^T\int_\R\int_{\TT^d}\lambda(\rho^1)\overline{\kappa}^\d_{s,1}(1-2\chi^\d_{s,2})\varphi_\beta(\eta)\zeta_M(\eta)+ \int_0^T\int_\R\int_{\TT^d}\lambda(\rho^2)\overline{\kappa}^\d_{s,2}(1-2\chi^\d_{s,1})\varphi_\beta(\eta)\zeta_M(\eta)\right].
\end{equation}
After passing to the limit $\d\rightarrow 0$ using \eqref{3_25} and the fact that $\lambda(0)=0$, equation \eqref{9_700} becomes
\begin{align}\label{9_71}
& \E\left[\int_0^T\int_{\TT^d}\lambda(\rho^1)(1-\mathbf{1}_{\{\rho^1=\rho^2\}}-2\mathbf{1}_{\{\rho^1<\rho^2\}})\varphi_\beta(\rho^1)\zeta_M(\rho^1)\right.
\\ \nonumber& \quad \quad \left.+ \int_0^T\int_{\TT^d}\lambda(\rho^2)(1-\mathbf{1}_{\{\rho^1=\rho^2\}}-2\mathbf{1}_{\{\rho^2<\rho^1\}})\varphi_\beta(\rho^2)\zeta_M(\rho^2)\right].
\end{align}
After passing to the limit $\beta\rightarrow 0$ using the dominated convergence theorem and $\lambda(0)=0$ and to the limit $M\rightarrow\infty$ using the dominated convergence theorem, Assumption~\ref{assume_5}, and the $L^1$-integrability of the $\rho^1$, it follows from the equality $\mathbf{1}_{\{\rho^1<\rho^2\}}=1-\mathbf{1}_{\{\rho^1=\rho^2\}}-\mathbf{1}_{\{\rho^2<\rho^1\}}$ that \eqref{9_71} becomes
\begin{equation}\label{9_72} \E\left[\int_0^T\int_{\TT^d}(\lambda(\rho^1)-\lambda(\rho^2))(1-\mathbf{1}_{\{\rho^1=\rho^2\}}-2\mathbf{1}_{\{\rho^1<\rho^2\}})\right]=\E\left[\int_0^T\int_{\TT^d}(\lambda(\rho^1)-\lambda(\rho^2))\sgn(\rho^1-\rho^2)\right].\end{equation}
It then follows from \eqref{9_72} and the Lipschitz continuity of $\lambda$ that there exists $c\in(0,\infty)$ such that
\begin{equation}\label{9_72000}\E\left[\int_0^T\int_{\TT^d}(\lambda(\rho^1)-\lambda(\rho^2))(1-\mathbf{1}_{\{\rho^1=\rho^2\}}-2\mathbf{1}_{\{\rho^1<\rho^2\}})\right]\leq c\E\left[\int_0^T\int_{\TT^d}\abs{\rho^1-\rho^2}\right].\end{equation}
It then follows from Theorem~\ref{thm_unique}, \eqref{9_10}, \eqref{9_14}, and \eqref{9_72000} that there exists $c\in(0,\infty)$ such that, almost surely for every $t\in[0,T]$,
\[\E\left[\norm{\rho^1(\cdot,t)-\rho^2(\cdot,t)}_{L^1(\TT^d)}\right]\leq \E\left[\norm{\rho^1_0-\rho^2_0}_{L^1(\TT^d)}\right]+c\E\left[\int_0^t\int_{\TT^d}\abs{\rho^1-\rho^2}\right].\]
The proof of \eqref{10_0} follows from an application of Gr\"onwall's inequality.  It remains to prove estimate \eqref{10_1}, for which it is only necessary to consider the additional martingale term.

That is, in analogy with \eqref{3_23} in Theorem~\ref{thm_unique}, it remains only to estimate the term
\begin{equation}\label{10_6} \sup_{t\in[0,T]}\abs{\int_0^t\int_{\TT^d}\left((1-\mathbf{1}_{\{\rho^1=\rho^2\}}-2\mathbf{1}_{\{\rho^1<\rho^2\}})\phi(\rho^1)+(1-\mathbf{1}_{\{\rho^1=\rho^2\}}-2\mathbf{1}_{\{\rho^2<\rho^1\}})\phi(\rho^2)\right)\dd\xi^G}.\end{equation}
The identity $\mathbf{1}_{\{\rho^1<\rho^2\}}=1-\mathbf{1}_{\{\rho^1=\rho^2\}}-\mathbf{1}_{\{\rho^2<\rho^1\}}$ proves that \eqref{10_6} is equal to
\begin{equation}\label{10_7}
\sup_{t\in[0,T]}\abs{\int_0^t\int_{\TT^d}\sgn(\rho^1-\rho^2)(\phi(\rho^1)-\phi(\rho^2))\dd\xi^G}.
\end{equation}
After taking the expectation of \eqref{10_7}, the Burkholder-Davis-Gundy inequality (see, for example, \cite[Chapter~4, Theorem~4.1]{RevYor1999}) and H\"older's inequality prove that there exists $c\in(0,\infty)$ such that
\begin{equation}\label{10_8}
 \E\sup_{t\in[0,T]}\abs{\int_0^t\int_{\TT^d}\sgn(\rho^1-\rho^2)(\phi(\rho^1)-\phi(\rho^2))\dd\xi^G}  \leq c\left(\E\left[\int_0^T\left(\int_{\TT^d}\abs{\phi(\rho^1)-\phi(\rho^2)}\dx\right)^2\ds\right]\right)^\frac{1}{2}.
\end{equation}
It then follows from Assumption~\ref{assume_5} and specifically \eqref{9_1700} and H\"older's inequality that there exists $c\in(0,\infty)$ such that
\begin{align}\label{10_9}
& \left(\E\left[\int_0^T\left(\int_{\TT^d}\abs{\phi(\rho^1)-\phi(\rho^2)}\dx\right)^2\ds\right]\right)^\frac{1}{2}
\\ \nonumber & \leq c\left(\E\left[\int_0^T\int_{\TT^d}\abs{\rho^1-\rho^2}\right]\right)^\frac{1}{2}+c\E\left[\sup_{t\in[0,T]}\norm{\rho^1(\cdot,)-\rho^2(\cdot,t)}^\frac{1}{2}_{L^1(\TT^d)}\left(\int_0^T\int_{\TT^d}\abs{\rho^1-\rho^2}\right)^\frac{1}{2}\right].
\end{align}
H\"older's inequality, Young's inequality, and \eqref{10_9} prove that there exists $c\in(0,\infty)$ such that
\begin{align}\label{10_10}
& \left(\E\left[\int_0^T\left(\int_{\TT^d}\abs{\phi(\rho^1)-\phi(\rho^2)}\dx\right)^2\ds\right]\right)^\frac{1}{2}\leq c\left(\E\left[\int_0^T\int_{\TT^d}\abs{\rho^1-\rho^2}\right]\right)^\frac{1}{2}
\\ \nonumber & +c\E\left[\int_0^T\int_{\TT^d}\abs{\rho^1-\rho^2}\right]+\frac{1}{2}\E\left[\sup_{t\in[0,T]}\norm{\rho^1(\cdot,t)-\rho^2(\cdot,t)}_{L^1(\TT^d)}\right],
\end{align}
Finally, it follows from \eqref{10_0} and \eqref{10_10} that there exists $c\in(0,\infty)$ such that
\begin{align}\label{10_11}
& \E\left[\int_0^T\left(\int_{\TT^d}(\phi(\rho^1)-\phi(\rho^2)\right)^2\right]^\frac{1}{2} \leq c\exp(cT)\left(\left(\E\left[\norm{\rho^1_0-\rho^2_0}_{L^1(\TT^d)}\right]\right)^\frac{1}{2}+\E\left[\norm{\rho^1_0-\rho^2_0}_{L^1(\TT^d)}\right]\right)
\\ \nonumber &+\frac{1}{2}\E\left[\sup_{t\in[0,T]}\norm{\rho^1(\cdot,t)-\rho^2(\cdot,t)}_{L^1(\TT^d)}\right],
\end{align}
where the final term on the righthand side of \eqref{10_11} is absorbed into the lefthand side of the estimate.  This completes the proof of estimate \eqref{10_1}, and therefore the proof.  \end{proof}

\textbf{Existence of solutions to \eqref{9_0}}.  We will now construct a stochastic kinetic solution to \eqref{9_0}.  We introduce the assumptions on $\Phi$, $\sigma$, $\nu$, $\phi$, and $\lambda$ in Assumption~\ref{assume_4}.  We define a solution of the regularized version of \eqref{9_0} with $\a\in(0,\infty)$ and for smooth and bounded $\sigma$ in Definition~\ref{gen_approx}.  In Propositions~\ref{prop_gen_est} and \ref{gen_approx_time} we prove estimates analogous to Propositions~\ref{prop_approx_est} and \ref{prop_approx_time}.  Finally, in Theorem~\ref{thm_gen_exist} we construct a probabilistically strong solution to \eqref{9_0} in the sense of Definition~\ref{def_gen_sol}.  In Corollaries~\ref{cor_exist_3} and \ref{cor_exist_2} we extend these results to initial data with finite entropy and $L^1$-initial data respectively.

\begin{assumption}\label{assume_4}  Let $\Phi$, $\sigma$, and $\nu$ satisfy Assumption~\ref{assume} for some $p\in[2,\infty)$.  Let $\phi,\lambda\in\C([0,\infty))$ satisfy the following assumption.
\begin{enumerate}[(i)]
\item  There exists $c\in(0,\infty)$ such that
\[\abs{\phi(\xi)}\leq c(1+\xi)\;\;\textrm{and}\;\;\abs{\lambda(\xi)}\leq c\xi\;\;\textrm{for every}\;\;\xi\in[0,\infty).\]
\end{enumerate}
\end{assumption}

\begin{definition}\label{gen_approx} Let $\xi^F$, $\xi^G$, $\Phi$, $\sigma$, $\nu$, $\phi$, and $\lambda$ satisfy Assumptions~\ref{assume_3}, \ref{assumption_noise_1}, and \ref{assume_4} for some $p\in[2,\infty)$, let $\a\in(0,1)$, and let $\rho_0\in L^{m+p-1}(\O;L^1(\TT^d))\cap L^p(\O;L^p(\TT^d))$ be nonnegative and $\F_0$-measurable.  A solution of \eqref{9_0} with initial data $\rho_0$ is a continuous $L^p(\TT^d)$-valued, nonnegative, $\F_t$-predictable process $\rho$ such that almost surely $\rho$ and $\Theta_{\Phi,2}(\rho)$ are in $L^2([0,T];H^1(\TT^d))$ and such that for every $\psi\in\C^\infty(\TT^d)$, almost surely for every $t\in[0,T]$,
\begin{align*}
& \int_{\TT^d}\rho(x,t)\psi(x)\dx = \int_{\TT^d}\rho_0\psi\dx -\int_0^t\int_{\TT^d}\Phi'(\rho)\nabla\rho\cdot\nabla\psi - \a\int_0^t\int_{\TT^d}\nabla\rho\cdot\nabla\psi
\\ & \quad +\int_0^t\int_{\TT^d}\sigma(\rho)\nabla\psi\cdot\dd\xi^F+\int_0^t\int_{\TT^d}\nabla\psi\cdot \nu(\rho)+\int_0^t\int_{\TT^d}\psi\phi(\rho)\dd\xi^G+\int_0^t\int_{\TT^d}\lambda(\rho)\psi
\\ & \quad -\frac{1}{2}\int_0^t\int_{\TT^d}F_1[\sigma'(\rho)]^2\nabla\rho\cdot\nabla\psi-\frac{1}{2}\int_0^t\int_{\TT^d}\sigma(\rho)\sigma'(\rho)F_2\cdot\nabla\psi.
\end{align*}  \end{definition}

\begin{prop}\label{prop_gen_est}  Let $\xi^F$, $\xi^G$, $\Phi$, $\sigma$, $\nu$, $\phi$, and $\lambda$ satisfy Assumptions~\ref{assume_3}, \ref{assumption_noise_1}, and \ref{assume_4} for some $p\in[2,\infty)$, let $T\in[1,\infty)$, let $\a\in(0,1)$, let $\rho_0\in L^{m+p-1}(\O;L^1(\TT^d))\cap L^p(\O;L^p(\TT^d))$ be nonnegative and $\F_0$-measurable, and let $\rho$ be a solution in the sense of Definition~\ref{gen_approx}.  Then there exists $c\in(0,\infty)$ such that
\begin{equation}\label{9_4} \E\left[\sup_{t\in[0,T]}\norm{\rho(\cdot,t)}_{L^1(\TT^d)}\right]\leq c\exp(cT)\E\left[\norm{\rho_0}_{L^1(\TT^d)}\right].\end{equation}
For $\Theta_{\Phi,p}$ defined in Lemma~\ref{lem_aux} there exists $c\in(0,\infty)$ depending on $p$ but independent of $\a$ and $T$ such that
\begin{align}\label{9_7}
& \sup_{t\in[0,T]}\E\left[\int_{\TT^d}\rho^p(\cdot,t)\right]+\E\left[\int_0^T\int_{\TT^d}\abs{\nabla\Theta_{\Phi,p}(\rho)}^2\right]+\a\E\left[\int_0^T\int_{\TT^d}\abs{\rho}^{p-2}\abs{\nabla\rho}^2\dx\ds\right]
\\ \nonumber &  \leq c\exp(cT)\left(1+\E\left[\norm{\rho_0}^{m+p-1}_{L^1(\TT^d)}+\int_{\TT^d}\rho_0^p\right]\right).
\end{align}
\end{prop}

\begin{proof}  The proof of \eqref{9_4} follows by taking $\psi=1$ in Definition~\ref{gen_approx} and applying the Burkholder-Davis-Gundy inequality (see, for example, \cite[Chapter~4, Theorem~4.1]{RevYor1999}) in analogy with \eqref{10_8}, \eqref{10_9}, \eqref{10_10}, and \eqref{10_11} above using Assumption~\ref{assume_4} and Gr\"onwall's inequality.  To prove \eqref{9_7}, in comparison to Proposition~\ref{prop_approx_est} and \eqref{4_00}, since the expectation eliminates the martingale term it remains only to use Assumption~\ref{assume_4} to estimate the term, for some $c\in(0,\infty)$,
\[\abs{\int_0^t\int_{\TT^d}\lambda(\rho)\rho^{p-1}}+\int_0^t\int_{\TT^d}G_1\phi^2(\rho)\abs{\rho}^{p-2}\leq c\left(t+\int_0^t\int_{\TT^d}\rho^p\right).\]
Gr\"onwall's inequality completes the proof.\end{proof}

\begin{prop}\label{gen_approx_time}  Let $\xi^F$, $\xi^G$, $\Phi$, $\sigma$, $\nu$, $\phi$, and $\lambda$ satisfy Assumptions~\ref{assume_3}, \ref{assumption_noise_1}, and \ref{assume_4} for some $p\in[2,\infty)$, for every $\d\in(0,1)$ let $\Psi_\d$ be as in Definition~\ref{def_cutoff}, let $T\in[1,\infty)$, let $\a\in(0,1)$, let $\rho_0\in L^{m+p-1}(\O;L^1(\TT^d))\cap L^p(\O;L^p(\TT^d))$ be nonnegative and $\F_0$-measurable, and let $\rho$ be a solution in the sense of Definition~\ref{gen_approx}.  Then, for every $\beta\in(0,\nicefrac{1}{2})$ there exists $c\in(0,\infty)$ depending on $\d$ and $\beta$ but independent of $\a$ and $T$ such that, for every $s>\frac{d}{2}+1$,
\[\E\left[\norm{\Psi_\d(\rho)}_{W^{\beta,1}([0,T];H^{-s}(\TT^d))}\right]\leq c\exp(cT)\left(1+\E\left[\norm{\rho_0}^{m+p-1}_{L^1(\TT^d)}+\int_{\TT^d}\rho_0^p\right]\right).\]
\end{prop}

\begin{proof}  In comparison to Proposition~\ref{prop_approx_time}, it remains to estimate the $W^{1,1}([0,T];H^{-s}(\TT^d))$-norm of
\[\frac{1}{2}\int_0^tG_1\Psi_\d''(\rho)\phi^2(\rho)+\int_0^t\Psi'_\d(\rho)\lambda(\rho),\]
and, for every $\b\in(0,\nicefrac{1}{2})$, the $W^{\beta,2}([0,T];H^{-s}(\TT^d))$-norm of
\[\int_0^t\Psi_\d'(\rho)\phi(\rho)\dd\xi^G.\]
Since $s>\frac{d}{2}+1$, similarly to \eqref{4_400} and \eqref{4_402}, it follows from Assumption~\ref{assume_4}, the boundedness of $G_1$, the fact that $\Psi''_\d$ is supported on $[\nicefrac{\d}{2},\d]$, and \eqref{9_4} that there exists $c\in(0,\infty)$ such that
\begin{align*}
& \E\left[\norm{\frac{1}{2}\int_0^\cdot\Psi_\d''(\rho)\phi^2(\rho)+\int_0^\cdot\Psi'_\d(\rho)\lambda(\rho)}_{W^{1,1}([0,T];H^{-s}(\TT^d)}\right]\leq c\exp(cT)\left(1+\E\left[\norm{\rho_0}_{L^1(\TT^d)}\right]\right).
\end{align*}
Since $s>\nicefrac{d}{2}+1$, similarly to \eqref{4_403} and \eqref{4_404}, it follows from Assumption~\ref{assume_4} and \eqref{9_4} that there exists $c\in(0,\infty)$ such that
\[\E\left[\norm{\int_0^\cdot\Psi_\d'(\rho)\phi(\rho)\dd\xi^G}^2_{W^{\beta,2}([0,T];H^{-s}(\TT^d))}\right]\leq c\E\left[\int_0^t\int_{\TT^d}\phi^2(\rho)\right]\leq c\left(T+\E\left[\int_0^T\int_{\TT^d}\rho^2\right]\right).\]
The claim now follows from the same argument as in Proposition~\ref{prop_approx_time}, using $p\in[2,\infty)$, H\"older's inequality, Young's inequality, and the estimates of Proposition~\ref{prop_gen_est}.  This completes the proof.
\end{proof}

\begin{thm}\label{thm_gen_exist}  Let $\xi^F$, $\xi^G$, $\Phi$, $\sigma$, $\nu$, $\phi$, and $\lambda$ satisfy Assumptions \ref{assumption_noise_1} and \ref{assume_4} for some $p\in[2,\infty)$ and let $\rho_0\in L^{m+p-1}(\O;L^1(\TT^d))\cap L^p(\O;L^p(\TT^d))$ be nonnegative and $\F_0$-measurable.  Then there exists a stochastic kinetic solution of \eqref{9_0} in the sense of Definition~\ref{def_gen_sol}.  Furthermore, the solution satisfies the estimates of Proposition~\ref{prop_gen_est}.\end{thm}

\begin{proof}  The proof is identical to the proof of Theorem~\ref{thm_exist}, using Assumption~\ref{assume_4} and the estimates of Propositions~\ref{prop_gen_est} and \ref{gen_approx_time}.\end{proof}

\begin{assumption}\label{assumption_11}  Let $\Phi,\sigma\in\C([0,\infty))$ and $\nu\in\C([0,\infty);\R^d)$ satisfy Assumption~\ref{assumption_10}.  Assume that $\phi,\lambda\in\C([0,\infty))$ satisfy the following assumptions.
\begin{enumerate}[(i)]
\item There exists $c\in(0,\infty)$ such that
\[\phi^2(\xi)\left(1+\log^2(\Phi(\xi))+\frac{\Phi'(\xi)}{\Phi(\xi)}\right)\leq c(1+\xi+\Phi(\xi))\;\;\textrm{for every}\;\;\xi\in[0,\infty).\]
\item There exists $c\in(0,\infty)$ such that
\[\abs{\lambda(\xi)\log(\Phi(\xi))}\leq c(1+\xi+\Phi(\xi)).\]
\end{enumerate}
\end{assumption}

\begin{prop}\label{final_ent}  Let $\xi^F$, $\xi^G$, $\Phi$, $\sigma$, $\nu$, $\phi$, and $\lambda$ satisfy Assumptions~\ref{assume_3}, \ref{assumption_noise_1}, \ref{assume_4}, and \ref{assumption_11} for some $p\in[2,\infty)$, let $\a\in(0,1)$, let $T\in[1,\infty)$, let $\rho_0\in L^m(\O;L^1(\TT^d))\cap L^1(\O;\Ent(\TT^d))$ be $\F_0$-measurable, and let $\rho$ be a solution of \eqref{9_0} in the sense of Definition~\ref{gen_approx}.  Then there exists $c\in(0,\infty)$ independent of $\a$ and $T$ such that
\begin{align}\label{9_000007}
&\E\left[ \sup_{t\in[0,T]}\int_{\TT^d}\Psi_\Phi(\rho(x,t))\right] +\E\left[\int_0^T\int_{\TT^d}\abs{\nabla\Phi^\frac{1}{2}(\rho)}^2\right]+\a\E\left[\int_0^T\int_{\TT^d}\frac{\Phi'(\rho)}{\Phi(\rho)}\abs{\nabla\rho}^2\right]
 \\ \nonumber & \leq c\exp(cT)\left(1+\E\left[\norm{\rho_0}^m_{L^1(\TT^d)}+\int_{\TT^d}\Psi_\Phi(\rho_0)\right]\right).
\end{align}
\end{prop}

\begin{proof}To prove \eqref{9_000007}, in comparison to Proposition~\ref{prop_approx_est} and \eqref{4_025}, it is necessary to estimate the term
\[\E\left[\frac{1}{2}\int_0^T\int_{\TT^d}G_1\phi^2(\rho)\frac{\Phi'(\rho)}{\Phi(\rho)}+\int_0^T\int_{\TT^d}\abs{\lambda(\rho)\log(\Phi(\rho))}\right]+\E\left[\sup_{t\in[0,T]}\abs{\int_0^t\int_{\TT^d}\phi(\rho)\log(\Phi(\rho))\dd\xi^G}\right].\]
It follows from the Assumption~\ref{assumption_11}, $T\in[1,\infty)$,  \eqref{9_4}, the Burkholder-Davis-Gundy inequality (see, for example, \cite[Chapter~4, Theorem~4.1]{RevYor1999}), and Young's inequality that there exists $c\in(0,\infty)$ such that these terms are bounded by
\[c\exp(cT)\left(1+\E\left[\norm{\rho_0}_{L^1(\TT^d)}+\int_0^T\int_{\TT^d}\Phi(\rho)\right]\right).\]
The claim now follows using the interpolation estimate \eqref{50_2} with $\Psi=\Phi^\frac{1}{2}$ and \eqref{9_4}, which completes the proof.\end{proof}

\begin{cor}\label{cor_exist_3}  Let $\xi^F$, $\xi^G$, $\Phi$, $\sigma$, $\nu$, $\phi$, and $\lambda$ satisfy Assumptions \ref{assumption_noise_1}, \ref{assume_4}, and \ref{assumption_11} for some $p\in[2,\infty)$ and let $\rho_0\in L^m(\O;L^1(\TT^d))\cap L^1(\O;\Ent(\TT^d))$ be $\F_0$-measurable.  Then there exists a stochastic kinetic solution of \eqref{9_0} in the sense of Definition~\ref{def_gen_sol}.  Furthermore, the solution satisfies the estimates of Proposition~\ref{final_ent}.\end{cor}

\begin{proof}The proof is identical to Corollary~\ref{cor_exist} using Theorem~\ref{thm_gen_unique}, Theorem~\ref{thm_gen_exist}, and Proposition~\ref{final_ent}.\end{proof}

\begin{cor}\label{cor_exist_2}  Let $\xi^F$, $\xi^G$, $\Phi$, $\sigma$, $\nu$, $\phi$, and $\lambda$ satisfy Assumptions \ref{assumption_noise_1} and \ref{assume_4} for some $p\in[2,\infty)$, let $\rho_0\in L^1(\O;L^1(\TT^d))$ be nonnegative and $\F_0$-measurable, and assume that there exists $c\in(0,\infty)$ such that
\[\sigma^2(\xi)+\abs{\nu(\xi)}+\phi^2(\xi)\leq c(1+\xi)\;\;\textrm{for every}\;\;\xi\in[0,\infty).\]
Then there exists a stochastic kinetic solution of \eqref{9_0} in the sense of Definition~\ref{def_gen_sol}.  \end{cor}

\begin{proof} The proof is identical to Corollary~\ref{cor_exist_1} using Theorem~\ref{thm_gen_unique} and Theorem~\ref{thm_gen_exist}.\end{proof}

\section*{Acknowledgements}

The first author acknowledges financial support from the EPSRC through the EPSRC Early Career Fellowship EP/V027824/1.  This work was funded by the Deutsche Forschungsgemeinschaft (DFG, German Research Foundation) -- SFB 1283/2 2021 -- 317210226.

\bibliography{Untitled}
\bibliographystyle{plain}

\end{document}